\documentclass{amsart}
\usepackage[utf8]{inputenc}
\usepackage{amsmath,amsfonts,amsthm,amssymb,verbatim,etoolbox,color}
\usepackage{mathtools, hyperref}
\usepackage{tabularx}
\usepackage{relsize}
\usepackage{microtype}

\newcommand{\bbf}{\mathbb{F}}
\newcommand{\bbn}{\mathbb{N}}
\newcommand{\bbz}{\mathbb{Z}}
\newcommand{\bbr}{\mathbb{R}}

\newcommand{\bbc}{\mathbb{C}}

\newcommand{\ls}{\lesssim}
\newcommand{\gs}{\gtrsim}

\newcommand{\rk}{\mathrm{rk}}

\newcommand{\bbe}{\mathbb{E}}

\newcommand{\abs}[1]{\left\lvert #1\right\rvert}
\newcommand{\Abs}[1]{\lvert #1\rvert}
\newcommand{\brac}[1]{\left( #1\right)}
\newcommand{\inn}[1]{\left\langle #1 \right\rangle}
\newcommand{\Inn}[1]{\langle #1 \rangle}
\newcommand{\norm}[1]{\left\lVert #1\right\rVert}
\newcommand{\Norm}[1]{\lVert #1\rVert}
\newcommand{\ind}[1]{1_{#1}}
\newcommand{\bal}[2]{\mu_{#1 / #2 }}
\renewcommand{\geq}{\geqslant}
\renewcommand{\leq}{\leqslant}
\renewcommand{\epsilon}{\varepsilon}

\DeclareFontFamily{U}{mathx}{\hyphenchar\font45}
\DeclareFontShape{U}{mathx}{m}{n}{
      <5> <6> <7> <8> <9> <10>
      <10.95> <12> <14.4> <17.28> <20.74> <24.88>
      mathx10
      }{}
\DeclareSymbolFont{mathx}{U}{mathx}{m}{n}
\DeclareFontSubstitution{U}{mathx}{m}{n}
\DeclareMathAccent{\widecheck}{0}{mathx}{"71}
\DeclareMathAccent{\wideparen}{0}{mathx}{"75}

\DeclareMathOperator{\supp}{supp}

\newtheorem{theorem}[subsection]{Theorem}
\newtheorem{lemma}[subsection]{Lemma}
\newtheorem{corollary}[subsection]{Corollary}
\newtheorem{conjecture}[subsection]{Conjecture}
\newtheorem{proposition}[subsection]{Proposition}
\theoremstyle{definition}
\newtheorem{definition}[subsection]{Definition}

\newtheorem{innercustomthm}{Proposition}
\newenvironment{customprop}[1]
  {\renewcommand\theinnercustomthm{#1}\innercustomthm}
  {\endinnercustomthm}
 
\makeatletter
\newcommand{\listintertext}{\@ifstar\listintertext@\listintertext@@}
\newcommand{\listintertext@}[1]{
  \hspace*{-\@totalleftmargin}#1}
\newcommand{\listintertext@@}[1]{
  \hspace{-\leftmargin}#1}
\makeatother

\begin{document}

\title[Breaking the logarithmic barrier]{Breaking the logarithmic barrier in Roth's theorem on arithmetic progressions}
\author{Thomas F. Bloom}
\address{Mathematical Institute\\Woodstock Road\\ Oxford OX2 6GG, UK}
\email{bloom@maths.ox.ac.uk}
\author{Olof Sisask}
\address{Department of Mathematics\\Stockholm University, Sweden}
\email{olof.sisask@math.su.se}
\date{}

\begin{abstract}
We show that if $A\subset \{1,\ldots,N\}$ contains no non-trivial three-term arithmetic progressions then $\abs{A} \ll N/(\log N)^{1+c}$ for some absolute constant $c>0$. In particular, this proves the first non-trivial case of a conjecture of Erd\H{o}s on arithmetic progressions.
\end{abstract}

\maketitle 

\setcounter{tocdepth}{1}

\tableofcontents

\section{Introduction}
In this paper we improve the upper bound for Roth's theorem on three-term arithmetic progressions in the integers.
 
\begin{theorem}\label{mainthm}
Let $N\geq 2$ and $A\subset \{1,\ldots,N\}$ be a set with no non-trivial three-term arithmetic progressions, i.e. solutions to $x+y=2z$ with $x\neq y$. Then
\[\abs{A}\ll \frac{N}{(\log N)^{1+c}},\]
where $c>0$ is an absolute constant.
\end{theorem}

The best bound previously available was $N/(\log N)^{1-o(1)}$, of which there are now four different proofs in the literature (with slightly different behaviour in the $o(1)$ term). The first is due to Sanders \cite{Sa:2011}, the second due to the first author \cite{Bl:2016}, and a third due to both authors \cite{BlSi:2019}. Recently Schoen \cite{Sc:2020} has combined the result of \cite{Bl:2016} with some of the ideas of Bateman and Katz \cite{BaKa:2012}, which we discuss later, to obtain a further slight improvement in the $o(1)$ term. 

The constant $c$ is in principle effectively computable, but computing this would be an arduous task, and any $c$ produced by the method in this paper would certainly be very small.\footnote{A back of the envelope calculation suggests that $c\approx 2^{-2^{2^{1000}}}$ should be achievable, for example.} The main value of the bound in Theorem~\ref{mainthm} is that it is $o(N/\log N)$, and so pushes past the previous density barrier of $N/\log N$ that all other approaches have encountered. We defer further discussion of what the correct bound in this problem might be to Section~\ref{section:spec}.

A well-known conjecture of Erd\H{o}s states that if $A$ is a set of positive integers such that $\sum_{n \in A} 1/n$ diverges then $A$ contains arbitrarily long arithmetic progressions. As a corollary of Theorem~\ref{mainthm} we obtain the first non-trivial case of this conjecture. 

\begin{corollary}
If $A\subset \bbn$ is such that $\sum_{n\in A}\tfrac{1}{n}=\infty$ then $A$ contains infinitely many non-trivial three-term arithmetic progressions.
\end{corollary}
\begin{proof}
Suppose that $A\subset \bbn$ contains only finitely many three-term arithmetic progressions. Then, for all $N\geq 2$,
\[F(N) \coloneqq\abs{A\cap [1,N]}\ll \frac{N}{(\log N)^{1+c}},\]
where $c>0$ is the constant from Theorem~\ref{mainthm}. By partial summation,
\[\sum_{\substack{n\leq N\\ n\in A}}\frac{1}{n}=\frac{F(N)}{N}+\int_1^N \frac{F(t)}{t^2}\mathrm{d}t\ll 1+\int_2^N\frac{1}{t(\log t)^{1+c}}\mathrm{d} t\ll 1.\]
Taking $N\to \infty$ shows that $\sum_{n\in A}\frac{1}{n}$ converges.
\end{proof}

Moving past the density barrier of $1/\log N$ also allows us to find three-term arithmetic progressions in the primes using nothing stronger than Chebyshev's estimate that the number of primes in $\{1,\ldots,N\}$ is $\gg N/\log N$. For example, we immediately obtain a strong form of a theorem originally due to Green \cite{Gr:2005}: every subset of the primes of positive relative density contains infinitely many non-trivial three-term arithmetic progressions. Indeed, using nothing more than Chebyshev's estimate, we have the following quantitative result. 

\begin{corollary}
Let $\mathbb{P}$ denote the set of primes and suppose $A\subset \mathbb{P} \cap \{1,\ldots,N\}$. If $A$ has no non-trivial three-term arithmetic progressions then $A$ has relative density
\[\frac{\abs{A}}{\abs{\mathbb{P}\cap \{1,\ldots,N\}}}\ll \frac{1}{(\log N)^c}\]
for some absolute constant $c>0$.
\end{corollary}
\noindent The best bound previously known here, due to Naslund \cite{Na:2015}, was $(\log\log N)^{-1+o(1)}$.

Bateman and Katz \cite{BaKa:2012} have proved a bound analogous to that in Theorem~\ref{mainthm} for the cap set problem -- namely, they showed that if $A\subset \bbf_3^n$ contains no non-trivial three-term arithmetic progressions then $\abs{A} \ll 3^n/n^{1+c}$ (where $c>0$ is also a miniscule absolute constant, but possibly different from the $c$ in Theorem~\ref{mainthm}). Our proof builds upon many of their ideas, and in particular uses a detailed analysis of the additive structure of spectra (sets of large Fourier coefficients). We also introduce the new technique of `spectral boosting' which allows us to convert this structural information about spectra into structural information about $A$. We will also make crucial use of almost-periodicity, a purely physical technique introduced by Croot and the second author \cite{CrSi:2010}.

A new polynomial method introduced by Croot, Lev, and Pach \cite{CrLePa:2016} has since superseded the result of Bateman and Katz. Indeed, when studying three-term arithmetic progressions in $\mathbb{F}_3^n$ the polynomial method is both much simpler and gives far superior quantitative bounds (namely an upper bound of $\abs{A}\leq c^n$ for some $c<3$), as shown by Ellenberg and Gijswijt \cite{ElGi:2016}. Such algebraic methods have not been successfully adapted to the integers, for which Fourier analytic methods (as used in this paper) remain the most effective.

In Section 2 we will introduce the notation and basic conventions that we will hold to in the rest of the paper. In Section 3 we give a sketch of the proof in the model setting of $\mathbb{F}_3^n$, to help the reader understand the main ideas and overall strategy of the proof. The rest of the paper is taken up with the proof of Theorem~\ref{mainthm}, until Section~\ref{section:spec}, which contains some speculation on the correct bounds for Roth's theorem, and how further progress might be achieved.

\subsection*{Acknowledgements}
We thank C\'{e}dric Pilatte and an anonymous referee for many helpful suggestions and corrections to an earlier draft of this paper.

This proof uses, and would not be possible without, almost all of the ideas used in previous improvements of the quantitative bounds for Roth's theorem. While we have given complete proofs of almost all of the subsidiary results used in this paper, both to keep it as self-contained as possible and because often the precise versions we need have not appeared in the literature before, we have endeavoured to indicate the origin of the relevant ideas. We would like to acknowledge in general the huge debt we owe both to the work of Bateman and Katz on the cap set problem, and to Jean Bourgain and Tom Sanders, who created much of the modern theory of Bohr sets and quantitative additive combinatorics in the integers. 

The first author was supported by both the Heilbronn Institute for Mathematical Research and a
postdoctoral grant funded by the Royal Society. The second author was supported by the Swedish Research Council grant 2013-4896. Part of this work was carried out while the authors were visiting the Simons Institute for the Pseudorandomness 2017 programme. We thank all these institutions for their generous support.

\section{Basic notation}
In this section we introduce our notation, some of which is standard, and some of which is unusual but chosen for a more streamlined presentation. We include an index summarising the non-standard notation and definitions used in this paper as Table~\ref{tab-not}.
\renewcommand{\arraystretch}{1.5}
\begin{table}[h!]
\centering
\begin{tabularx}{1\textwidth} {
| >{\centering\arraybackslash\hsize=1\hsize}X | >{\centering\arraybackslash\hsize1.8\hsize}X | >{\centering\arraybackslash\hsize=.2\hsize}X |
}
 \hline
 
 $ \ll $, $\ls$, $\asymp$ & & p.\pageref{not-vin}\\
 $\norm{f}_{p(\mu)}$ & $\Inn{\abs{f}^p, \mu}$ & p.\pageref{not-lp} \\
 
$f\circ g(x)$ & $\mathbb{E}_{y}f(y)\overline{g(y-x)}$ & p.\pageref{not-circ} \\ 
$\widehat{f}\circ \widehat{g}(\gamma)$ & $\sum_\lambda \widehat{f}(\lambda)\overline{\widehat{g}(\lambda-\gamma)}$ & p.\pageref{not-circ} \\ 
$f^{(n)}$ & the $n$-fold iterated convolution of $f$ & p.\pageref{not-itconv}\\
 $\mu_A$ & $\frac{\abs{G}}{\abs{A}}\ind{A}$ & p.\pageref{not-norm}\\
$\bal{A}{B}$ & $\mu_A-\mu_B$ (when $A\subset B$) & p.\pageref{not-bal}\\ 
$T(A)$ & the (normalised) count of three-term arithmetic progressions in $A$ & p.\pageref{not-3ap}\\
$B_\rho$ & the Bohr set $B$ with width dilated by $\rho$ & p.\pageref{not-bohrdil} \\ 
Bohr set & & p.\pageref{def-bohr} \\
regularity of Bohr sets & & p.\pageref{def-reg} \\

$\mathrm{rk}(B)$ & the rank of the Bohr set $B$ & p.\pageref{not-rank}\\
$\Delta_\eta(f)$ & the $\eta$-level spectrum of $f$ & p.\pageref{not-spec}\\
density increment of strength $[\delta,d';C]$ & & p.\pageref{def-di}\\

covering by $\Gamma$ & & p.\pageref{def-cov}\\
$\Gamma_{\mathrm{top}},\Gamma_{\mathrm{bottom}},\Gamma^{(i)}$ & various levels of an additive framework & p.\pageref{not-frame}\\

additive framework & & p.\pageref{def-af}\\

$\Gamma$-orthogonality & & p.\pageref{def-orth} \\
$E_{2m}(\omega;\nu)$ & the $2m$-fold additive energy of $\omega$ with respect to $\nu$ & p.\pageref{not-energy}\\

$\Gamma$-dissociativity & & p.\pageref{def-diss} \\
$\dim(\Delta; \Gamma)$ & the largest size of a $\Gamma$-dissociated subset of $\Delta$ & p.\pageref{not-dim}\\

additively non-smoothing & & p.\pageref{def-ans} \\

(multiscale) viscosity & & p.\pageref{def-vis}\\

$\mathcal{S}$ & the collection of symmetric sets that contain $0$ & p.\pageref{not-symm}\\
 \hline
\end{tabularx}
\caption{Index of non-standard notation and definitions}
\label{tab-not}
\end{table}

We fix $G$ to be a finite abelian group of odd order $N$ (which for our application will be $\bbz/N\bbz$). The dual group $\widehat{G}$ is the group of additive characters on $G$, which is a finite abelian group isomorphic to $G$. We use addition to denote the group operation on both $G$ and $\widehat{G}$. For example, if $\gamma_1,\gamma_2:G\to\bbc$ are two characters in $\widehat{G}$ then 
\[(\gamma_1+\gamma_2)(x) = \gamma_1(x)\gamma_2(x).\]

When $A$ is a set we write $\ind{A}$ for the indicator function
\[\ind{A}(x) = \begin{cases} 1&\textrm{if }x\in A\textrm{ and}\\0&\textrm{otherwise.}\end{cases}\]
When $A\subset G$ we write $2\cdot A$ for the dilation of $A$ by 2, so $2\cdot A= \{ 2x : x\in A\}$. 

\subsection*{Asymptotic notation}
We will frequently use the Vinogradov notation, where $X \ll Y$ means that $\abs{X}\leq C\abs{Y}$ for some absolute constant $C>0$. If the constant $C$ depends (in some unspecified fashion) on some parameters $C_1,\ldots,C_r$ then we write $X\ll_{C_1,\ldots,C_r} Y$.\label{not-vin}

There will also be many logarithmic factors which we suppress for clarity; to this end we write $X \lesssim_{\alpha} Y$ to mean that $\abs{X} \leq C_1\log(4/\alpha)^{C_2}Y$ for some constants $C_1,C_2>0$. Related to these, we use the notation $O(Y)$ to denote a quantity that is $\ll Y$, and we use the notation $\tilde{O}_\alpha(Y)$ to indicate a quantity that is $\lesssim_\alpha Y$. The notation $X\asymp Y$ is shorthand for $X\ll Y\ll X$.

Many of our lemmas have some unspecified absolute constant in the hypotheses. We stress that these are all computable, effective constants (and, indeed, often quite reasonable in size). The precise values of these constants are usually unimportant for our application, and they have been left unspecified for clarity of exposition.

\subsection*{Normalisations}
 For any function $f:G\to\bbc$ and non-empty $X\subset G$ we write $\bbe_{x \in X} f(x)$ for the average $\abs{X}^{-1} \sum_{x \in X} f(x)$. This is usually used with $X=G$, and in such cases we abbreviate $\bbe_{x\in G}$ by $\bbe_x$. Our $L^p$ norms on $G$ will be taken with respect to this compact normalisation, and on $\widehat{G}$ with respect to the discrete normalisation. In particular, for $1\leq p<\infty$, if $f:G\to \bbc$ then 
\[\norm{f}_p=\brac{ \bbe_x \abs{f(x)}^p}^{1/p}\textrm{ and }\norm{f}_\infty = \sup_{x\in G} \abs{f(x)},\]
and if $\omega:\widehat{G}\to\bbc$ then 
\[\norm{\omega}_p=\brac{\sum_{\gamma\in \widehat{G}}\abs{\omega(\gamma)}^p}^{1/p}\textrm{ and }\norm{\omega}_\infty = \sup_{\gamma\in\widehat{G}}\abs{\omega(\gamma)}.\]
We will occasionally use $L^p$ norms with respect to other measures, which we will indicate in the subscript. Thus, if $1\leq p<\infty$ and $\mu:G\to \bbr_+$ is any probability measure, then \label{not-lp}
\[\norm{f}_{p(\mu)} = \Inn{\abs{f}^p, \mu}^{1/p}.\]

Similarly, our inner products and convolutions will also use the appropriate normalisations, according as the functions are supported on $G$ or $\widehat{G}$. For example, if $f,g:G\to\bbc$ and $\omega,\nu:\widehat{G}\to\bbc$ then
\[\Inn{f,g} = \bbe_x f(x) \overline{g(x)}\quad\textrm{and}\quad \Inn{\omega,\nu}=\sum_\gamma \omega(\gamma)\overline{\nu(\gamma)}.\]
In particular, if $f:G\to\bbc$ then the Fourier transform $\widehat{f}:\widehat{G}\to\bbc$ is defined by
\[\widehat{f}(\gamma) = \bbe_x f(x)\overline{\gamma(x)}.\]
By these normalisations, Parseval's identity states that for any $f,g:G\to \bbc$
\[\Inn{ f,g } = \Inn{ \widehat{f}, \widehat{g}}.\]
The convolution of two functions $f,g:G\to\bbc$ is defined by 
\[f\ast g(x) = \bbe_{y}f(y)g(x-y).\]
We will frequently use the convenient notation $f\circ g$\label{not-circ} to denote $f\ast g_-$, where $g_-(x)=\overline{g(-x)}$. This operation is not associative in general, but it satisfies that $(f\circ g)\circ h = f\circ(g \circ h_{-})$, so if $h$ is conjugate symmetric then associativity holds, and in such cases we omit brackets. To be explicit, if $f,g:G\to \bbc$, then
\[f\circ g(x) = \bbe_y f(y)\overline{g(y-x)},\]
while
\[\widehat{f} \ast \widehat{g}(\gamma) = \sum_\lambda f(\lambda)g(\lambda-\gamma)\quad\textrm{and}\quad \widehat{f}\circ \widehat{g}(\gamma) = \sum_\lambda f(\lambda)\overline{g(\lambda-\gamma)},\]
These satisfy the adjoint property that
\[\langle f,g\ast h\rangle = \langle f\circ h,g\rangle.\]
We use $f^{(n)}$\label{not-itconv} to denote the $n$-fold repeated convolution. The elementary properties
\[\widehat{f\ast g} = \widehat{f}\cdot \widehat{g}\quad\textrm{and}\quad\widehat{f\circ g} = \widehat{f}\cdot \overline{\widehat{g}}\]
will be used often in what follows. In particular, we will frequently use the fact that, for any $f:G\to\bbc$, 
\[\widehat{f\circ f} = \Abs{\widehat{f}}^2.\]
A probability measure is a non-negative function $\mu:G\to\bbr_+$ such that $\norm{\mu}_1=1$.

We will also sometimes use the inverse Fourier transform, defined for functions $\omega:\widehat{G}\to \bbc$ by 
\[\widecheck{\omega}(x) = \sum_{\gamma\in\widehat{G}}\omega(\gamma)\gamma(x).\]
We note that our normalising conventions ensure that $\widehat{\widecheck{\omega}}=\omega$ and $\widecheck{\hat{f}}=f$.
\subsection*{Densities and balanced functions}If $A \subset B$, then we call the ratio $\abs{A}/\abs{B}$ the \emph{(relative) density} of $A$ with respect to (or in) $B$. When the enveloping set is clear from the context, we generally use the corresponding lower-case Greek character to denote a set's relative density, so that if $A \subset B \subset G$ we write $\alpha=\abs{A}/\abs{B}$. 

Unless otherwise specified, $\mu$ will denote the uniform measure on $G$, so that, for example, $\mu(B)=\norm{\ind{B}}_1=\abs{B}/\abs{G}$. We shall write $\mu_B$\label{not-norm} for two related things: as a function on $G$, it is the normalised indicator function $\mu(B)^{-1} \ind{B}$, and as a measure on subsets of $G$ it is given by $\mu_B(A) = \abs{A \cap B}/\abs{B} = \bbe_{x \in B} \ind{A}(x)$. When $A \subset B$, we shall often be interested in the (\emph{relative}) \emph{balanced function} of $A$\label{not-bal}, which we define as
\[\bal{A}{B}=\mu_A-\mu_B=(\alpha^{-1}\ind{A}-\ind{B})\mu(B)^{-1},\]
where $\alpha$ denotes the relative density of $A$ in $B$. Note in particular that $\bbe \bal{A}{B} = 0$, which is why it is called the balanced function.

It is straightforward to calculate the $L^p$ norms of balanced functions. For example,
\[\norm{\bal{A}{B}}_1=2(1-\alpha)\textrm{ and }\norm{\bal{A}{B}}_2^2 = \mu(B)^{-1}(\alpha^{-1}-1).\]
In particular, provided $\alpha\in (0,1/2]$ we have $\norm{\bal{A}{B}}_1\asymp 1$ and $\norm{\bal{A}{B}}_2^2\asymp \alpha^{-1}\mu(B)^{-1}$. These estimates will be used frequently in what follows.

\subsection*{Dyadic pigeonholing}

We will make frequent use of dyadic pigeonholing. This technique, while elementary, may not be familiar to some, and since it is essential for many of the proofs in this paper we will give some examples here.

We will give exact statements of dyadic pigeonholing in both $L^1$ and $L^2$ forms. We do not refer to these exact statements in the sequel, but they will often be used implicitly, and we hope that any confusion about what is meant by an invocation of `dyadic pigeonholing' is dispelled by consulting the statements and proof below.

\begin{lemma}
If $f:X\to [0,M]$ and $\delta\in(0,1]$ are such that
\[\sum_{x\in X} f(x) \geq \delta M\abs{X}\]
then there exists some $\eta$ with $\delta/2\leq \eta\leq 1$ and $X'\subset X$ of size $\abs{X'}\gs_\delta \delta\eta^{-1}\abs{X}$ such that if $x\in X'$ then $\eta M\leq f(x) < 2\eta M$. 
\end{lemma}
\begin{proof}
Let $\tilde{X}=\{ x\in X : f(x)\geq \tfrac{1}{2}\delta M\}$. By assumption, 
\[\sum_{x\in \tilde{X}}f(x)=\sum_{x\in X}f(x) - \sum_{x\not\in \tilde{X}}f(x) \geq\tfrac{1}{2}\delta M\abs{X}.\]
We now let 
\[X_i = \{ x\in \tilde{X} : 2^{i-1}\delta M\leq f(x) < 2^i \delta M\}.\]
Clearly the $X_i$ are disjoint and $X_i$ is empty for $i> \log_2(2\delta^{-1})$ and for $i<0$. By the pigeonhole principle there exists some $0\leq i\leq \log_2(2\delta^{-1})$ such that
\[\sum_{x\in X_i} f(x) \geq \frac{\delta}{2\lceil \log_2(2\delta^{-1})+1\rceil} M\abs{X}.\]
The conclusion now follows letting $X'=X_i$ and $\eta=2^{i-1}\delta$.
\end{proof}

A very similar argument delivers the following alternative form.
\begin{lemma}
If $f:X\to [0,M]$ and $\delta\in(0,1]$ are such that
\[\sum_{x\in X} f(x)^2 \geq \delta M\sum_{x\in X}f(x)\]
then there exists some $\eta$ with $\delta/2\leq \eta\leq 1$ and $X'\subset X$ of size 
\[\abs{X'}\gs_\delta \delta\eta^{-2}M^{-1}\sum_{x\in X}f(x)\]
such that if $x\in X'$ then $\eta M\leq f(x) < 2\eta M$. 
\end{lemma}

\section{Sketch of the proof}\label{section:sketch}

Our goal is to give an upper bound on the size of sets $A$ with few solutions to $x+y=2z$ (e.g. those with only the trivial $\abs{A}$-many solutions where $x=y=z$). As is common with analytic techniques, we will achieve this by proving a lower bound for the number of solutions in an arbitrary set of a given density. We will consider the normalised count of three-term arithmetic progressions (including the trivial ones) given by \label{not-3ap}
\[T(A)=\bbe_{x,d}\ind{A}(x)\ind{A}(x+d)\ind{A}(x+2d)=\Inn{\ind{A}\ast \ind{A},\ind{2\cdot A}}.\]
Our main result is the following. 
\begin{theorem}\label{mainthm2}
If $G$ is a finite abelian group of odd order and $A\subset G$ has density $\alpha$ then
\[T(A)\geq \exp(-O(\alpha^{-1+c})),\]
where $c>0$ is an absolute constant. 
\end{theorem}

Since a set $A$ with only trivial three-term arithmetic progressions has $T(A)=\alpha/N$, we obtain the following corollary.

\begin{corollary}
If $G$ is a finite abelian group of odd order $N$ and $A\subset G$ is a set with no non-trivial three-term arithmetic progressions then
\[\abs{A} \ll \frac{N}{(\log N)^{1+c}},\]
where $c>0$ is an absolute constant.
\end{corollary}
Theorem~\ref{mainthm} is an easy consequence of this, and follows immediately after embedding $\{1,\ldots,N\}$ into $\bbz/N'\bbz$ for $N'=2N+1$, say, so that any three-term arithmetic progressions in the image of $\{1,\ldots,N\}$ are genuine ones, without any wrap-around issues. 

In the remainder of this section we give a sketch of the proof of Theorem~\ref{mainthm2} in the model case when $G=\mathbb{F}_3^n$. The result thus obtained is the same as that of Bateman and Katz \cite{BaKa:2012} (with a slightly better, though still tiny, value of $c$). The proof is similar in much of its structure to that of \cite{BaKa:2012}, but is different in several key respects. These differences have little impact when $G=\mathbb{F}_3^n$, but are vital to our aim when $G=\bbz/N\bbz$. The value of this setting is that the technical obstacles are significantly reduced, but many of the main concepts are still present, and so we hope that this section will help the reader navigate through what may otherwise seem unmotivated technical statements in the remainder of the paper.

\subsection*{Density increment}
Let $A\subset \bbf_3^n$ be a set of density $\alpha=\abs{A}/3^n$. Our goal is to give a lower bound for $T(A)$. The overall approach, just like Roth's original argument, uses a density increment strategy. If the count of three-term progressions in $A$ differs significantly from what we expect in a random set of density $\alpha$, then there must be some subspace $V\leq \bbf_3^n$ of small codimension on which some translate of $A$ has increased density: $\abs{(A+x)\cap V}/\abs{V}\geq (1+\delta)\alpha$ for some $\delta>0$. Observing that a lower bound for the number of three-term progressions in $(A+x)\cap V$ gives the same lower bound for the number of progressions in $A$, we now repeat the same argument, using $V$ in place of $\bbf_3^n$. Since the density can never exceed 1, this argument must halt in $\tilde{O}_\alpha(\delta^{-1})$ many steps, and provided the final subspace still has reasonable dimension, we may deduce a respectable lower bound for $T(A')$, where $A'$ is some subset of a translate of $A$, and hence for $T(A)$. 

Let us be slightly more precise. Let $V\leq \bbf_3^n$ and suppose that $A$ is a subset of $V$ with relative density $\alpha$. We say that $A$ has a density increment of strength $[\delta,d]$ relative to $V$ if there is some $V'\leq V$ of codimension (in $V$) at most $Cd$ such that 
\[\norm{\ind{A}\ast \mu_{V'}}_\infty=\sup_{x}\frac{\abs{(A+x)\cap V'}}{\abs{V'}}\geq (1+C^{-1}\delta)\alpha,\]
for some $C=\tilde{O}_\alpha(1)$ (which is a fixed quantity we will not specify, as this is only a sketch). It is important to note that, despite the use of the letter $\delta$, the density increment parameter $\delta$ may be significantly larger than 1 (and sometimes significantly smaller). 

Let $K=K(\alpha)\geq 1$ be some decreasing function of $\alpha$ which we will choose later, but which will be fixed throughout the proof. (In parsing what follows it may be helpful to know that we will eventually take $K=\alpha^{-c}$ for some small absolute constant $c>0$.) We will show that if $V\leq \bbf_3^n$ and $A\subset V$ has relative density $\alpha$ then one of the following must hold:
\begin{enumerate}
\item (many progressions) $T(A)\gg \alpha^3\mu(V)^2$, or
\item $A$ has a density increment of strength
\begin{enumerate}
\item (small increment) $[K^{-O(1)},K^{O(1)}]$ or
\item (large increment) $[K,K^{-1}\alpha^{-1}]$
\end{enumerate}
 relative to $V$.
\end{enumerate}
As well as this density increment result, we will also need to use a weaker lower bound for $T(A)$ as a black box; namely, that if $A\subset V\leq \bbf_3^n$ with relative density $\alpha$ then
\begin{equation}\label{eq-mesh}
T(A)\gg \exp(-O(\alpha^{-1}))\mu(V)^2.
\end{equation}
This bound, due to Meshulam \cite{Me:1995}, has a simple Fourier analytic proof, but for this application we do not need to know anything about the proof.

We now explain how to deduce the bound of Theorem~\ref{mainthm2} from this density increment dichotomy. We begin with a subset $A$ of $V=\bbf_3^n$ with density $\alpha$, and repeatedly apply this dichotomy (replacing $A$ with some appropriate subset of a translate and $V$ with an appropriate subspace each time). Since a small increment can occur at most $\tilde{O}_\alpha(K^{O(1)})$ many times, each time increasing the codimension by $\tilde{O}_\alpha(K^{O(1)})$, we must eventually arrive at some subspace $V'$ with codimension $\tilde{O}_\alpha(K^{O(1)})$ and some $A'\subset V'$ of relative density at least $\alpha$, which is a subset of a translate of $A$, such that either we are in the first case and $T(A')\gg \alpha^3\mu(V')^2$, or else we have a large increment. In the latter case there is some $V''\leq V'$ of codimension $\tilde{O}_\alpha(K^{O(1)}+K^{-1}\alpha^{-1})$ and some $A''\subset V''$, which is a subset of a translate of $A$, with relative density $\gs_\alpha K\alpha$. We may then apply Meshulam's bound \eqref{eq-mesh} to yield
\[T(A)\geq T(A'')\gg \exp(-\tilde{O}_\alpha(K^{-1}\alpha^{-1})) \mu(V'')^2.\]
In either case, we have
\[T(A)\gg \exp(-\tilde{O}_\alpha(K^{O(1)}+K^{-1}\alpha^{-1})).\]
Choosing $K=\alpha^{-c'}$ for some sufficiently small constant $c'>0$, therefore, 
\[T(A)\geq \exp(-O(\alpha^{-1+c}))\]
as required. The rest of this proof sketch will explain how these density increments are obtained. For simplicity we will assume that $V=\mathbb{F}_3^n$. 
\subsection*{From progressions to large spectra}
Our starting point, as in the original proof of Roth, is to express $T(A)$ in terms of the Fourier transform $\widehat{\ind{A}}$ and thus deduce spectral information about $\widehat{\ind{A}}$ from a bound on $T(A)$. We first note the identity\footnote{This identity only holds when every non-identity element of $G$ has order $3$, such as $G=\mathbb{F}_3^n$, but a similar identity holds for arbitrary finite abelian groups.}
\[T(A)=\sum_\gamma \widehat{1_A}(\gamma)^3=\alpha^3+\sum_{\gamma\neq 0}\widehat{1_A}(\gamma)^3.\]
In particular, if $T(A)\leq \alpha^3/2$, say, then 
\[\sum_{\gamma\neq 0}\abs{\widehat{1_A}(\gamma)}^3 \gg \alpha^3.\]
Furthermore, by Parseval's identity, we have $\sum\Abs{\widehat{\ind{A}}}^2=\alpha$. Together, these imply that there is some $\gamma\neq 0$ such that $\Abs{\widehat{1_A}(\gamma)}\gg \alpha^2$. This is already non-trivial information, and it can be deduced from this that $A$ has a density increment of strength $[\alpha,1]$, which would give a bound of $T(A)\geq \exp(-O(\alpha^{-1}))$. This was the strategy of Meshulam \cite{Me:1995}.

An application of the dyadic pigeonhole principle, however, allows us to deduce something even stronger. We define the $\eta$-level spectrum of a set $A$ to be 
\[\Delta_\eta(A)=\{\gamma : \Abs{\widehat{1_A}(\gamma)}\geq \eta\alpha\}.\]
By the dyadic pigeonhole principle, for some $\eta$ with $1\geq \eta \gg\alpha$, 
\[\sum_{\substack{\gamma\neq 0\\ \gamma\in\Delta_\eta(A)\backslash\Delta_{2\eta}(A)}}\abs{\widehat{1_A}(\gamma)}^3\gtrsim_\alpha \alpha^3,\]
whence
\[\abs{\Delta_\eta(A)\backslash\{0\}}\gtrsim_\alpha \eta^{-3}.\]
This should be compared to the trivial upper bound of $\abs{\Delta_\eta(A)}\leq \eta^{-2}\alpha^{-1}$ from Parseval's identity. We will obtain our density increments by finding large subsets of the spectrum with relatively small dimension (with dimension meaning `the size of the smallest spanning subset'). An efficient way to capture this is with a standard $L^2$ increment argument, which shows that if
\begin{equation} \textrm{ there exists }\Delta\subset \Delta_\eta(A)\backslash\{0\}\textrm{ with }\lvert \Delta\rvert \geq \delta\eta^{-2}\textrm{ and }\dim\Delta \leq d \label{L2inc_model} \end{equation}
then
\[A\textrm{  has a density increment of strength }[\delta,d].\]
To see why this is true, observe that, for a set $\Delta$ as above, if 
\[V= \{ x \in \mathbb{F}_3^n : \gamma(x) = 1 \textrm{ for all }\gamma\in\Delta\}\]
is the subspace which annihilates all characters in $\Delta$, then $V$ has codimension at most $d$, and 
\begin{align*}
\alpha \norm{\ind{A}\ast \mu_V}_\infty 
&\geq \Inn{\ind{A}\ast \mu_V,\ind{A}\ast \mu_V}\\
&=\Inn{\ind{A}\circ \ind{A},\mu_V\circ \mu_V}\\
&= \sum_\gamma \abs{\widehat{\ind{A}}(\gamma)}^2\abs{\widehat{\mu_V}(\gamma)}^2\\
&\geq \sum_{\Delta\cup\{0\}} \abs{\widehat{1_A}(\gamma)}^2\\
&\geq \alpha^2+\eta^2\alpha^2\abs{\Delta},
\end{align*}
using the fact that $\abs{\widehat{\mu_V}}^2\geq \ind{\Delta\cup\{0\}}$ (indeed, we have exactly $\widehat{\mu_V}= \ind{\textrm{Span}(\Delta)}$). 

Simply using the lower bound $\abs{\Delta_\eta(A)\backslash\{0\}}\gs_\alpha \eta^{-3}$ in this way (and recalling that $\eta$ is some unknown parameter satisfying $1\geq \eta \gg \alpha$) is enough to show that either $T(A)\gg \alpha^3$ or $A$ has a density increment of strength $[1,\alpha^{-3}]$, which would lead to the bound $T(A)\geq \exp(-\tilde{O}_\alpha(\alpha^{-3}))$ -- worse than the bound of Meshulam mentioned above, which only uses a single character! To improve upon this, we need more information than just the size of the spectrum.  
\subsection*{Additive structure of spectra}\label{section:sketchenergy}
We will now fix some $\eta$ such that $\abs{\Delta_\eta(A)}\gs_\alpha  \eta^{-3}$. Our task then is to find a large subset $\Delta\subset \Delta_\eta(A)$ of small dimension. 

A fundamental result in additive combinatorics, proved by Chang \cite{Ch:2002}, states that $\Delta_\eta(A)$ itself has dimension $\tilde{O}_\alpha(\eta^{-2})$. This already would yield an improvement to the above, giving a density increment of strength $[1,\alpha^{-2}]$, and thence $T(A)\geq \exp(-\tilde{O}_\alpha(\alpha^{-2}))$. Chang's lemma has played an important role in many previous approaches to Roth's theorem, and was the first result to show that there is some non-trivial amount of additive structure within spectra. Although we do not require Chang's lemma itself explicitly in this paper, the underlying ideas are related to what follows.

Our principal method of studying the additive structure of spectra will be using higher additive energies. For $m\geq 1$ we define
\[E_{2m}(\Delta)=\abs{\{ \gamma_1+\cdots+\gamma_m=\gamma_1'+\cdots +\gamma_m': \gamma_i,\gamma_i'\in \Delta\}}=\bbe_x \Abs{\widecheck{\ind{\Delta}}(x)}^{2m}.\]
This is trivially bounded between $\abs{\Delta}^m$ and $\abs{\Delta}^{2m-1}$. When $\Delta\subset \Delta_\eta(A)$ we have the extremely useful lower bound
\begin{equation}\label{shenergy}
E_{2m}(\Delta)\geq \alpha\eta^{2m}\abs{\Delta}^{2m}.
\end{equation}
For comparison, note that if $\abs{\Delta}\approx \eta^{-3}$, then this lower bound is $\gg \alpha \abs{\Delta}^{\tfrac{4}{3}m}$, significantly better than the trivial lower bound $\abs{\Delta}^m$, though still far short of the optimal $\abs{\Delta}^{2m-1}$. 

To prove this, we first note that by definition $\Inn{\Abs{\widehat{\ind{A}}},\ind{\Delta}}\geq \eta \alpha \abs{\Delta}$. Writing out the left-hand side in physical space, there exists some choice of signs $c:\Delta\to \bbc$ such that
\[\abs{\bbe_x \ind{A}(x)\sum_{\gamma\in\Delta}c_\gamma \gamma(x)} \geq \eta \alpha\abs{\Delta}.\]
Applying H\"{o}lder's inequality to the left-hand side, then using orthogonality of characters and the triangle inequality to discard the signs $c_\gamma$, one arrives at \eqref{shenergy}. This inequality first seems to have been observed by Shkredov \cite{Sh:2008}, who introduced it to provide a variant proof of Chang's lemma.

To exploit this lower bound, we will use that a set with large higher additive energy contains a large subset with small dimension. This can be proved using random sampling, where the large subset in question is generated by a small number of random elements. This technique was introduced in the paper of Bateman and Katz \cite{BaKa:2012}, and in \cite{Bl:2016} it was used to prove that if $E_{2m}(\Delta)\geq d^{-2m}\abs{\Delta}^{2m}$ then $\Delta$ has a subset $\Delta'\subset \Delta$ of size $\abs{\Delta'}\gg m^{-O(1)}\frac{\abs{\Delta}}{d}$ and dimension $\dim \Delta'\ll m^{O(1)}d$. In particular, if we choose $m =C\lceil\log(2/\alpha)\rceil$ for some large constant $C$ and $d\approx \eta^{-1}$ then the large spectrum $\Delta_\eta(A)$ satisfies the required lower bound on the energy, whence we have some $\Delta\subset \Delta_\eta(A)$ of size $\abs{\Delta}\gs_\alpha \eta^{-2}$ and $\dim \Delta\ls_\alpha \eta^{-1}$. This produces a density increment of strength $[1,\eta^{-1}]$, which at worst is of strength $[1,\alpha^{-1}]$, and hence produces $T(A)\gg \exp(-\tilde{O}_\alpha(\alpha^{-1}))$. It was precisely this strategy, carried out in $\bbz/N\bbz$, which the first author used in \cite{Bl:2016} to obtain logarithmic bounds.

To do better, we first bootstrap this dimension bound using a simple `remove and repeat' procedure. Before proceeding, we recall that our objective is to obtain a density increment either of strength $[K^{-O(1)},K^{O(1)}]$ or $[K,K^{-1}\alpha^{-1}]$, where $K$ is a fixed parameter, from the fact that $\abs{\Delta_\eta(A)}\gs_\alpha \eta^{-3}$ for some $1\geq \eta\gg \alpha$.

We first use the above random sampling argument to find some $\Delta_1\subset \Delta_\eta(A)$ of size $\approx \eta^{-2}$ and dimension $\ls_\alpha \eta^{-1}$. We then remove this $\Delta_1$ from $\Delta_\eta(A)$ and, provided at least half of $\Delta_\eta(A)$ remains, apply the argument again. Repeating this $\approx K$ times and taking the union of all the pieces, and using the trivial upper bound on dimension $\dim(\Delta_1\cup\cdots \cup \Delta_K)\leq \sum \dim\Delta_i$, yields a set $\Delta\subset \Delta_\eta(A)$ of size $\abs{\Delta}\gs K\eta^{-2}$ and dimension $\dim \Delta\ls K\eta^{-1}$. This iterative argument is valid provided $K^{-1}\gs_\alpha \eta$, say. If we further suppose that $\eta \gg K^2\alpha$ then the dimension bound is $\ls_\alpha K^{-1}\alpha^{-1}$, and the $L^2$ method discussed above produces a density increment of strength $[K,K^{-1}\alpha^{-1}]$, and so we have produced a large increment as required.

The regime where $1\geq \eta\gs_\alpha K^{-1}$ is even simpler, since taking a single character from $\Delta_\eta(A)\backslash\{0\}$ produces a small increment of strength $[K^{-1},1]$. The hardest case is when $K^2\alpha \gg \eta\gg \alpha$. For the remainder of this sketch, then, we will suppose that $\eta=\alpha$, and that we have a spectrum $\Delta_\alpha(A)$ of size $\abs{\Delta_\alpha(A)}\gs_\alpha \alpha^{-3}$.

\subsection*{Structure of non-smoothing sets}
The only case remaining is when we have a large spectrum at level $\alpha$, say $\Delta=\Delta_\alpha(A)$ with $\abs{\Delta}\gs_\alpha \alpha^{-3}$. In this case we will show that either $\Delta$ has additive energy large enough that we can deduce a large density increment directly, or else it satisfies a property that (following Bateman and Katz) we call `additive non-smoothing', which allows us to deduce information about the structure of $\Delta$ that we can exploit further.

We first briefly address the case when $\Delta$ has large additive energy. The bound \eqref{shenergy} shows that $E_4(\Delta)\geq \alpha^5\abs{\Delta}^4$. Suppose that something stronger holds, say $E_4(\Delta) \geq K^2\alpha^5\abs{\Delta}^4$. In this case we can apply H\"{o}lder's inequality to find some large $m\asymp \log(2/\alpha)$ such that $E_{2m}(\Delta) \gg (mK\alpha)^{2m}\abs{\Delta}^{2m}$, and hence there exists a $\Delta'\subset \Delta$ of size $\gs_\alpha K\alpha^{-2}$ and dimension $\ls_\alpha K^{-1}\alpha^{-1}$ by the method of the previous subsection, and hence we have a density increment of strength $[K,K^{-1}\alpha^{-1}]$.

Therefore, if there is no large density increment then we can in fact assume that (up to polynomial losses in $K$) we have $E_4(\Delta)\approx \alpha^5\abs{\Delta}^4\approx \alpha^2\abs{\Delta}^3$. By an identical argument we can similarly assume that $E_8(\Delta)\approx \alpha^6\abs{\Delta}^7$, again up to polynomial losses in $K$ (for the rest of this proof sketch we suppress errors that are $K^{O(1)}$). 

If we consider the normalised energies $e_4(\Delta)=E_4(\Delta)/\abs{\Delta}^3$ and $e_8(\Delta)=E_8(\Delta)/\abs{\Delta}^7$, then this can be stated as $e_8\approx e_4^3$. A simple application of H\"{o}lder's inequality shows that, for any set, $e_8\geq e_4^3$, so this shows that the $E_8$ energy of $\Delta$ is very small -- almost as small as possible given $E_4(\Delta)$. The key insight of Bateman and Katz \cite{BaKa:2012} was that this can be leveraged to produce structural information about $\Delta$ -- even when $e_4$ itself is quite small. This is in contrast with the usual methods of additive combinatorics, which require $e_4\gg 1$. These are unable to help us directly here, since we have $e_4\approx \abs{\Delta}^{-2/3}$. 

They call sets $\Delta$ with such a property ($e_8\approx e_4^3$) additively non-smoothing sets, and proved a structural theorem for such sets in certain regimes.  In a rough sense, $\Delta$ being non-smoothing implies that there are some $X,H\subset \Delta$ with $\dim H\ll 1$ such that 
\begin{equation}\label{sketchstruc1}
\abs{X}\abs{H}\approx \alpha^2\abs{\Delta}^2
\end{equation}
and
\begin{equation}\label{sketchstruc2}
E(X,H) = \Inn{ \ind{X}\circ \ind{X}, \ind{H}\circ \ind{H}}\gg \abs{X}\abs{H}^2.
\end{equation}
Such a structural result was proved by Bateman and Katz \cite[Theorem 6.13]{BaKa:2012}. (The exact statement there is specialised to their particular set-up, but the methods of Bateman and Katz would deliver a statement of this type in the general case.) When $G=\mathbb{F}_3^n$, as is the case for this sketch proof, one could use as a black box the result of Bateman and Katz, but for $G=\bbz/N\bbz$ we require a more refined structural result -- namely, one which also applies to sets where we can only control the additive energies `relative' to some other approximately structured set. We will discuss this in more detail later. 

Suppose that \eqref{sketchstruc1} and \eqref{sketchstruc2} hold with $\abs{H}\geq K^{C_1}\alpha^{-1}$ (where $C_1$ is some sufficiently large absolute constant). In this case we remove $H$ from $\Delta_\alpha(A)$ and apply the same procedure to $\Delta_\alpha(A)\backslash H$, and so on, until we have partitioned at least half of $\Delta_\alpha(A)$ into disjoint sets $H_1\sqcup\cdots \sqcup H_r$ with $\dim H_i\ll 1$. For simplicity, suppose that all of these $H_i$ have the same size $\abs{H_i}\approx K^{C_1}\alpha^{-1}$. We then take the union of any $K^{-C_2}\alpha^{-1}$ of these, where $C_2$ is some absolute constant, to create a new set $\Delta'\subset \Delta_\alpha(A)$ with $\abs{\Delta'}\gg K^{C_1-C_2}\alpha^{-2}$ and $\dim \Delta' \ll K^{-C_2}\alpha^{-1}$. Recalling that there are some unspecified losses polynomial in $K$ hidden in the $\ll$ notation, provided $C_1$ and $C_2$ are chosen sufficiently large to overcome these losses, this yields a density increment of $[K,K^{-1}\alpha^{-1}]$, which is sufficient for our purposes. 

The above procedure only works if we can be sure that $\abs{H}\gg K^{C_1}\alpha^{-1}$ each time we apply the structural result, which cannot be assumed. This leads us to the final, and most difficult, case to handle: when an application of the structural result to (some large subset of) $\Delta_\alpha(A)$ produces some $H\subset \Delta_\alpha(A)$ of size $\abs{H}\approx \alpha^{-1}$ and $\dim H\ll 1$, with $E(\Delta,H)\gg \abs{\Delta}\abs{H}^2$. A naive application of the $L^2$ increment method results in a density increment of strength $[\alpha,1]$, which is not strong enough for our purposes. To handle this most difficult case we use the idea of `spectral boosting'. 

\subsection*{Spectral boosting}
The final task is, given some $H\subset \Delta_\alpha(A)$ of size $\abs{H}\approx \alpha^{-1}$ and $\dim H\ll 1$, with $E(\Delta,H)\gg \abs{\Delta}\abs{H}^2$, to produce a suitable density increment.

In the work of Bateman and Katz this case was dealt with by first showing that there is some $X$ such that $\Delta\approx H+X$, where the translates $(H+x)_{x\in X}$ are almost disjoint, and then using a random sampling argument to show that this implies that there must be a single character $\gamma$ such that $\Abs{\widehat{\ind{A}}(\gamma)}\gg K\alpha^{2}$, resulting in a density increment of strength $[K\alpha,1]$. When $G=\mathbb{F}_3^n$ this is sufficient to prove the result we require. In the case of the integers, however, since we have to work with approximate subgroups rather than actual subgroups, this increment is too weak. This is because each time we iterate the argument we incur an overhead cost to deal with the fact that the subgroups are only approximate. These costs quickly build up, and so we would like a density increment of strength $[\delta,d]$ with a larger value of $\delta$ (so that the total number of iterations required is smaller) even if this comes at the cost of increasing $d$.\footnote{Roughly speaking, a density increment of quality $[\delta,d]$ results in the bound $T(A) \gg \exp(-\tilde{O}_\alpha(\delta^{-1}d))$ in $\mathbb{F}_3^n$ but only $T(A) \gg \exp(-\tilde{O}_\alpha(\delta^{-2}d))$ in $\bbz/N\bbz$.}

This means that we need to produce not a density increment of strength $[K\alpha,1]$ from a single large Fourier coefficient, but instead one of strength $[1,1]$, from a large collection of Fourier coefficients with small dimension (which, recalling that we are suppressing errors polynomial in $K$, is actually of strength $[K^{-O(1)},K^{O(1)}]$).

We therefore require a new technique, which we call spectral boosting. The idea of spectral boosting is that, just as subsets of spectra are additively structured, a subset of a spectrum with an unusually large amount of additive structure is forced to have some translate lie in a spectrum of a higher level. That is, the `spectral level' of a set is automatically `boosted' by its inherent additive structure.

To see why this is useful, recall that $H\subset \Delta_\alpha(A)$, and $\abs{H}\approx \alpha^{-1}$. Since $\dim H\ls_\alpha 1$, the $L^2$ increment method naively produces a density increment of strength $[\alpha^2\abs{H},1]=[\alpha,1]$. This just uses the fact that $H$ is a subset of $\Delta_\alpha(A)$. Spectral boosting allows us to exploit the fact that $H$ is very additively structured to find a translate of $H$ which behaves like a subset of $\Delta_{\alpha^{1/2}}(A)$. This means that the $L^2$ increment method now produces a very strong density increment of $[\alpha\abs{H},1]=[1,1]$ as required. We stress that the assumptions needed for spectral boosting to work are quite strict, but fortunately for our application, if they do not hold, then the methods previously discussed can be used instead.

It remains to give a sketch of how spectral boosting works. Suppose then that $\Delta=\Delta_\alpha(A)$ has near-maximal size, $\abs{\Delta}\approx \alpha^{-3}$, and there is a set $H\subset \Delta$ such that
\[\abs{H}\approx \alpha^{-1}\textrm{ and }\langle \ind{\Delta}\circ \ind{\Delta}, \ind{H}\circ \ind{H}\rangle \approx \abs{\Delta}\abs{H}^2.\]
The idea is to combine this additive information with the fact that $\Delta=\Delta_\alpha(A)$ by applying H\"{o}lder's inequality and using the fact that we can assume that $E_{2m}(\Delta)$ is small for large $m$. 

The first step is to use the spectral information $\ind{\Delta}\ll \alpha^{-4}\Abs{\widehat{\ind{A}}}^2$ to obtain the bound
\begin{equation}\label{sketchsb}
\inn{\Abs{\widehat{\ind{A}}}^2\ast \ind{H}\circ \ind{H}, \ind{\Delta}} \gg \alpha^4\abs{\Delta}\abs{H}^2.
\end{equation}
If we try to trivially bound this inner product using $L^1$ and $L^\infty$ norms, we only obtain that $\Norm{\Abs{\widehat{\ind{A}}}^2\ast \ind{H}}_\infty\gg \alpha^4\abs{H}$ -- that is, there is some translate of $H$ on which the average of $\Abs{\widehat{\ind{A}}}$ is $\gg \alpha^2$, which already followed from the fact that $H$ is a subset of the $\alpha$-level spectrum. We will now show that, assuming the energy of $\Delta$ is small, we can in fact show that $H$ (or a translate) behaves like a subset of the $\alpha^{1/2}$-level spectrum.

We first convert \eqref{sketchsb} to physical space, which gives
\[\bbe_x \ind{A}\circ \ind{A}(x)\Abs{\widecheck{\ind{H}}(x)}^2\sum_\gamma \ind{\Delta}(\gamma)\gamma(x)\gg \alpha^4\abs{\Delta}\abs{H}^2.\]
Applying H\"{o}lder's inequality with $m$ some large integer to be chosen shortly, and using the trivial bound $\Abs{\widecheck{\ind{H}}}\leq \abs{H}$, the left-hand side is at most
\[\abs{H}^{1/m}\brac{\bbe_x \ind{A}\circ \ind{A}(x)\Abs{\widecheck{\ind{H}}(x)}^{2}}^{1-1/m}\brac{\bbe_x\ind{A}\circ \ind{A}(x)\abs{\sum_\gamma \ind{\Delta}(\gamma)\gamma(x)}^{m}}^{1/m}.\]
By the Cauchy-Schwarz inequality and Parseval's identity, 
\[\bbe_x\ind{A}\circ \ind{A}(x)\abs{\sum_\gamma \ind{\Delta}(\gamma)\gamma(x)}^m \leq E_4(A)^{1/2}E_{2m}(\Delta)^{1/2}.\]
Trivially, $E_4(A) \leq \alpha^3$. Recalling that $\abs{H}\approx \alpha^{-1}$, therefore, provided we choose $m=C\lceil \log(2/\alpha)\rceil$ for some large constant $C$, we deduce that 
\[\inn{\Abs{\widehat{\ind{A}}}^2, \ind{H}\circ \ind{H}}=\bbe_x \ind{A}\circ \ind{A}(x)\Abs{\widecheck{\ind{H}}(x)}^{2}\gg \brac{\frac{\abs{\Delta}^{2m}}{E_{2m}(\Delta)}}^{1/2m} \alpha^4\abs{H}^2.\]
We recall the lower bound $E_{2m}(\Delta)\geq \alpha^{2m+1}\abs{\Delta}^{2m}$. If $\Delta$ is somewhat generic then we expect this lower bound to be reasonably close to the truth -- for our present application, if it is not, and $E_{2m}(\Delta)$ is much larger than this, then we can find a strong density increment as outlined above. Therefore, recalling that $m\gg \log(1/\alpha)$, 
\[\inn{\Abs{\widehat{\ind{A}}}^2, \ind{H}\circ \ind{H}}\gg \alpha^3\abs{H}^2,\]
so that on average $\Abs{\widehat{\ind{A}}}\gg \alpha^{3/2}$ over (a translate of) $H$ as required.

We must take care when speaking about the `average' in this way, however, since at the moment the left-hand side includes the trivial character, however, which already trivially contributes $\alpha^2\abs{H}\gg \alpha^3\abs{H}^2$ to the inner product, and so the size of the average is being dominated by the contribution from a single large Fourier coefficient, which offers no useful information. Therefore for this to be a useful and non-trivial deduction we must somehow perform the previous calculations avoiding the trivial character throughout. For this we employ the standard trick of replacing $\ind{A}$ at the beginning by its balanced function $f=\ind{A}-\alpha$, which has a zero Fourier coefficient at the trivial character. The above sketch can be repeated with the balanced function instead, but we now face the problem that the application of H\"{o}lder's inequality replaces $f\circ f$ with $\abs{f\circ f}$. This means that the conclusion we arrive at will not be true spectral information for $A$ itself, but rather a weaker `physical side' analogue including an unwanted absolute value.

This roughly corresponds to detecting \emph{some} discrepancy for $\mu_A-1$ relative to a large subspace, but it may well be a density \emph{decrement} over some coset, rather than the desired density \emph{increment}. We therefore finally face the problem of turning a discrepancy for the absolute value of the balanced function of $A$ into a genuine density increment for $A$. This is accomplished using physical methods, namely the almost-periodicity technique introduced by Croot and the second author. This allows us to convert the discrepancy into a genuine density increment of strength $[1,1]$ (which is actually $[K^{-O(1)},K^{O(1)}]$, since we have been suppressing polynomial dependence on $K$), and we are done.

\subsection*{Comparison to the argument of Bateman--Katz.} As is hopefully clear from the sketch, parts of the argument are heavily inspired by the paper of Bateman and Katz, even if we do not use any of the results from \cite{BaKa:2012} directly. For the benefit of the reader familiar with \cite{BaKa:2012}, we mention here some of the similarities and differences between the approaches, when our approach is restricted to the setting of $\bbf_3^n$.

First, in the setting of \cite{BaKa:2012} one can make the assumption at the outset that $\ind{A}$ has no non-trivial large Fourier coefficients\footnote{In this context a large Fourier coefficient is some $\gamma$ such that $\lvert \widehat{\ind{A}}(\gamma)\rvert\geq \alpha^{2-c}$, with $c>0$ some small fixed absolute constant.} whatsoever, which allows one to immediately say that the spectrum $\Delta_\alpha(A)$ is large. Indeed, if there existed some $\gamma\neq 0$ such that $\lvert \widehat{\ind{A}}(\gamma)\rvert\geq \alpha^{2-c}$ then a simple $L^\infty$ density increment argument (due to Meshulam \cite{Me:1995}) immediately implies a density increment of strength $[\alpha^{1-c},1]$. Iterating increments of this strength suffices to obtain a bound of $N/(\log N)^{1+c'}$ (where $c'>0$ is some constant depending on $c$) in the model case of $\bbf_3^n$.

Increments of this type are too weak for the case of $\bbz/N\bbz$, however, due to the lack of subgroups and the need to work with only approximately structured replacements throughout. Iterating a density increment of this strength would only lead to a final bound of the form $N/(\log N)^{1/2+c'}$.

To form an argument which can be generalised to handle $\bbz/N\bbz$, we cannot accept density increments of strength $[\alpha^{1-c},1]$, and hence cannot assume that there are no large Fourier coefficients. Instead, we find a spectrum $\Delta_\eta(A)$ that captures most of the $L^3$-mass of $\widehat{\ind{A}}$, at some unknown level $\eta\in [1,\alpha]$, and work with this throughout. We accordingly develop our tools to work with arbitrary spectral levels. Some of these are applied in the cases where $\eta$ is substantially bigger than $\alpha$; clearly there is no need for such applications in \cite{BaKa:2012}.

In both \cite{BaKa:2012} and this paper we show that some spectrum is additively non-smoothing, and use similar arguments for this. In both approaches one then needs to establish some kind of structural properties for non-smoothing sets. Although we use quite different language for doing this, the underlying ideas are similar in the model setting. The form of the conclusion we aim for is however different: our end goal is to find the subsets $X, H$ with large cross-energy $E(X,H)$, whereas \cite{BaKa:2012} seeks a more explicit structure coming from applications of the asymmetric Balog--Szemer\'edi--Gowers lemma and Freiman's theorem.

Finally, \cite{BaKa:2012} leverages its explicit structural description by considering the fibres of $A$ under quotienting by the subspace coming from the structure theorem. We instead use the spectral boosting technique discussed earlier to convert the energy conclusion on $X, H$ into a viable density increment. This technique is much more robust, and thus can be generalised appropriately when subspaces are replaced by Bohr sets.

\subsection*{Relativising} The remainder of the paper is taken up with carrying out the above sketch in general finite abelian groups (and in $\bbz/N\bbz$ in particular). To make the strategy work in these settings, where there might not be any non-trivial subspaces or subgroups, we shall need to make all definitions and results be relative to sets that are only approximately group-like. There is a well-established source of such sets, namely Bohr sets, whose use in this context was pioneered by Bourgain \cite{Bo:1999} and further developed by Sanders \cite{Sa:2011,Sa:2012}.

Some of the ideas in this sketch have already been adapted to Bohr sets in previous work, in particular the additive energy and dimension ideas of Section~\ref{section:sketchenergy}, which were used in \cite{Bl:2016}. For the most part, however, we have had to develop the tools that we need from scratch in this paper, highlighting where they are repackaged versions of tools that have come before. We have found it necessary to introduce a couple of different kinds of `relativisation' in Fourier space. In some cases the natural thing to work with has been `analytic cut-offs', typically involving convolutions with functions of the form $\abs{\widehat{\mu_B}}^2$. In others, where more combinatorial arguments are needed -- such as with the structure theorem for non-smoothing sets -- we have found it necessary to work with genuine sets, and then relativisation typically involves convolutions with indicator functions of spectra of Bohr sets. To work with the hierarchies of the different spectra that become involved, we introduce the notion of an `additive framework' that captures the relevant properties in a concise definition. It is relative to such an additive framework that we will prove our structural result for additively non-smoothing sets. 

\subsection*{Paper structure}
Sections~\ref{section:bohr} and \ref{section:di} introduce Bohr sets and density increments, and give a rigorous demonstration of how Theorem~\ref{mainthm2} follows from strong enough density increments. Section~\ref{section:addframe} introduces the concept of `additive framework' mentioned above, which is required for both the statement and proof of the kind of structural result for additively non-smoothing sets that we require. Sections~\ref{section:addprop} and \ref{section:lack} implement the basic structure of the sketch above, converting information about a lack of three-term arithmetic progressions into a statement about an additively non-smoothing large spectrum.

Perhaps the most technically demanding parts of the paper are Sections~\ref{section:structure1} and \ref{section:structure2}, which carry out the proof of the required structural result for additively non-smoothing sets. The original proof by Bateman and Katz of a result of this kind was already quite complicated, and an extra dimension of complexity is here introduced owing to the need to produce a structural result which works `relative' to an additive framework (which in this paper will be formed by the spectra of some Bohr sets). We present a sketch of our approach in the non-relative case (which is different to, although has much in common with, the approach used by Bateman and Katz) at the beginning of Section~\ref{section:structure1}, before proving the full relative statement.

In Section~\ref{section:boost} we use spectral boosting, along with almost-periodicity, to leverage the structure of additively non-smoothing sets to obtain a suitable density increment. Section~\ref{section:concluding} brings together the results of the previous sections to conclude the proof of Theorem~\ref{mainthm2}. Finally, Section~\ref{section:spec} contains ideas on how the methods of this paper could be improved.

\section{Bohr sets}\label{section:bohr}
Our overall argument structure is one of density increment, which means that we pass from considering a subset of the group $G$ to a (denser) subset of some subgroup-like structure. When the ambient group is $\mathbb{F}_3^n$ we can take these substructures to be genuine subgroups, the rigidity of which simplifies much of the analysis. Unfortunately, due to the paucity of subgroups of $\mathbb{Z}/N\bbz$, we need to be perform all arguments relative to structures which are only approximately group-like. The appropriate objects are approximate level sets of characters, known as Bohr sets, whose importance within density increment arguments was first realised by Bourgain \cite{Bo:1999}.

\begin{definition}[Bohr sets]
For a non-empty $\Gamma\subset \widehat{G}$ and $\nu:\Gamma\to[0,2]$ we define the Bohr set\label{def-bohr} $B=\mathrm{Bohr}_\nu(\Gamma)$ as 
\[\mathrm{Bohr}_\nu(\Gamma)=\left\{ x\in G : \abs{1-\gamma(x)}\leq \nu(\gamma)\right\}.\]
We call $\Gamma$ the frequency set of $B$ and $\nu$ the width, and define the rank of $B$ to be the size of $\Gamma$, denoted by $\rk(B)$\label{not-rank}. We note here that all Bohr sets are symmetric and contain $0$.

When we speak of a Bohr set we implicitly refer to the triple $(\Gamma,\nu,\mathrm{Bohr}_\nu(\Gamma))$, since the set $\mathrm{Bohr}_{\nu}(\Gamma)$ alone does not uniquely determine the frequency set nor the width. When we use subset notation, such as $B'\subset B$, this refers only to the set inclusion (and does not, in particular, imply any particular relation between the associated frequency sets or width functions). Furthermore, if $B=\mathrm{Bohr}_\nu(\Gamma)$ and $\rho>0$ then we write $B_\rho$\label{not-bohrdil} for the same Bohr set with the width dilated by $\rho$, i.e. $\mathrm{Bohr}_{\rho\nu}(\Gamma)$, which is known as a dilate of $B$.
\end{definition}

Bohr sets are, in general, not even approximately group-like, and may grow exponentially under addition. Bourgain \cite{Bo:1999} observed that certain Bohr sets are approximately closed under addition in a weak sense which is suitable for our applications.

\begin{definition}[Regularity\footnote{The constant $100$ here is fairly arbitrary. Smaller constants are permissible, but this has no significant effect on our arguments.}]
A Bohr set $B$ of rank $d$ is regular\label{def-reg} if for all $\abs{\kappa}\leq 1/100d$ we have 
\[(1-100 d\abs{\kappa})\abs{B}\leq \abs{B_{1+\kappa}}\leq(1+ 100 d\abs{\kappa})\abs{B}.\]
\end{definition}

For further introductory discussion of Bohr sets see, for example, \cite[Chapter 4]{TaVu:2006}, in which the following basic lemmas are established.\footnote{Technically, these lemmas are proved in \cite{TaVu:2006} only when the width function $\nu$ is constant, but the adaptation to our definition is routine. This is done explicitly in, for example, \cite[Lemma 2.12, Lemma 2.14]{Blthesis}. }

\begin{lemma}\label{lemma:bohrreg}
For any Bohr set $B$ there exists $\rho\in[\tfrac{1}{2},1]$ such that $B_\rho$ is regular.
\end{lemma}
\begin{lemma}\label{lemma:bohrsiz}
Let $\Gamma\subset \widehat{G}$ and $\nu,\nu':\Gamma\to[0,2]$ be such that $\nu'(\gamma)\leq \nu(\gamma)$ for $\gamma\in\Gamma$. We have
\[\abs{\mathrm{Bohr}_{\nu'}(\Gamma)}\geq \brac{ \prod_{\gamma\in\Gamma}\frac{\nu'(\gamma)}{4\nu(\gamma)}}\abs{\mathrm{Bohr}_\nu(\Gamma)}.\]
In particular, if $\rho\in (0,1)$ and $B$ is a Bohr set of rank $d$ then $\abs{B_\rho}\geq (\rho/4)^d\abs{B}$.
\end{lemma}

The following lemmas indicate how regularity of Bohr sets will be exploited. Using regularity in this way is a recurring feature in the works of Bourgain \cite{Bo:1999,Bo:2008} and Sanders \cite{Sa:2011,Sa:2012}.

\begin{lemma}\label{lemma:regConv}
If $B$ is a regular Bohr set of rank $d$ and $\mu$ is a probability measure supported on $B_\rho$, with $\rho \in (0,1)$, then
\[ \norm{ \mu_B*\mu - \mu_B }_{1} \ll \rho d. \]
\end{lemma}
\begin{proof}
By the triangle inequality,
\begin{align*}
\bbe_{x\in G} \Abs{ \mu_B \ast \mu(x) - \mu_B(x) } 
&\leq \bbe_{y \in G} \mu(y)\bbe_{x\in G} \Abs{ \mu_B(x-y) - \mu_B(x) }\\
& = \bbe_{y \in G} \mu(y)\frac{ \Abs{ (y+B) \bigtriangleup B } }{\Abs{B}}.
\end{align*}
where $\bigtriangleup$ is the symmetric set difference operator. Since $B_{1-\rho} \subset y+B \subset B_{1+\rho}$ for each $y \in \supp{\mu}$, the definition of regularity implies that this is $O(\rho d)$ provided $\rho \leq 1/100d$. For $\rho > 1/100d$ the statement is trivial.
\end{proof}

We remind the reader that $f^{(L)}$ denotes the $L$-fold convolution of $f$ with itself.
\begin{lemma}\label{lemma:fourierbohr}
There is a constant $c>0$ such that the following holds.  Let $B$ be a regular Bohr set of rank $d$ and $L\geq 1$ be any integer. If $B'\subset B_\rho$ where $\rho \leq c/Ld$ then 
\[\mu_B \leq 2\mu_{B_{1+L\rho}}\ast \mu_{B'}^{(L)}.\]
\end{lemma}
\begin{proof}
We write
\[\ind{B_{1+L\rho}}\ast \mu_{B'}^{(L)}(x) =\bbe_{y_1,\ldots,y_L\in B'} 1_{B_{1+L\rho}}(x-y_1-\cdots-y_L).\]
If $x\in B$ and $y_i\in B_\rho$ for $1\leq i\leq L$, the containment $x-y_1-\cdots-y_L\in B_{1+L\rho}$ is immediate from the definition of a Bohr set and the triangle inequality. It follows that for $x\in B$
\[\ind{B_{1+L\rho}}\ast \mu_{B'}^{(L)}(x) =1.\]
The lemma then follows using the regularity of $B$, which implies that
\[\frac{\abs{B_{1+L\rho}}}{\abs{B}}\leq 1+O(d L\rho)\leq 2\]
provided $\rho\leq c/Ld$ for some sufficiently small constant $c>0$.
\end{proof}

We will be working with Bohr sets on both the physical and frequency side -- for the latter, this entails working with spectra of Bohr sets. We recall the definition of a spectrum, which is a set of large Fourier coefficients.
\begin{definition}[Spectrum]
Let $f:G\to\bbc$ and $\eta\in[0,1]$. The $\eta$-large spectrum is defined to be\label{not-spec}
\[\Delta_\eta(f)=\{ \gamma\in\widehat{G} : \Abs{\widehat{f}(\gamma)}\geq \eta\norm{f}_1 \}.\]
If $f=\ind{A}$ then we write $\Delta_\eta(A)$ for $\Delta_\eta(\ind{A})$. Note that if $f$ takes on only real values then $\Delta_\eta(f)$ is a symmetric set.
\end{definition}

The following lemma collects some useful properties of the spectra of Bohr sets. Similar properties were first observed by Green and Konyagin \cite[Lemma 3.6]{GrKo:2009}.

\begin{lemma}\label{lemma:bohrspectra}
Let $B$ be a regular Bohr set of rank $d$ and $\rho \in (0,1)$. For any $\delta\in(0,1)$ and any $B'\subset B_\rho$
\begin{enumerate} 
\item if $\gamma\in \Delta_\delta(B)$ then 
\[\abs{1-\gamma(x)}\ll \rho d/\delta\quad\textrm{ for all }x\in B_\rho,\]
\item \[\Delta_\delta(B)\subset \Delta_{1-O(\rho d/\delta)}(B'),\]
\item for any $k\geq 1$
\[k \Delta_{1/2}(B)\subset \Delta_{1-O(\rho d k)}(B'),\]
where the left-hand side is the $k$-fold iterated sumset of $\Delta_{1/2}(B)$, and
\item for any $\epsilon\in(0,\tfrac{1}{2})$
\[\Delta_{1/2}(B)+\Delta_{1-\epsilon}(B')\subset \Delta_{1-\epsilon-O(\rho d)}(B').\]
\end{enumerate}
\end{lemma}
\begin{proof}
By definition, for any $x$, if $\gamma\in\Delta_\delta(B)$,
\[\delta\abs{1-\gamma(x)}\leq \abs{\widehat{\mu_B}(\gamma)}\abs{1-\gamma(x)}=\abs{\langle \mu_B,\gamma\rangle-\langle \mu_{B-x},\gamma\rangle}\leq \norm{\mu_{B}-\mu_{B-x}}_1.\]
In particular, if $x\in B_\rho$ then, by the regularity of $B$, the right-hand side is $O(\rho d)$, and the first property follows. The second property follows from the first by the triangle inequality and the fact that, if $\abs{1-\gamma(x)}\leq \epsilon$ for all $x\in B'$, then
\[\abs{1-\widehat{\mu_{B'}}(\gamma)}\leq \bbe_{x\in B'}\abs{1-\gamma(x)}\leq \epsilon.\]

The third property also follows from the first property in a similar way, since by the triangle inequality if $\abs{1-\gamma_i(x)}\leq \epsilon$ for $1\leq i\leq k$ then 
\[\abs{1-(\gamma_1+\cdots+\gamma_k)(x)}\leq k\epsilon.\]
Finally, by the regularity of $B$, for any $\gamma,\lambda\in\widehat{G}$,
\begin{align*}
\abs{\widehat{\mu_{B}}(\gamma)}
\lvert\widehat{\mu_{B'}}(\lambda)&-\widehat{\mu_{B'}}(\gamma+\lambda)\rvert\\
& = \abs{\bbe_{x,y} \mu_B(x)\mu_{B'}(y)\gamma(x)\lambda(y) - \bbe_{z,y}\mu_B(z)\mu_{B'}(y)\gamma(z)(\gamma+\lambda)(y)}\\
&= \abs{\bbe_{x,y}(\mu_{B}(x)-\mu_B(x-y))\mu_{B'}(y)\gamma(x-y)(\gamma+\lambda)(y)}\\
&\leq \bbe_{y\in B'}\bbe_x\abs{\mu_B(x)-\mu_B(x-y)}\\
&\ll \rho d.
\end{align*}
It follows that if $\gamma\in \Delta_{1/2}(B)$ then for any $\lambda\in\widehat{G}$
\[\widehat{\mu_{B'}}(\lambda)=\widehat{\mu_{B'}}(\gamma+\lambda)+O(\rho d),\]
and the fourth property follows.
\end{proof}

\section{Density increments}\label{section:di}
In this section we introduce the precise types of density increment that our argument will employ.

\begin{definition}[Density increments]
Let $B$ be a regular Bohr set, and let $B'\subset B$ be a regular Bohr set of rank $d$. Suppose that $A\subset B$ has density $\alpha$. We say that $A$ has a density increment of strength $[\delta,d'; C]$ relative to $B'$ if there is a regular Bohr set $B''\subset B'$ of rank\label{def-di}
\[\rk(B'')\leq d+Cd'\]
and size
\[\abs{B''}\geq (2d(d'+1))^{-C(d+d')}\abs{B'}\]
such that $\norm{\ind{A}\ast \mu_{B''}}_\infty \geq (1+C^{-1}\delta)\alpha$. That is, some translate of $A$ has relative density within $B''$ at least $(1+C^{-1}\delta)\alpha$. The parameter $C$ should be thought of as a constant factor, and in our applications will always be $\tilde{O}_\alpha(1)$. 
\end{definition}

The reader may be slightly puzzled by this definition, since no condition on the size of $B'$ relative to $B$ is imposed. This small flexibility is useful in what follows, but the reader should be reassured that we will only deal with density increments relative to $B'$ where $B'$ is reasonably large within the enveloping set $B$. 

Although $A$ is a subset of $B$, where $B$ rarely changes in the proof, we will be taking our density increments relative to many different $B'$. To simplify some of these, the following lemma is useful. 

\begin{lemma}\label{lemma:disimp}
Let $B$ be a regular Bohr set and $B'\subset B$ be a regular Bohr set of rank $d$. Let $\rho\in(0,1]$. If $A\subset B$ has a density increment of strength $[\delta,d';C]$ relative to $B'_{\rho/d}$ then $A$ has a density increment of strength $[\delta,d';\tilde{O}_\rho(C)]$ relative to $B'$. 
\end{lemma}
\begin{proof}
This follows immediately from the definition of density increment, after noting that, by Lemma~\ref{lemma:bohrsiz}
\[\Abs{B'_{\rho/d}}\geq (\rho/4d)^{d}\abs{B'},\]
and that the rank of $B'_{\rho/d}$ is the same as the rank of $B'$. 
\end{proof}

We will use the following density increment lemma as a black box. By itself it is sufficient to prove that a set $A\subset \{1,\ldots, N\}$ free of non-trivial three-term arithmetic progressions has size $\abs{A}\ll N/(\log N)^{1-o(1)}$. The most direct reference for this lemma is \cite[Proposition 5.7]{BlSi:2019}, where it is proved as stated below. This result was first proved, via different methods (and in slightly different language), in the work of Sanders \cite[Lemma 6.2]{Sa:2011}, with an alternative proof in \cite[Theorem 7.1]{Bl:2016}.

\begin{lemma}[\cite{Sa:2011,Bl:2016,BlSi:2019}]\label{lem-olddi}
Let $B$ be a regular Bohr set of rank $d$ and suppose that $A\subset B$ has density $\alpha$. Then either $T(A)\geq \exp(-\tilde{O}_\alpha(d\log 2d))\mu(B)^2$ or $A$ has a density increment of strength $[1,\alpha^{-1};\tilde{O}_\alpha(1)]$ relative to $B$. 
\end{lemma}

Iterating this increment yields the following result, implicit in \cite{Sa:2011,Bl:2016,BlSi:2019}, stated here in a suitably relative form. The proof is a routine application of Lemma~\ref{lem-olddi}, which we include in full here, both because there is no version of Theorem~\ref{oldbound} in the form that we require available in the literature, and because it is a simpler version of the more elaborate density increment argument which follows. 

\begin{theorem}[\cite{Sa:2011,Bl:2016,BlSi:2019}]\label{oldbound}
Let $B$ be a regular Bohr set of rank $d$ and suppose that $A\subset B$ has density $\alpha$. Then
\[T(A)\gg \exp(-\tilde{O}_{\alpha}(d+ \alpha^{-1})\log 2d)\mu(B)^2.\]
\end{theorem}
\begin{proof}
An immediate consequence of Lemma~\ref{lem-olddi} and the definition of density increment is the existence of some $C=\tilde{O}_\alpha(1)$ such that for every regular Bohr set $B^*$ of rank $d^*$, whenever $A^*\subset B^*$ with density $\alpha^*\geq \alpha$, either $T(A^*)\geq \exp(-Cd^*\log 2d^*)\mu(B^*)^2$, or else there exists a regular Bohr set $(B^*)'\subset B^*$ of rank at most $d^*+C\alpha^{-1}$ and size
\[\abs{(B^*)'}\geq \exp(-C(d^*+\alpha^{-1})\log 2d^*)\abs{B^*},\]
such that there is some translate of $A$ whose intersection with $(B^*)'$ has relative density at least $(1+C^{-1})\alpha^*$. 

Let $B^{(0)}=B$ and $\ell\geq 0$ be maximal such that there exists a sequence of regular Bohr sets 
\[B^{(0)}\supset B^{(1)}\supset \cdots \supset B^{(\ell)}\]
such that
\begin{enumerate}
\item the rank of $B^{(i)}$, denoted by $d_i$, satisfies $d_i\leq d+i\cdot C\alpha^{-1}$ for $0\leq i\leq \ell$, 
\item the size of $B^{(i)}$ satisfies  
\[\Abs{B^{(i)}}\geq  \exp\brac{-C(d_i+\alpha^{-1})\log 2d_i}\Abs{B^{(i-1)}},\]
for $1\leq i\leq \ell$,
\item there are $A^{(i)}\subset B^{(i)}$ such that $A^{(0)}=A$ and $A^{(i+1)}$ is a subset of a translate of $A^{(i)}$ for $0\leq i<\ell$, and 
\item if $A^{(i)}$ has relative density $\alpha_i$ in $B^{(i)}$ then 
\[\alpha_i\geq (1+ C^{-1})^i \alpha.\]
\end{enumerate}
Note that the set of $\ell$ which satisfy all these requirements (except possibly maximality) is non-empty, since $\ell=0$ is permissible. Furthermore, the set of such $\ell$ is bounded above, since by condition (4), and the fact that relative density can never exceed 1, we have that any such $\ell$ satisfies
\[\ell \leq \log(1/\alpha)/\log(1+C^{-1})\ls_\alpha 1.\]
Therefore the notion of a maximal such $\ell$ is well-defined. We further note that this upper bound on $\ell$ implies that $d_i \leq d+\tilde{O}_\alpha(\alpha^{-1})$ for all $1\leq i\leq \ell$, and hence $\log 2d_i\ls_\alpha \log 2d$. 

We apply Lemma~\ref{lem-olddi} to $A^{(\ell)}\subset B^{(\ell)}$, in the form described in the first paragraph of the proof. The maximality of $\ell$ and the conditions above imply immediately that the density increment alternative cannot hold. Therefore
\[ T(A)\geq T(A^{(\ell)})\gg \exp(-\tilde{O}_\alpha(d_\ell \log 2d_\ell)\mu(B^{(\ell)})^2.\]
The conclusion now follows since $d_\ell \leq d+\tilde{O}_\alpha(\alpha^{-1})$ and a simple induction shows that
\[\mu(B^{(\ell)})\geq \exp(-\tilde{O}_\alpha(d+\alpha^{-1})\log 2d)\mu(B).\qedhere\]
\end{proof}

This result will be used in our arguments when we obtain large density increments. Roughly speaking, these are when our density increment $[\delta, d]$ has $\delta$ much larger than $1$, so that we move to a much denser set. Once we have done so, we immediately apply Theorem~\ref{oldbound} to bound $T(A)$. 

The technical heart of this paper is the following proposition. (In reading the statement it might be helpful to know that this will be applied with $k$ being some large, but absolute, constant.)

\begin{proposition}\label{mainprop}
There is a constant $C > 0$ such that, for all $k\geq C$, the following holds. Let $B$ be a regular Bohr set of rank $d$ and suppose that $A\subset B$ has density $\alpha$. Either 
\begin{enumerate}
\item $\alpha \geq 2^{-O(k^2)}$,
\item 
\[T(A)\gg \exp(-\tilde{O}_{\alpha}(d\log 2d))\mu(B)^2,\]
or
\item $A$ has a density increment of one of the following strengths relative to $B$:
\begin{enumerate}
\item (small increment) $[\alpha^{O(\epsilon(k))}, \alpha^{-O(\epsilon(k))}; \tilde{O}_\alpha(1)]$, or 
\item (large increment)  $[\alpha^{-1/k}, \alpha^{-1+1/k};\tilde{O}_\alpha(1)]$,
\end{enumerate}
where $\epsilon(k)=\frac{\log\log \log k}{\log\log k}$.
\end{enumerate}
\end{proposition}

Given this proposition, a routine iterative argument delivers Theorem~\ref{mainthm2}. The precise form of $\epsilon(k)$ is not particularly relevant for our application -- any function that $\to 0$ as $k\to \infty$ would suffice (with correspondingly worse values for the final value of $c$ as the decay rate decreased).

\begin{proof}[Proof of Theorem~\ref{mainthm2}]
We fix some $A\subset G$ with density $\alpha$. In this proof $\alpha$ will always denote the density of this initial $A$. Let $1\leq C_1=O(1)$ be some fixed quantity, chosen in particular larger than the implied constants in the exponents of the small increment case of Proposition~\ref{mainprop}. Let $k$ be some constant large enough such that Proposition \ref{mainprop} holds and 
\[10C_1\epsilon(k)\leq \frac{1}{2}.\]
Let $1\leq C_2=\tilde{O}_\alpha(1)$ be some quantity depending only on $\alpha$, chosen in particular larger than the implicit constants of Proposition~\ref{mainprop} hidden in the $\gg$, $O(\cdot)$, and $\tilde{O}_\alpha(\cdot)$ notations. 
We note that we may assume that $\alpha \leq 1/2C_2^2$, or else we are done by an application of Theorem~\ref{oldbound} with $B=G$. We may similarly suppose that 
\[\alpha\leq 2^{-C_2k^2},\]
and
\[\log(1/\alpha) \leq \alpha^{-C_1\epsilon(k)}.\]

Let $B^{(0)}=G$, which we regard as a regular Bohr set of rank $1$. Let $\ell\geq 0$ be maximal such that there exists a sequence of regular Bohr sets
\[B^{(0)}\supset B^{(1)}\supset \cdots \supset B^{(\ell)}\]
with ranks
\[ d^{(i)}\leq 1+i\cdot C_2\alpha^{-C_1\epsilon(k)},\]
sizes 
\[\Abs{B^{(i)}}\geq  \exp\brac{-20C_2^3\alpha^{-4C_1\epsilon(k)}}\Abs{B^{(i-1)}},\]
for $1\leq i\leq \ell$ and associated sets $A^{(i)}\subset B^{(i)}$ such that $A^{(0)}=A$ and $A^{(i+1)}$ is a subset of a translate of $A^{(i)}$, and furthermore, if $A^{(i)}$ has relative density $\alpha^{(i)}$ inside $B^{(i)}$ then 
\begin{equation}\label{eq:densbound}
\alpha^{(i)}\geq (1+ C_2^{-1}\alpha^{C_1\epsilon(k)})^i \alpha.
\end{equation}
Note that the set of $\ell$ which satisfy all these requirements (except possibly maximality) is non-empty, since $\ell=0$ is permissible. Furthermore, the set of such $\ell$ is bounded above, since by equation \eqref{eq:densbound}, and the fact that relative density can never exceed 1, we have that any such $\ell$ satisfies
\[\ell \leq 2C_2\alpha^{-2C_1\epsilon(k)},\]
say. Therefore the notion of a maximal such $\ell$ is well-defined.

We now apply Proposition~\ref{mainprop} to $A^{(\ell)}\subset B^{(\ell)}$.  Note that the rank of $B^{(\ell)}$ satisfies
\[d^{(\ell)}\leq C_2\alpha^{-C_1\epsilon(k)}\ell+1\leq 2C_2^2\alpha^{-3C_1\epsilon(k)}+1.\]
It immediately follows that we also have the cruder bound $d^{(\ell)}\leq 2\alpha^{-3}$, say. We may bound the size of $B^{(\ell)}$ using induction by
\[\mu(B^{(\ell)})\geq \exp(-\tilde{O}_\alpha(\ell \alpha^{-4C_1\epsilon(k)}))\geq \exp(-\tilde{O}_\alpha(\alpha^{-6C_1\epsilon(k)})).\]

Furthermore, although we have applied it with $\alpha$ replaced by $\alpha^{(\ell)}$, all the implicit constants in the $\tilde{O}_\alpha(\cdot)$ notation remain bounded by our choice of $C_2$, since $\log(1/\alpha)$ is decreasing as $\alpha$ increases, and $\alpha^{(\ell)}\geq \alpha$.

Suppose first that we are in the small increment case, so that there is a density increment of strength $[\alpha^{C_1\epsilon(k)},\alpha^{-C_1\epsilon(k)};C_2]$ relative to $B^{(\ell)}$. By definition, there exists a regular Bohr set $B'\subset B^{(\ell)}$ of rank 
\[\rk(B')\leq d^{(\ell)}+ C_2\alpha^{-C_1\epsilon(k)}\leq 1+(\ell+1)C_2\alpha^{-C_1\epsilon(k)},\]
size at least (after some simplification, using that $d^{(\ell)}\leq 2\alpha^{-3}$)
\[\abs{B'}\geq \exp\brac{-20C_2^3\alpha^{-4C_1\epsilon(k)}}
\Abs{B^{(\ell)}}.\]
and such that there exists some translate $A'$ of $A^{(\ell)}$ such that 
\[\mu_{B'}(A')\geq (1+C_2^{-1}\alpha^{C_1\epsilon(k)})\alpha^{(\ell)}.\]
In particular, choosing $B^{(\ell+1)}=B'$ and $A^{(\ell+1)}=A'\cap B'$, this contradicts the maximality of $\ell$, and so the small increment case of Proposition~\ref{mainprop} cannot occur.

Suppose that the first case occurs, so that $\alpha^{(\ell)}\geq 2^{-C_2k^2}$. In this case we apply Theorem~\ref{oldbound} for the bound
\[T(A)\geq T(A^{(\ell)})\gg \exp\brac{-\tilde{O}_\alpha\brac{ \alpha^{-3C_1\epsilon(k)}+2^{C_2k^2}}}\mu(B^{(\ell)})^2.\]
In the second case, we directly obtain
\[T(A)\geq T(A^{(\ell)})\gg \exp\brac{-\tilde{O}_\alpha\brac{\alpha^{-3C_1\epsilon(k)}}}\mu(B^{(\ell)})^2.\]

Finally, in the large increment case, we have some regular Bohr set $B'\subset B^{(\ell)}$ with rank 
\[\rk(B')\leq d^{(\ell)}+C_2\alpha^{-1+1/k},\]
size at least
\[\abs{B'}\geq\exp\brac{-\tilde{O}_\alpha\brac{\alpha^{-3C_1\epsilon(k)}+\alpha^{-1+1/k}}}\Abs{B^{(\ell)}},\]
and some translate $A$, say $A'$, such that $\mu_{B'}(A')\gs_\alpha \alpha^{1-1/k}$. We now apply Theorem~\ref{oldbound} once again. 

In any of these three cases, we have obtained a lower bound of at least
\[T(A)\gg \exp\brac{-\tilde{O}_\alpha\brac{\alpha^{-6C_1\epsilon(k)}+\alpha^{-1+1/k}+2^{C_2k^2}}}.\]
The result follows with $c=1/2k$, say.
\end{proof}

It remains to prove Proposition~\ref{mainprop}, which will be the goal of the rest of this paper. In the remainder of this section we will prove three different lemmas which we will use in obtaining density increments. The most common method will be using $L^2$ Fourier concentration on a `low-dimensional' set of characters. The precise notion of being low-dimensional that we use is captured by the following. 

\begin{definition}[Covering]\label{acdef-cov}
We say that $\Delta$ is $d$-covered\label{def-cov} by $\Gamma$ if there is a set $\Lambda$ of size $\abs{\Lambda}\leq d$ such that 
\[\Delta \subset \langle \Lambda\rangle + \Gamma-\Gamma,\]
where
\[\langle \Lambda\rangle = \left\{ \sum_{\lambda\in \Lambda} c_\lambda \lambda : c_\lambda\in \{-1,0,1\}\right\}.\]
\end{definition}

We now show that a large $L^2$ Fourier mass on a set of small dimension can be converted into a good density increment. The idea of obtaining a density increment from $L^2$ information (rather than the $L^\infty$ approach of Roth \cite{Ro:1953}) first appeared in the work of Heath-Brown \cite{He:1987} and Szemer\'{e}di \cite{Sz:1990}. The following will be a useful tool in obtaining both large and small density increments. 

\begin{lemma}\label{lemma:L2inc}
There is a constant $c>0$ such that the following holds. Let $B\subset G$ be a regular Bohr set of rank $d$ and suppose that $A\subset B$ has density $\alpha$. Let $\delta>0$ be some parameter. Suppose $B' \subset B_\rho$ is a regular Bohr set, where $\rho\leq  c\delta\alpha/d$.\footnote{There will be many conditions of this type, where the dilate cannot be too large, but the precise forms of the bounds in these conditions is not important -- in general, any bound of the type $\rho=(\alpha/d)^{O(1)}$ would be enough.}

If there is a set $\Delta$ which is $D$-covered by $\Delta_{1/2}(B')$ such that
\[\sum_{\gamma\in\Delta}\abs{\widehat{\bal{A}{B}}(\gamma)}^2\geq \delta\mu(B)^{-1}\]
then $A$ has a density increment of strength $[\delta,D;O(1)]$ relative to $B'$.
\end{lemma}
\begin{proof}
Let $\Lambda$ be a set of size $\abs{\Lambda} \leq D$ as given by the definition of covering. If $B'=\mathrm{Bohr}_{\nu'}(\Gamma)$ then let $B''=\mathrm{Bohr}_{\nu''}(\Gamma\cup \Lambda) \subset B'_\kappa$, where 
\[ \nu''(\gamma)=\begin{cases} 
    \kappa \nu'(\gamma) & \text{ if $\gamma\in\Gamma$ and }\\
    \frac{\kappa}{D}& \text{ if $\gamma\in\Lambda$,}
    \end{cases}
\]
where we take the minimum of these widths if $\gamma$ lies in $\Gamma\cap \Lambda$, and $\kappa \leq 1/4$ is to be specified later but is chosen so that $B''$ is regular, and will satisfy $\kappa \gg 1/\rk(B')$. This $B''$ will be the Bohr set on which we have a density increment; for the strength of the increment note that $\rk(B'')\leq \rk(B')+D$, and that by Lemma~\ref{lemma:bohrsiz} (viewing $B'$ as a Bohr set on frequency set $\Gamma\cup \Lambda$ by letting $\nu'(\lambda)=2$ for $\lambda\in \Lambda\backslash\Gamma$),
\[\abs{B''}\geq (\kappa/4D)^{O(\rk(B')+D)}\abs{B'}.\]

We now show that $\abs{\widehat{\mu_{B''}}(\lambda)}\geq 1/2$ for every $\lambda \in \Delta$. Indeed, we show the stronger property that $\abs{1-\lambda(x)}\leq 1/2$ for every $\lambda\in \Delta$ and $x\in B''$. Fix $x \in B''$. Every $\lambda \in \Delta$ can be written as the sum or difference\footnote{Recall that we are using additive notation for the group operation on the dual group; thus $(\gamma_1+\gamma_2)(x) = \gamma_1(x)\gamma_2(x)$.} of at most $D$ elements from $\Lambda$ and 2 elements from $\Delta_{1/2}(B')$. For $\lambda \in \Lambda$ we have, by construction, $\abs{1 - \lambda(x)} \leq \kappa/D\leq 1/4D$. For $\gamma \in \Delta_{1/2}(B')$ we have, by Lemma~\ref{lemma:bohrspectra}, $\abs{1-\gamma(x)} \leq 1/8$ provided $\kappa$ is a sufficiently small constant multiple of $1/\rk(B')$. Thus, for an arbitrary $\lambda = \gamma_1 - \gamma_2 \pm\lambda_1 \pm \cdots \pm \lambda_j \in \Delta$ ($j \leq D$) and any $x \in B''$,
\[ \abs{1 - \lambda(x)} \leq \abs{1-\gamma_1(x)} + \abs{1-\gamma_2(x)} + \sum_j \abs{1-\lambda_j(x)} \leq 1/2, \]
as we wished to show.

Thus
\[\norm{\bal{A}{B}\ast \mu_{B''}}_2^2=\sum_{\gamma}\abs{\widehat{\bal{A}{B}}(\gamma)}^2\abs{\widehat{\mu_{B''}}(\gamma)}^2\geq \tfrac{1}{4}\delta\mu(B)^{-1}.\]
Expanding out the left hand side using $\bal{A}{B}=\mu_A-\mu_B$ yields
\[\norm{\mu_A\ast \mu_{B''}}_2^2+\norm{\mu_B\ast \mu_{B''}}_2^2-2\langle \mu_B\ast \mu_{B''}, \mu_A\ast \mu_{B''}\rangle \geq \tfrac{1}{4} \delta\mu(B)^{-1}.\]
By the regularity of $B$, and since $\mu_{B''}*\mu_{B''}$ is supported on $B''+B''\subset B'\subset B_{\rho}$, by Lemma~\ref{lemma:regConv},
\[ \abs{\langle \mu_B\ast \mu_{B''}, \mu_A\ast \mu_{B''}\rangle-\mu(B)^{-1}} = \abs{\langle \mu_B\ast \mu_{B''}\ast \mu_{B''}-\mu_B,\mu_A\rangle} \ll \rho d\alpha^{-1}\mu(B)^{-1}, \]
and so provided $\rho\leq c\delta\alpha/d$ where $c$ is a small enough constant,
\[ \langle \mu_B\ast \mu_{B''}, \mu_A\ast \mu_{B''}\rangle \geq \mu(B)^{-1} - \tfrac{1}{16}\delta\mu(B)^{-1}. \]
Using the trivial bound $\norm{\mu_B\ast \mu_{B''}}_2^2 \leq \mu(B)^{-1}$, it follows that
\[\norm{\mu_A\ast \mu_{B''}}_2^2\geq (1+\tfrac{1}{8}\delta)\mu(B)^{-1},\]
whence $\norm{\mu_A\ast \mu_{B''}}_\infty \geq (1+\tfrac{1}{8}\delta)\mu(B)^{-1}$, providing a density increment of the required strength.
\end{proof}

The following lemma is a variant of Lemma~\ref{lemma:L2inc} which is useful when we have some kind of spectral information on a set which is possibly quite small (e.g. $\ll \alpha^{-O(1)}$), but which is in some sense `orthogonal' to the spectrum of a Bohr set. Unlike Lemma~\ref{lemma:L2inc}, which will be used to produce both large and small increments, the following will only be used to produce large increments. 

\begin{lemma}\label{lemma:lowdiminc}
There is a constant $c>0$ such that the following holds.
Let $B$ be a regular Bohr set of rank $d$ and  suppose that $A\subset B$ has density $\alpha$. Let $\Delta\subset \widehat{G}$ be some set and $K\geq 1$ be some parameter.

Suppose that $B'\subset B_\rho$ is a regular Bohr set, where $\rho \leq c\alpha^3/dK$. Suppose further that $B''=B'_{\rho'}$ is another regular Bohr set (for some dilate $\rho'>0$). Furthermore, suppose that
\begin{enumerate}
\item
\[\abs{\widehat{\bal{A}{B}}}^2\circ \abs{\widehat{\mu_{B'}}}^2(\gamma)\geq K^{-1}\alpha^2\mu(B)^{-1}\]
for all $\gamma\in \Delta$, 
\item \[\norm{\ind{\Delta}\ast \abs{\widehat{\mu_{B'}}}^2}_\infty \leq 2,\]
and
\item $\Delta$ is $D$-covered by $\Delta_{1/2}(B'')$.
\end{enumerate}
Then $A$ has a density increment of strength $[\tfrac{1}{K}\alpha^2\abs{\Delta},D;\tilde{O}_{\alpha/K}(1)]$ relative to $B''$.
\end{lemma}
\begin{proof}
We have
\[\Inn{\abs{\widehat{\bal{A}{B}}}^2, \ind{\Delta}\ast \abs{\widehat{\mu_{B'}}}^2}\geq K^{-1}\alpha^2\mu(B)^{-1}\abs{\Delta}.\]
Let
\[\Delta_0 = \{ \gamma :  \ind{\Delta}\ast \abs{\widehat{\mu_{B'}}}^2(\gamma) \geq \tfrac{1}{2K}\alpha^3\abs{\Delta}\}.\]
Since $\norm{\bal{A}{B}}_2^2\leq \alpha^{-1}\mu(B)^{-1}$, we have
\[\Inn{\abs{\widehat{\bal{A}{B}}}^2, \ind{\Delta_0}}\geq \tfrac{1}{4K}\alpha^2\mu(B)^{-1}\abs{\Delta}.\]

We will apply Lemma~\ref{lemma:L2inc} to obtain our density increment, for which we require $\Delta_0$ to be covered efficiently. Although $\Delta$ itself is covered efficiently by $\Delta_{1/2}(B'')$ by assumption, it does not follow that $\Delta_0$ is, and we will instead cover it by $\Delta_{1/2}(B''')$ for some appropriate $B'''\subset B''$.

If $\gamma\in \Delta_0$, then by averaging there exists some $\lambda\in \Delta$ such that
\[\abs{\widehat{\mu_{B'}}(\gamma-\lambda)}^2\geq \tfrac{1}{2K}\alpha^3.\]
By assumption, there exists a set $\Lambda$ of size at most $D$ such that $\Delta\subset \langle \Lambda\rangle+2 \Delta_{1/2}(B'')$. We therefore have
\[\Delta_0\subset \langle \Lambda\rangle + 2\Delta_{1/2}(B'')+\Delta_{(\alpha^3/2K)^{1/2}}(B').\]
Let $B'''=B''_{c\alpha^3/K\rk(B')}$, where $c>0$ is some small absolute constant, chosen in particular such that $B'''$ is regular. By two applications of Lemma~\ref{lemma:bohrspectra} we have (assuming $c$ is chosen sufficiently small) first that 
\[2\Delta_{1/2}(B'')\subset \Delta_{1/2}(B''')\]
and secondly
\[\Delta_{(\alpha^3/2K)^{1/2}}(B')\subset \Delta_{1/2}(B''').\]
It follows that $\Delta_0$ is $D$-covered by $\Delta_{1/2}(B''')$, so by Lemma~\ref{lemma:L2inc} we have a density increment of strength $[\tfrac{1}{K}\alpha^2\abs{\Delta}, D; O(1)]$ relative to $B'''$. The required density increment relative to $B''$ follows by Lemma~\ref{lemma:disimp}.
\end{proof}

Finally, there is a third method we will use to produce a density increment. Unlike the previous two lemmas, which work on the spectral side, the proof of this density increment uses physical methods. We will use the following form of almost-periodicity, a straightforward consequence of \cite[Theorem 6.7]{BlSi:2019}, recast into a form suited to our application. 

\begin{lemma}\label{lemma:secondap}
There is a constant $c>0$ such that the following holds. Let $\delta,\epsilon\in (0,1)$ and $m\geq 1$ be some parameters. Let $A,L\subset G$ with $\eta=\abs{A}/\abs{L}\leq 1$, let $B$ and $B'$ be regular Bohr sets of rank $d$ such that $B'\subset B_\rho$, where $\rho \leq c/d$. Suppose that $A\subset B$ has density $\alpha$. There is a regular Bohr set $B''\subset B'$ of rank at most $d+d'$ and size
\[ \abs{B''}\geq (\delta^m\eta/2dd')^{O(d+d')}\abs{B'},\]
where
\[d' \ls_{\delta\eta\alpha} m \epsilon^{-2}\]
such that 
\[ \norm{\mu_A\ast \ind{L}\ast \mu_{B''} - \mu_A\ast \ind{L}}_{2m(\mu)}\leq \epsilon\norm{\mu_A\ast \ind{L}}_{m(\mu)}^{1/2}+ \epsilon^{2-1/m}\norm{\mu_A\ast \ind{L}}_{1(\mu)}^{1/2m}+\delta.\]
where $\mu=\mu_{B'}\ast \mu_{B'}$.
\end{lemma}
\begin{proof}
This is a consequence of Theorem~6.7 of \cite{BlSi:2019} applied relative to the Bohr set $B'_{\rho'}$, where $\rho'$ will be chosen later, but in particular chosen such that $B'_{\rho'}$ is regular. We first note that, by regularity, if we let $S=B'_{\rho'}\subset B_\rho$, then provided $\rho$ is sufficiently small, 
\[\abs{A+S}\leq \abs{B+B'}\leq 2\abs{B}\leq 2\alpha^{-1}\abs{A},\]
so the $K$ parameter can be chosen to be $2\alpha^{-1}$. 

We will use the pair of measures $(\nu,\mu)=(\mu_{B'_{1-r\rho'}}\ast \mu_{B'} , 2\mu_{B'}\ast \mu_{B'})$. We need to check that this pair is $rB'_{\rho'}$-invariant, in the language of \cite{BlSi:2019}, for some $r\geq C\log(2/\delta\eta)$ for some large constant $C$. That is, if $t\in rB'_{\rho'}$ and $x\in G$, then 
\[\mu_{B'_{1-r\rho'}}\ast \mu_{B'}(x+t) \leq 2\mu_{B'}\ast \mu_{B'}(x).\]
To this end, we observe that for such $t$ and $x$, if $r\rho'\leq c'/d$ for some sufficiently small $c'>0$, then
\begin{align*}
\mu_{B'_{1-r\rho'}}\ast \mu_{B'}(x+t) 
& = \frac{N}{\Abs{B'_{1-r\rho'}}\abs{B'}}\abs{ (B'_{1-r\rho'}-t)\cap (B'+x)}\\
&\leq   \frac{N}{\Abs{B'_{1-r\rho'}}\abs{B'}}\abs{ B'\cap (B'+x)}\\
&= \frac{\abs{B'}}{\Abs{B'_{1-r\rho'}}} \mu_{B'}\ast \mu_{B'}(x)\\
&\leq 2\mu_{B'}\ast \mu_{B'}(x)
\end{align*}
as required. 

The lemma now almost follows from Theorem~6.7 of \cite{BlSi:2019}, except that the left-hand side is relative to $\mu_{B'_{1-r\rho'}}\ast \mu_{B'}$ rather than the required $\mu_{B'}\ast \mu_{B'}$. To finish the proof, therefore, we note that for any function $F:G\to \bbr_{\geq 0}$, by regularity,
\begin{align*}
\frac{\abs{B'}}{\Abs{B'_{1-r\rho'}}}\Inn{F, \mu_{B'}\ast \mu_{B'}}  -  \Inn{ F, \mu_{B'_{1-r\rho'}}\ast \mu_{B'}} 
&=  \frac{N}{\Abs{B'_{1-r\rho'}}}\Inn{F\circ \mu_{B'}, \ind{B'\backslash B'_{1-r\rho'}}}\\
& \ll \frac{\Abs{B'\backslash B'_{1-r\rho'}}}{\Abs{B'_{1-r\rho'}}} \norm{F}_\infty\\
& \ll \delta^{2m}\norm{F}_\infty,
\end{align*}
provided $r\rho' \leq c'\delta^{2m}/d$ for some sufficiently small $c'>0$. Furthermore, again by regularity, $\abs{B'}\leq 2 \abs{B'_{1-r\rho'}}$, and so
\[\Inn{F, \mu_{B'}\ast \mu_{B'}}  \ll  \Inn{ F, \mu_{B'_{1-r\rho'}}\ast \mu_{B'}}+ \delta^{2m}\norm{F}_\infty\]
and hence
\[\Inn{F, \mu_{B'}\ast \mu_{B'}}^{1/2m} \ll  \Inn{ F, \mu_{B'_{1-r\rho'}}\ast \mu_{B'}}^{1/2m}+ \delta\norm{F}_\infty^{1/2m}.\]
We apply this with $F=\abs{\mu_A\ast \ind{L}\ast \mu_{B''}-\mu_A\ast \ind{L}}^{2m}$, which satisfies $\norm{F}_\infty \leq 2^{2m}$, say.

The lemma now follows, choosing $r=C\lceil \log(2/\delta \eta)\rceil$ for some sufficiently large $C>0$ and $\rho'=c'\delta^{2m}/dr$ for some sufficiently small $c'>0$.
\end{proof}

We now use Lemma~\ref{lemma:secondap} to obtain a density increment. Unlike Lemmas~\ref{lemma:L2inc} and \ref{lemma:lowdiminc}, which begin with a large $L^2$ Fourier mass, the following lemma instead uses a large $L^{2m}$ physical mass, where $m$ is large. This is a similar density increment to that used in \cite[Proposition 5.1]{BlSi:2019}, which is roughly the case $K=10$ of the following lemma. 

\begin{lemma}\label{lemma:physdi}
There is a constant $c>0$ such that the following holds.  Let $K\geq 10$ be some parameter. Let $B$ be a regular Bohr set of rank $d$, and $B'\subset B_\rho$ is a regular Bohr set, also of rank $d$, with $\rho \leq c\alpha^2/d$. Let $m\geq 1$ and suppose that $A\subset B$ has density $\alpha\leq 1/K$ such that 
\[\norm{\bal{A}{B}\circ \bal{A}{B}}_{2m(\mu_{B'}\circ \mu_{B'})}\geq K \mu(B)^{-1}.\]
Then $A$ has a density increment relative to $B'$ of strength $[K,K^{-1}\alpha^{-1};C]$ for some $C\ls_\alpha m \alpha^{-O(1/m)}$.
\end{lemma}
It may be useful to know that this lemma will be applied with $m=C_1\lceil \log(2/\alpha)\rceil$ for some large constant $C_1$, so that in particular the quantity $C$ is $\tilde{O}_\alpha(1)$. 
\begin{proof}
We first convert the balanced function $\bal{A}{B}$ into the unbalanced $\mu_A$. This is straightforward, since 
\[\bal{A}{B}\circ \bal{A}{B} = \mu_A\circ \mu_A - \mu_A\circ \mu_B-\mu_B\circ \mu_A+\mu_B\circ \mu_B\]
and, for example, 
\[\norm{\mu_A\circ \mu_B}_{2m(\mu_{B'}\circ \mu_{B'})}\leq \norm{\mu_A\circ \mu_B}_\infty \leq \mu(B)^{-1}.\]
Therefore by the triangle inequality for the $L^{2m}$ norm, and recalling that $\mu_A=\alpha^{-1}\mu(B)^{-1}\ind{A}$,
\[\norm{\mu_A\circ \ind{A}}_{2m(\mu_{B'}\circ \mu_{B'})} \geq \tfrac{1}{2}K\alpha.\]
We now apply Lemma~\ref{lemma:secondap} with $L=-A$ and $\epsilon =c'(K\alpha)^{1/2+\frac{1}{4m-2}}$ for some small constant $c'>0$ (note that our assumptions ensure that $K\alpha\leq 1$), and also $\delta=\alpha$. Let $X = \Inn{\abs{\mu_A\circ \ind{A}}^{2m}, \mu_{B'}\circ \mu_{B'}}^{1/2m}$. By the Cauchy--Schwarz inequality, the upper bound obtained from Lemma~\ref{lemma:secondap} is  at most
\[\leq\epsilon X^{1/2}+ \epsilon^{2-1/m}\norm{\mu_A\circ \ind{A}}_{1(\mu_{B'}\circ \mu_{B'})}^{1/2m}+\alpha.\]
Using the trivial estimate  
\[\norm{\mu_A\circ \ind{A}}_{1(\mu_{B'}\circ \mu_{B'})}\leq \norm{\ind{A}}_\infty \norm{\mu_A\ast \mu_{B'}\circ \mu_{B'}}_1 = 1,\]
we deduce, by the triangle inequality (and recalling that $X\geq \tfrac{1}{2}K\alpha$)
\begin{align*}
\norm{\mu_A\circ \ind{A}\ast \mu_{B''}}_{2m(\mu_{B'}\circ \mu_{B'})}
&\geq X(1-\epsilon X^{-1/2})-\epsilon^{2-1/m}-\alpha\\
&\geq \tfrac{1}{2}X-\epsilon^{2-1/m}-\alpha\\
&\geq \tfrac{1}{8}K\alpha,
\end{align*}
supposing we choose the constant in the choice of $\epsilon$ sufficiently small. The conclusion follows, since the left-hand side is at most $\norm{\ind{A}\ast \mu_{B''}}_\infty$.
\end{proof}

\section{Additive frameworks}\label{section:addframe}
For the proof of our relative structure theorem for additively non-smoothing sets we will need to perform a delicate iteration between different `scales', in some kind of structure where translations from lower scales do not perturb higher scales too much. The kind of structure that we require is captured by the following definition.
\begin{definition}[Additive framework]\label{not-frame}
An additive framework\label{def-af} $\widetilde{\Gamma}$ of height $h$ and tolerance $t$ is a collection of $h+2$ symmetric sets (in any fixed abelian group), all containing $0$, arranged so that 
\[\Gamma_{\mathrm{top}}\supset \Gamma^{(1)}\supset \Gamma^{(2)}\supset\cdots\supset \Gamma^{(h)}\supset \Gamma^{(h+1)}=\Gamma_{\mathrm{bottom}}\]
(we picture this framework vertically) such that $2\Gamma^{(1)}-2\Gamma^{(1)}\subset \Gamma_{\mathrm{top}}$, and for $1\leq i<h$,
\[t \Gamma^{(i+1)}\subset \Gamma^{(i)},\]
\[\Abs{\Gamma^{(i)}+t\Gamma^{(i+1)}}\leq 2\Abs{\Gamma^{(i)}}\]
and, for all $1\leq i\leq h$, if $x\in \Gamma^{(i+1)}-\Gamma^{(i+1)}$ then 
\[\ind{\Gamma^{(i)}}\circ \ind{\Gamma^{(i)}}(x) \geq \tfrac{1}{2} \Abs{\Gamma^{(i)}}.\]
\end{definition}
The final condition is actually slightly stronger than we need for the proof of the structural result, but producing a framework satisfying this stronger condition is slightly more natural.

To digest this definition it may help to observe that if $\Gamma$ is an additive subgroup then $\Gamma=\Gamma_{\mathrm{top}}=\cdots=\Gamma_{\mathrm{bottom}}$ span an additive framework of arbitrary height and tolerance. A less rigid example in $\bbz$ is the collection of centred arithmetic progressions $\Gamma^{(i)}=[-(2t)^{h+1-i} L,(2t)^{h+1-i} L]\cap \bbz$, which forms an additive framework of height $h$ and tolerance $t$ between $\Gamma_{\mathrm{bottom}}=[-L,L]$ and $\Gamma_{\mathrm{top}}=[-(2t)^{h+1}L,(2t)^{h+1}L]$. 

One should think of $\Gamma_{\mathrm{top}}$ and $\Gamma_{\mathrm{bottom}}$ as being two sets that are given, on which we have certain desirable properties, and an additive framework is then constructed between them; a scaffold on which the iterative structure argument will be performed. 

In this paper, the only additive frameworks required will be composed of spectra of Bohr sets. We have highlighted the particular properties required as a separate definition to show what kind of additive properties are needed for our proof of a structural result for additively non-smoothing sets, and to avoid discussing notions such as Bohr sets and spectra in its proof (which is entirely in `physical space').

The key construction of additive frameworks which we require states that there exists an additive framework of any specified height and tolerance between two spectra of two dilates of the same Bohr set, provided we have dilated by a sufficient amount.

\begin{lemma}\label{lemma:AF1}
There is a constant $c>0$ such that the following holds for any $h,t\geq 1$.

If $B$ is a regular Bohr set of rank $d$ and $\rho\leq (c/td)^{4h}$ then there exists an additive framework $\widetilde{\Gamma}$ of height $h$ and tolerance $t$ such that $\Gamma_{\mathrm{top}} = \Delta_{1/2}(B_\rho)$ and $\Gamma_{\mathrm{bottom}}=\Delta_{1/2}(B)$.
\end{lemma}

Before explaining the construction required for Lemma~\ref{lemma:AF1} we will establish a useful auxiliary lemma. In some sense, this is a Fourier analogue of the regularity of Bohr sets, and the argument is similar to that which Bourgain used in physical space. This lemma will be used repeatedly to prove Lemma~\ref{lemma:AF1}, each application building one `level' of the additive structure.

\begin{lemma}\label{lemma:AF4}
There are constants $c_1,c_2>0$ such that the following holds. Let $B$ be a regular Bohr set of rank $d$, let $t \in \bbn$ and let $\rho \leq c_1/t^2d^3$. There exist $\epsilon,\epsilon'$ satisfying $c_2/td\log(1/\rho)\leq \epsilon'\leq \epsilon$ and $\epsilon+4t\epsilon'\leq 1/2$ such that 

\[\abs{ \Delta_{1-\epsilon-4t\epsilon'}(B_\rho)}\leq 2\abs{ \Delta_{1-\epsilon}(B_\rho)}\]
and for all $0\leq j\leq 4t$
\begin{equation}\label{sumspecinc}
\Delta_{1/2}(B) + \Delta_{1-\epsilon-j\epsilon'}(B_\rho)\subset \Delta_{1-\epsilon-(j+1)\epsilon'}(B_\rho).
\end{equation}
\end{lemma}
\begin{proof}
Let $\Gamma_\nu = \Delta_{1-\nu}(B_\rho)$. We begin by establishing upper and lower bounds for the size of $\Gamma_\nu$. For an upper bound, we note that by Parseval's identity,
\[(1-\nu)^2\abs{\Gamma_\nu}\leq \sum_\gamma \abs{\widehat{\mu_{B_\rho}}(\gamma)}^2=\mu(B_\rho)^{-1},\]
whence $\abs{\Gamma_\nu}\leq 4\mu(B_\rho)^{-1}$ for all $0\leq \nu\leq 1/2$. We will now show that, for a certain range of $\nu$, we also have $\abs{\Gamma_\nu}\geq \tfrac{1}{4}\mu(B)^{-1}$. 

Let $m=\lfloor 1/C\rho d\rfloor$, where $C$ is some large absolute constant to be chosen shortly (note that $m\geq 1$ provided $\rho$ is sufficiently small). By $2m$ applications of Lemma~\ref{lemma:regConv},
\begin{align*}
\langle \mu_{B}\ast \mu_B,\mu_{B_\rho}^{(2m)}\rangle
& = (1+O(m\rho d))\mu(B)^{-1}\\
&\geq \tfrac{1}{2}\mu(B)^{-1}
\end{align*}
provided $C$ is chosen sufficiently large. By Parseval's identity and the triangle inequality, 
\[\sum\abs{\widehat{\mu_B}(\gamma)}^2\abs{\widehat{\mu_{B_\rho}}(\gamma)}^{2m}\geq \tfrac{1}{2} \mu(B)^{-1}.\]
For any $\nu \geq 0$, the contribution from $\gamma\not\in\Gamma_\nu$ is bounded above by $(1-\nu)^{2m}\mu(B)^{-1}\leq e^{-2\nu m}\mu(B)^{-1}$, which is at most $\tfrac{1}{4}\mu(B)^{-1}$ provided $\nu \geq 1/m$. Hence, for such $\nu$,
\[\abs{\Gamma_\nu}\geq \sum_{\gamma \in \Gamma_\nu}\abs{\widehat{\mu_B}(\gamma)}^2\abs{\widehat{\mu_{B_\rho}}(\gamma)}^{2m}\geq \tfrac{1}{4}\mu(B)^{-1}.\]
It follows that if $\nu\geq 2C\rho d$, say, then $\abs{\Gamma_\nu}\geq \tfrac{1}{4}\mu(B)^{-1}$.

We have, for any $K\geq 2$, and $\tfrac{1}{2K}\geq\nu\geq 2C\rho d$,
\[\prod_{i=1}^{K-1}\frac{\abs{\Gamma_{i\nu}}}{\abs{\Gamma_{(i+1)\nu}}}=\frac{\abs{\Gamma_\nu}}{\abs{\Gamma_{K\nu}}}\geq 2^{-4} \mu(B)^{-1}\mu(B_\rho)\geq 2^{-2d-4}\rho^d\]
using the fact that $\mu_B(B_\rho)\geq (\rho/4)^d$, as given by Lemma~\ref{lemma:bohrsiz}. In particular, there exists some $1\leq i<K$ such that 
\[\frac{\abs{\Gamma_{i\nu}}}{\abs{\Gamma_{(i+1)\nu}}}\geq (2^{-6}\rho)^{d/(K-1)}.\]
Thus, if we choose $K=10\lceil d\log(2^6/\rho)\rceil$ then, for any $\nu$ as above, there is some $1\leq i< K$ such that $\abs{\Gamma_{i\nu}}\geq \tfrac{1}{2}\abs{\Gamma_{(i+1)\nu}}$. 

We now pick $\nu = 1/2K$, take the corresponding $i$ and write $\epsilon = i\nu$, $\epsilon' = \nu/4t$. Note that we may choose the constant $c_1$ in the statement of the lemma small enough to ensure that $1/2K \geq 2C\rho d$, so that this choice of $\nu$ is valid. The first part of the conclusion is then immediate. For \eqref{sumspecinc}, by Lemma~\ref{lemma:bohrspectra} there exists an absolute constant $C'>0$ such that
\[\Delta_{1/2}(B)+\Delta_{1-\delta}(B_\rho)\subset \Delta_{1-(\delta+C'\rho d)}(B_\rho)\]
for any $\delta \geq 0$. We choose $\delta=\epsilon+j\epsilon'$, and choosing $c_1$ a small enough constant, we have $\epsilon' \geq C'\rho d$, and thus property \eqref{sumspecinc} holds.
\end{proof}

We are now able to construct the additive framework promised by Lemma~\ref{lemma:AF1}.

\begin{proof}[Proof of Lemma~\ref{lemma:AF1}]
We begin at the bottom level, with $\Gamma_{\mathrm{bottom}}=\Gamma^{(h+1)}= \Delta_{1/2}(B)$. We then build up the additive framework one level at a time: each $\Gamma^{(h-i+1)}$ will be of the form $\Delta_{1-\epsilon_i}(B^{(i)})$ where $B^{(i)}=B_{\rho_i}$ for some suitably chosen $\epsilon_i \leq 1/2$ and $\rho_i$, for which $B^{(i)}$ is regular.

We have already chosen $\Gamma^{(h+1)}$. Suppose in general that, with $0\leq i<h$, we have constructed $\Gamma^{(h-i+1)}= \Delta_{1-\epsilon_i}(B^{(i)})$. We will now construct $\Gamma^{(h-i)}$. Applying Lemma \ref{lemma:AF4} to $B^{(i)}$, with parameters $t$ and $\rho = c_3/t^3d^3$, we let $\epsilon_{i+1} = \epsilon + 2\epsilon'$, these values being as given by the conclusion of that lemma. We pick the constant $c_3$ here so that, in particular, the hypothesis of Lemma~\ref{lemma:AF4} is satisfied, and so that $B^{(i+1)} = B^{(i)}_\rho = B_{\rho_{i+1}}$ is regular. Note that $\rho_{i+1} = c_3\rho_i/t^3d^3$. It remains to verify that 
\[ \Gamma^{(h-i)} = \Delta_{1-\epsilon_{i+1}}(B^{(i+1)}) \]
satisfies the requirements for the next level of the additive framework, namely that
\[ t \Gamma^{(h-i+1)} \subset \Gamma^{(h-i)}, \]
\[\Abs{\Gamma^{(h-i)}+t\Gamma^{(h-i+1)}}\leq 2\Abs{\Gamma^{(h-i)}},\]
and that, for all $x\in \Gamma^{(h-i+1)}-\Gamma^{(h-i+1)}$,
\[\ind{\Gamma^{(h-i)}}\circ \ind{\Gamma^{(h-i)}}(x) \geq \tfrac{1}{2}\Abs{\Gamma^{(h-i)}}.\]
We begin with the first property. Here we have, by Lemma \ref{lemma:bohrspectra},
\[ t \Gamma^{(h-i+1)} \subset t \Delta_{1/2}(B^{(i)}) \subset \Delta_{1-C\rho d t}(B^{(i)}_\rho) \]
for some absolute constant $C>0$. Provided $C\rho d t \leq \epsilon_{i+1}$, this is contained in $\Gamma^{(h-i)}$, and our choice of $\rho$ ensures this for a small enough constant $c_3$.

For the second property, we have, by the conclusion \eqref{sumspecinc} of Lemma \ref{lemma:AF4},
\begin{align*}
\Gamma^{(h-i)} + t\Gamma^{(h-i+1)} &\subset t\Delta_{1/2}(B^{(i)}) + \Delta_{1-\epsilon-2\epsilon'}(B^{(i+1)}) \\
&\subset (t-1)\Delta_{1/2}(B^{(i)}) + \Delta_{1-\epsilon-3\epsilon'}(B^{(i+1)}) \\
&\subset \cdots \\
&\subset \Delta_{1-\epsilon-(t+2)\epsilon'}(B^{(i+1)}).
\end{align*}
Since this is contained in $\Delta_{1-\epsilon - 4t\epsilon'}(B^{(i+1)})$, Lemma \ref{lemma:AF4} ensures that it has size at most $2 \Abs{\Gamma^{(h-i)}}$.

For the third property, note that if $x\in \Gamma^{(h-i+1)} - \Gamma^{(h-i+1)} \subset 2\Delta_{1/2}(B^{(i)})$ then, by Lemma \ref{lemma:AF4},
\[ x + \Delta_{1-\epsilon}(B^{(i+1)}) \subset \Delta_{1-\epsilon-2\epsilon'}(B^{(i+1)}) = \Gamma^{(h-i)}, \]
whence
\begin{align*}
\tfrac{1}{2}\Abs{\Delta_{1-\epsilon-2\epsilon'}(B^{(i+1)})}
&\leq \Abs{\Delta_{1-\epsilon}(B^{(i+1)})}\\
&\leq \ind{\Delta_{1-\epsilon}(B^{(i+1)})}\circ \ind{\Gamma^{(h-i)}}(x)\\
&\leq \ind{\Gamma^{(h-i)}}\circ \ind{\Gamma^{(h-i)}}(x),
\end{align*}
and hence $\ind{\Gamma^{(h-i)}}\circ \ind{\Gamma^{(h-i)}}(x) \geq \tfrac{1}{2}\Abs{\Gamma^{(h-i)}}$ as required.

We continue this procedure until we have built $h$ levels of the framework. It remains to check that $2\Gamma^{(1)}-2\Gamma^{(1)}\subset \Gamma_{\mathrm{top}}$. We note that $\Gamma^{(1)}\subset \Delta_{1/2}(B^{(h)})$ where $B^{(h)}=B_{\rho_h}$ and $\rho_h \geq (c/t^3d^3)^h$ for some absolute constant $c>0$. The required inclusion therefore follows from Lemma~\ref{lemma:bohrspectra}. 
\end{proof}

\section{Additive properties of spectra and symmetry sets}\label{section:addprop}

In this section we introduce various useful measures of additive structure, and examine how they behave in both spectra and so-called `symmetry sets', which are sets of large values of convolutions.

\subsection*{Orthogonality} The first concept that we require is an appropriate notion of orthogonality. For our purposes, a relatively crude notion will suffice.

\begin{definition}[Orthogonality]
A set $\Delta$ is $\Gamma$-orthogonal\label{def-orth} if the translates $(\gamma+\Gamma)_{\gamma\in \Delta}$ are all disjoint. 
\end{definition}

The following trivial lemma will be used frequently in what follows.
\begin{lemma}\label{lemma:trivorth}
If $\Delta'\subset \Delta$ is a maximal $\Gamma$-orthogonal subset then $\Delta\subset \Delta'+\Gamma-\Gamma$.
\end{lemma}
\begin{proof}
If $\gamma\in \Delta\backslash \Delta'$ then there must exist some $\gamma'\in \Delta'$ such that $(\gamma+\Gamma)\cap(\gamma'+\Gamma)\neq\emptyset$, and hence $\gamma\in \gamma'+\Gamma-\Gamma$.
\end{proof}

We will show that spectra cannot contain orthogonal sets which are too large. This result is related to Bessel's inequality, and is similar to \cite[Lemma 4.2]{Sa:2011}, although we use a simpler definition of orthogonality.

\begin{lemma}\label{lemma:orthbessel}
There is a constant $c>0$ such that the following holds. Let $B$ be a regular Bohr set of rank $d$ and suppose $A\subset B$ has density $\alpha$. Let $B'\subset B_\rho$ for some $\rho \leq c/d\log(2/\alpha\eta)$. If $\Delta\subset \Delta_\eta(A)$ is $\Delta_{1/2}(B')$-orthogonal then 
\[\abs{\Delta}\ll \eta^{-2}\alpha^{-1}.\]
\end{lemma}

\begin{proof}
By definition of the spectrum,
\[\eta \alpha\mu(B)\abs{\Delta}\leq \sum_{\Delta}\Abs{\widehat{\ind{A}}(\lambda)}=\bbe_x\ind{A}(x)\sum_{\lambda\in\Delta} c_\lambda \overline{\lambda(x)}\]
for some choice of signs $c_\lambda\in\bbc$. By the Cauchy--Schwarz inequality
\[\eta^2\alpha^2\mu(B)^2\abs{\Delta}^2\leq \alpha\mu(B)\bbe_x\ind{B}(x)\sum_{\lambda_1,\lambda_2\in\Delta}\overline{c_{\lambda_1}}c_{\lambda_2}(\lambda_1-\lambda_2)(x),\]
and so
\[\alpha\eta^2\abs{\Delta}^2 \leq \bbe_x \mu_B(x)\sum_{\lambda_1,\lambda_2\in\Delta}\overline{c_{\lambda_1}}c_{\lambda_2}(\lambda_1-\lambda_2)(x).\]
By Lemma~\ref{lemma:fourierbohr}, if $L$ is some parameter to be chosen later, and $\mu=\mu_{B_{1+L\rho}}\ast \mu_{B'}^{(L)}$, where $B'\subset B_{c'/Ld}$ for some sufficiently small absolute constant $c'>0$ (which will be guaranteed by our upper bound on $\rho$), then $\mu_B\leq 2\mu$, and so by the triangle inequality, 
\[\tfrac{1}{2}\alpha\eta^2\abs{\Delta}^2 \leq \sum_{\lambda_1,\lambda_2\in \Delta}\abs{\widehat{\mu}(\lambda_1-\lambda_2)}.\]
Since $\abs{\widehat{\mu}}\leq \abs{\widehat{\mu_{B'}}}^L$ the contribution from $\lambda_1-\lambda_2\not\in \Delta_{1/2}(B')$ is negligible provided $L=C\lceil \log (2/\alpha \eta)\rceil$, for some large constant $C>0$, and hence
\[\tfrac{1}{4}\alpha\eta^2 \abs{\Delta}^2\leq \Inn{\ind{\Delta}\circ \ind{\Delta},\ind{\Delta_{1/2}(B')}}.\]
By orthogonality, the right-hand side is $\abs{\Delta}$, and the lemma follows.
\end{proof}

\subsection*{Additive energy}
The notion of additive energy is ubiquitous in additive combinatorics, as it offers a `smooth' way to measure additive structure, particularly amenable to analytic techniques. The classical definition is
\[E_4(\Delta) = \{ (\gamma_1,\gamma_2,\gamma_3,\gamma_4)\in \Delta^4 : \gamma_1+\gamma_2=\gamma_3+\gamma_4\}.\]
We need to generalise this definition in two ways, extending the number of summands involved, and weakening the notion of equality.

\begin{definition}[Relative additive energy]
Let $m\geq 1$ be any integer and $\nu:\widehat{G}\to\bbc$. For any function $\omega:\widehat{G}\to \bbc$ we define the $2m$-fold additive energy with respect to $\nu$ as\label{not-energy}
\[
E_{2m}(\omega;\nu) = \sum_{\gamma_1,\ldots,\gamma_m,\gamma_1',\ldots,\gamma_m'}\omega(\gamma_1)\cdots \omega(\gamma_m) \overline{\omega(\gamma_1')\cdots \omega(\gamma_m')}\nu(\gamma_1+\cdots -\gamma_m').\]
\end{definition}
\noindent 
Note that by the triangle inequality
\[ \abs{E_{2m}(\omega; \nu)} \leq E_{2m}(\abs{\omega}; \abs{\nu}), \]
and that
\[ E_{2m}(\omega; \nu) = \bbe_x \widecheck{\nu}(x) \abs{\widecheck{\omega}(x)}^{2m}. \]
If $\Delta,\Gamma\subset\widehat{G}$ then we write $E_{2m}(\Delta;\Gamma)=E_{2m}(\ind{\Delta};\ind{\Gamma})$. For example, when $\Gamma=\{0\}$, this is just the conventional higher additive energy, which counts the number of $(\gamma_1,\ldots,\gamma_m')\in \Delta^{2m}$ such that $\gamma_1+\cdots+\gamma_m = \gamma_1'+\cdots +\gamma_m'$. 

We now introduce a notion of dissociativity. Intuitively, a dissociated set is one which has no non-trivial additive relations between its elements. There are many different ways to make this precise. The following (rather unusual) definition of dissociativity is taken from \cite{Bl:2016}. The more usual definition is that $\Delta$ is dissociated when $\sum_{\gamma\in \Delta}c_\gamma\gamma=0$ with $c_\gamma\in \{-1,0,1\}$ if and only if all $c_\gamma=0$. We firstly need to make this definition relative to some $\Gamma$. The most natural way to do this would be to replace $=0$ with $\in \Gamma$. This does not seem to be enough, however, as our main tool used for studying dissociated sets will be the additive energy, and it is still possible for such a set to have a large additive energy. 

Some experimentation results in the following, stronger, definition, which also allows us to control the additive energy. There are alternative ways one could proceed (most notably through a more analytic argument invoking Rudin's inequality), but this path seems the most straightforward.

\begin{definition}[Dissociativity and dimension]\label{acdef-diss}
We say that $\Lambda$ is $\Gamma$-dissociated\label{def-diss} if for all $k\geq 1$ and $\gamma\in\widehat{G}$ there are at most $2^k$ many pairs $(\Lambda_1,\Lambda_2)$ of disjoint subsets of $\Lambda$ with $\abs{\Lambda_1\cup \Lambda_2}=k$ such that
\[\sum_{\lambda\in \Lambda_1}\lambda-\sum_{\lambda'\in\Lambda_2}\lambda'\in \Gamma+\gamma.\]
We say that $\Delta$ has $\Gamma$-dimension\label{not-dim} at most $d$ (also written $\dim(\Delta;\Gamma)\leq d$) if every $\Gamma$-dissociated subset of $\Delta$ has size at most $d$.
\end{definition}

We remark that for the applications in the present paper (when $G$ is a finite abelian group of odd order) the following, slightly stronger, notion of relative dissociativity would suffice: `for each $\gamma\in \widehat{G}$ there is at most one $\epsilon_\lambda\in \{-1,0,1\}^\Lambda$ such that $\sum_{\lambda\in \Lambda}\epsilon_\lambda \lambda \in \Gamma+\gamma$'. This is certainly more natural, and is much closer to the usual notions of dissociativity found in the literature. Definition~\ref{acdef-diss} is that used in \cite{Bl:2016}, where it is introduced to address subtle issues concerning the case when $G$ has even order, which do not concern us here. We have kept Definition~\ref{acdef-diss} as is, however, so that we can more easily import some of the following results from \cite{Bl:2016}. 

We will use the following crude bound on the dimension of unions often.
\begin{lemma}\label{lemma:dimunion}
For any $\Delta_1$, $\Delta_2$, and $\Gamma$, 
\[\dim(\Delta_1\cup \Delta_2; \Gamma) \leq \dim(\Delta_1;\Gamma)+\dim(\Delta_2;\Gamma).\]
\end{lemma}
\begin{proof}
This is immediate, since if $\Lambda\subset \Delta_1\cup \Delta_2$ is $\Gamma$-dissociated then so are both $\Lambda\cap \Delta_1$ and $\Lambda\cap \Delta_2$.
\end{proof}

The two important properties that we require are that dissociated sets have small (almost minimal) additive energy, and that sets with small dimension are efficiently covered. The first property is an immediate corollary of Lemmas 3.1 and 3.2 from \cite{Bl:2016}, which yield the following.
\begin{lemma}\label{lemma:dimenergy}
If $\Lambda$ is $\Gamma$-dissociated then for any $m\geq 2$
\[E_{2m}(\Lambda;\Gamma)\leq 2^{7m}(m+1)!\abs{\Lambda}^m.\]
\end{lemma}

The most important feature of dimension is that it allows us to create a small spanning set (relative to $\Gamma$). Recall that, by Definition~\ref{acdef-cov}, $\Delta$ being $k$-covered by $\Gamma$ means that there exists some $\Lambda$ of size $\abs{\Lambda}\leq k$ such that $\Delta\subset \langle \Lambda\rangle +\Gamma-\Gamma$.

\begin{lemma}\label{lemma:dimcovering}
If $\Delta$ has $\Gamma$-dimension at most $d$ then $\Delta$ is $2d$-covered by $\Gamma$. In particular, this implies that any translate of $\Delta$ is $(2d+1)$-covered by $\Gamma$.
\end{lemma}
\begin{proof}
Let $\Lambda\subset \Delta$ be a maximal $\Gamma$-dissociated subset, so that $\abs{\Lambda}\leq d$. Suppose $\gamma\in \Delta\backslash \Lambda$. Since $\Lambda\cup\{\gamma\}$ is not $\Gamma$-dissociated there exists 	$k\geq 1$ and $\lambda\in\widehat{G}$ such that there are more than $2^k$ many triples $(s,\Delta_1',\Delta_2')$ such that $s\in \{-1,0,1\}$, the sets $\Delta_1',\Delta_2'\subset \Lambda$ are disjoint, with $\abs{\Delta_1'}+\abs{\Delta_2'}+\abs{s}=k$, and further
\[s\gamma+\sum_{\gamma_1'\in \Delta_1'}\gamma_1'-\sum_{\gamma_2'\in \Delta_2'}\gamma_2'\in \Gamma+\lambda.\]
If there exists at least one such triple with $s=0$ and at least one with $s\neq 0$ then $\gamma\in \langle \Lambda\rangle-\langle \Lambda\rangle+\Gamma-\Gamma$, and the conclusion follows. If $s=0$ for all such triples then this contradicts the $\Gamma$-dissociativity of $\Lambda$.

Suppose finally that $s\in \{-1,1\}$ for all such triples. This is impossible for $k=1$, and for $k>1$ by the pigeonhole principle there are strictly more than $2^{k-1}$ many triples with identical $s$. This is another contradiction to $\Gamma$-dissociativity, considering the translate $\Gamma+\lambda-s\gamma$. 

Finally, if $\Delta\subset \langle \Lambda\rangle+\Gamma-\Gamma$ then, for any $\gamma$, $\Delta+\gamma\subset \langle \Lambda\cup\{\gamma\}\rangle+\Gamma-\Gamma$.
\end{proof}

The following result is an immediate corollary of Theorem 3.1 in \cite{Bl:2016}, and is our main direct link between energies and sets with small dimension.

\begin{lemma}\label{lemma:energytodimension}
Let $\Gamma\subset\widehat{G}$ be a symmetric set and $\omega:\widehat{G}\to[0,1]$. If $\ell,m\geq 2$ are any integers such that $\ell\geq 4m$ then either
\begin{enumerate}
\item there is $\Delta\subset \widehat{G}$ such that 
\[\sum_{\gamma\in \Delta}\omega(\gamma)\geq \min\brac{1,\frac{\norm{\omega}_1}{\ell}}\frac{m}{2\ell}\norm{\omega}_1\]
and $\dim(\Delta;\Gamma)\ll \ell$, or
\item $E_{2m}(\omega;\Gamma)\leq (Cm\ell^{-1})^{2m}\norm{\omega}_1^{2m}$ for some constant $C>0$.
\end{enumerate}
\end{lemma}
\begin{proof}
Suppose first that  $\norm{\omega}_2^2\leq m \ell^{-2}\norm{\omega}_1^2$. In this case, the lemma follows immediately from Theorem 3.1 from \cite{Bl:2016} (which is stated in terms of covering rather than dimension, but the proof also gives the dimension bound) with $d=\ell$ and $n=m$.

Suppose now, on the other hand, that $\norm{\omega}_2^2> m \ell^{-2}\norm{\omega}_1^2$. We first note that if there are at least $\ell$ distinct $\gamma$ with $\omega(\gamma)>\tfrac{1}{2}m\ell^{-2}\norm{\omega}_1$ then, letting $\Delta$ be an arbitrary collection of any such $\ell$, we have
\[\sum_{\gamma\in \Delta}\omega(\gamma) > \frac{m}{2\ell}\norm{\omega}_1\]
and the first case of the lemma holds. Otherwise, we note that the contribution to $\norm{\omega}_2^2$ from those $\gamma$ such that $\omega(\gamma)\leq \tfrac{1}{2}m\ell^{-2}\norm{\omega}_1$ is trivially at most $\tfrac{1}{2}m\ell^{-2}\norm{\omega}_1^2$. Therefore, if $\Delta$ is the set of all $\gamma$ with $\omega(\gamma)>\tfrac{1}{2}m\ell^{-2}\norm{\omega}_1$ (which has $\dim(\Delta;\Gamma)\leq \abs{\Delta}\leq \ell$ trivially),
\[\sum_{\gamma\in\Delta}\omega(\gamma)\geq \sum_{\gamma\in\Delta}\omega(\gamma)^2> \frac{m}{2\ell^2}\norm{\omega}_1^2,\]
and we are again in the first case of the lemma. 
\end{proof}

\subsection*{Dimension of spectra}
We shall couple this with the following lemma, which says that the relative energy of a weight function on a spectrum is always quite large. That such a general lower bound for the energy of subsets of a spectrum is possible was first observed by Shkredov in \cite{Sh:2008}. The reader should interpret the following lemma with $	f=\ind{A}$ for some $A\subset B$ and $\omega=\ind{\Delta}$ for some $\Delta\subset \Delta_\eta(A)$. 

\begin{lemma}\label{lemma:energylower}
There is a constant $c>0$ such that the following holds. Let $B$ be a regular Bohr set of rank $d$ and let $f:B\to\bbc$ be supported on $B$ with $\alpha=\norm{f}_1^2\norm{f}_2^{-2}\mu(B)^{-1}$. Let $m\geq 1$ be some parameter and suppose $B' \subset B_\rho$ for some $\rho \leq c(md\log(2/\alpha\eta))^{-1}$. For any $\omega:\Delta_\eta(f)\to \bbr_{\geq 0}$ we have the lower bound 
\[E_{2m}(\omega; \Delta_{1/2}(B'))\gg \eta^{2m}\alpha \norm{\omega}_1^{2m}.\]
\end{lemma}
\begin{proof}
Without loss of generality, suppose that $\norm{f}_1=1$. 
By the definition of the spectrum 
\[\eta\norm{\omega}_1\leq \sum \omega(\gamma) \abs{\widehat{f}(\gamma)}=\bbe_x f(x)\sum c_\gamma\omega(\gamma)\overline{\gamma(x)}\]
for some choice of signs $c_\gamma\in\bbc$. H\"{o}lder's inequality then implies that 
\[(\eta\norm{\omega}_1)^m\leq \bbe_x \abs{f(x)}\abs{\sum c_\gamma\omega(\gamma)\overline{\gamma(x)}}^m,\]
and the Cauchy--Schwarz inequality then implies that
\[(\eta\norm{\omega}_1)^{2m}\leq \mu(B)\norm{f}_2^2\bbe_x \mu_B(x)\abs{\sum_\gamma c_\gamma\omega(\gamma)\overline{\gamma(x)}}^{2m}.\]
If we let $\mu=\mu_{B_{1+L\rho}}\ast \mu_{B'}^{(L)}$ then by Lemma~\ref{lemma:fourierbohr} we have $\mu_B\leq 2\mu$, provided $\rho \leq c/Ld$ for some sufficiently small $c>0$, and so by the triangle inequality, 
\begin{align*}
\eta^{2m}\norm{\omega}_1^{2m}\mu(B)^{-1}\norm{f}_2^{-2}
&\leq 2E_{2m}(\omega; \abs{\widehat{\mu}})\\
&\leq 2E_{2m}(\omega; \abs{\widehat{\mu_{B'}}}^L).
\end{align*}
Since $\norm{\omega^{(m)}\circ \omega^{(m)}}_1=\norm{\omega}_1^{2m}$ the contribution to the energy where $\abs{\widehat{\mu_{B'}}}<1/2$ is negligible, provided we choose $L$ to be a sufficiently large constant multiple of $m\lceil\log(2/\alpha\eta)\rceil$. 
\end{proof}

We note that replacing Lemma 4.1 of \cite{Bl:2016} with this more efficient lemma allows the method of \cite{Bl:2016} to show that if $A\subset \{1,\ldots,N\}$ has no non-trivial three-term arithmetic progressions then $\abs{A}\ll \frac{(\log\log N)^3}{\log N}N$, already an improvement on the $\frac{(\log\log N)^4}{\log N}N$ bound presented there.

Combining the previous two lemmas implies immediately that spectra contain large subsets with relatively low dimension, a fact that was the driving force behind \cite{Bl:2016}.

\begin{corollary}\label{cor:massdim}
There exists a constant $c>0$ such that the following holds. Let $B$ be a regular Bohr set of rank $d$ and let $f:B\to\bbc$ be supported on $B$ with $\alpha=\norm{f}_1^2\norm{f}_2^{-2}\mu(B)^{-1}$. Let $\eta \in (0,1]$ and $B' \subset B_\rho$ where $\rho \leq c/d\log^2(2/\alpha\eta)$. For any $\omega: \Delta_\eta(f)\to [0,1]$ there is a set $\Delta$ such that
\[\sum_{\gamma\in\Delta}\omega(\gamma)\gs_\alpha \min\brac{1,\eta\norm{\omega}_1}\eta \norm{\omega}_1\quad\textrm{and}\quad \dim(\Delta; \Delta_{1/2}(B'))\lesssim_\alpha \eta^{-1}. \]
\end{corollary}
\begin{proof}
Fix $m = C_1\lceil\log(2/\alpha)\rceil$ and $\ell = C_2m\lceil \eta^{-1}\rceil$, with precise constants to be determined later, but certainly picked so that $\ell \geq 4m \geq 8$. Applying Lemma \ref{lemma:energylower} to $\omega$, we see that 
\[ E_{2m}(\omega;  \Delta_{1/2}(B')) \gg \eta^{2m} \alpha \norm{\omega}_1^{2m}. \]
For suitable values of the constants in the definitions of $m$ and $\ell$, this ensures that we cannot be in case (2) of Lemma \ref{lemma:energytodimension}, and hence we are in case (1), as claimed.
\end{proof}

One can increase the $\ell^1$ mass of $\omega$ on $\Delta$ at the expense of dimension by a crude `remove and repeat' procedure as follows.
\begin{corollary}\label{cor:massdimboot}
There is a constant $c>0$ such that the following holds. Let $B$ be a regular Bohr set of rank $d$ and let $f:B\to\bbc$ be supported on $B$ with $\alpha=\norm{f}_1^2\norm{f}_2^{-2}\mu(B)^{-1}$. Let $\eta \in (0,1]$ and let $B' \subset B_\rho$ where $\rho \leq c/d\log^2(2/\alpha\eta)$. For any $0\leq \delta \leq 1/2$ and any $\omega: \Delta_\eta(f)\to [0,1]$ there is a set $\Delta$ such that
\[\sum_{\gamma\in\Delta}\omega(\gamma)\geq \delta \norm{\omega}_1\]
and
\[\dim(\Delta; \Delta_{1/2}(B'))\lesssim_\alpha \max\brac{1,\max\brac{1,\eta^{-1}\norm{\omega}_1^{-1}}\delta\eta^{-1}}\eta^{-1}. \]
\end{corollary}
\begin{proof}
We iteratively apply the previous corollary to produce a sequence of sets $\Delta_1, \Delta_2, \ldots$ whose union will be the set $\Delta$ of the conclusion. Write $\Gamma = \Delta_{1/2}(B')$ throughout. At the first stage, apply Corollary \ref{cor:massdim} to $\omega$ to obtain a set $\Delta_1$ such that
\[ \sum_{\gamma\in\Delta_1} \omega(\gamma) \gs_\alpha \min\brac{1,\eta\norm{\omega}_1} \eta \norm{\omega}_1 \quad\textrm{and}\quad \dim(\Delta_1; \Gamma)\lesssim_\alpha \eta^{-1}. \]
If $\sum_{\gamma\in\Delta_1} \omega(\gamma) \geq  \delta\norm{\omega}_1$ we are done, so assume that
\[ \sum_{\gamma\in\Delta_1} \omega(\gamma) < \delta\norm{\omega}_1 \leq \tfrac{1}{2}\norm{\omega}_1. \]
Applying Corollary \ref{cor:massdim} to $\omega \cdot 1_{\Delta_1^c}$, noting that $\Norm{\omega \cdot 1_{\Delta_1^c}}_1 \geq \tfrac{1}{2}\norm{\omega}_1$, we obtain a set $\Delta_2 \subset \widehat{G}\setminus \Delta_1$ such that
\[ \sum_{\gamma\in\Delta_2} \omega(\gamma) \gs_\alpha \min\brac{1,\eta\norm{\omega}_1}\eta \norm{\omega}_1 \quad\textrm{and}\quad \dim(\Delta_2; \Gamma)\lesssim_\alpha \eta^{-1}. \]
If $\Delta_1 \cup \Delta_2$ has
\[ \sum_{\gamma\in\Delta_1 \cup \Delta_2} \omega(\gamma) \geq \delta\norm{\omega}_1 \]
then we are done; otherwise we repeat the argument with $\omega \cdot 1_{(\Delta_1 \cup \Delta_2)^c}$ (whose $\ell^1$-norm is then still at least $\tfrac{1}{2}\norm{\omega}_1$). Carrying on in this way, we obtain at each stage a set $\Delta_{i+1} \subset (\Delta_1 \cup \cdots \cup \Delta_i)^c$ with
\[ \sum_{\gamma\in\Delta_{i+1}} \omega(\gamma) \gs_\alpha \min\brac{1,\eta\norm{\omega}_1}\eta \norm{\omega}_1 \quad\textrm{and}\quad \dim(\Delta_{i+1}; \Gamma)\lesssim_\alpha \eta^{-1}, \]
and we halt the iteration as soon as
\[ \sum_{\gamma\in\Delta_1 \cup\cdots \cup \Delta_k} \omega(\gamma) \geq \delta\norm{\omega}_1. \]
At this point, by Lemma~\ref{lemma:dimunion},
\[ \dim(\Delta_1 \cup \cdots \cup \Delta_k; \Gamma) \leq \dim(\Delta_1; \Gamma) + \cdots + \dim(\Delta_k; \Gamma) \ls_\alpha k \eta^{-1}. \]
Since $\sum_{\gamma\in\Delta_i} \omega(\gamma) \gs_\alpha \min\brac{1,\eta\norm{\omega}_1}\eta \norm{\omega}_1$ for each $i$, the iteration must halt after at most $k \ls_\alpha\max(1, \max(1,\eta^{-1}\norm{\omega}_1^{-1})\delta\eta^{-1})$ steps, and so we are done.
\end{proof}


\subsection*{Dimension of symmetry sets}

We will also need an upper bound for the dimension of symmetry sets, which are sets of large values of convolutions. A result of this type was proved by Shkredov and Yekhanin in \cite{ShYe:2011} by an ingenious combinatorial argument. Their method does not seem to adapt well to the notion of dimension that we use here, and we also need to work with `relative' symmetry sets, and so we employ an alternative argument using H\"{o}lder's inequality, similar to the argument used to prove Lemma~\ref{lemma:energylower}.

This lemma will be applied with $X$ a subset of $\Delta$ coming from the structural theorem for additively non-smoothing set. 

\begin{lemma}\label{lemma:dimsymmetry}
There is a constant $c>0$ such that the following holds. Let $B$ be a regular Bohr set of rank $d$. Let $\delta>0$ be some parameter, and suppose $B'\subset B_\rho$ is a regular Bohr set where $\rho \leq c(d\log^2(2\abs{X}/\delta))^{-1}$. Suppose that $X\subset \widehat{G}$ is a set with
\[\norm{\ind{X}\circ \abs{\widehat{\mu_{B'}}}^2}_\infty\leq 2.\] 
The set
\[S = \{ \gamma : \ind{X}\circ \ind{X}\circ\ind{\Delta_{1/2}(B)}(\gamma) \geq \delta \abs{X}\}\]
has $\Delta_{1/2}(B')$-dimension $O(\delta^{-2}\log \abs{X})$.
\end{lemma}

It may be possible to improve this bound to $O(\delta^{-1}\log \abs{X})$, which is the bound obtained in \cite{ShYe:2011} in the non-relative case, but this would have little effect on our final result, since we will apply this lemma only in the regime where $\delta \gs_\alpha 1$.

\begin{proof}
For brevity, let $\Gamma=\Delta_{1/2}(B)$ and $\Gamma'=\Delta_{1/2}(B')$. Let $\Lambda\subset S$ be a maximal $\Gamma'$-dissociated subset -- we will show that $\abs{\Lambda} \ll \delta^{-2}\log\abs{X}$. 

We begin by noting 
\[\delta\abs{X}\abs{\Lambda} \leq \Inn{\ind{\Lambda}, \ind{X}\circ \ind{X}\circ \ind{\Gamma}}.\]
Since $\tfrac{1}{4}\ind{\Gamma}\leq \abs{\widehat{\mu_B}}^2$, by writing the right-hand side in physical space and applying the triangle inequality, we deduce that 
\[\tfrac{1}{4}\delta \abs{X}\abs{\Lambda}\leq \bbe_{x}\mu_B\circ \mu_B(x)\Abs{\widecheck{\ind{\Lambda}}(x)\widecheck{\ind{X}}(x)^2}.\]
In particular, applying H\"{o}lder's inequality, for any $m\geq 1$,
\[\brac{ \bbe_{x}\mu_B\circ \mu_B(x)\abs{\widecheck{\ind{\Lambda}}(x)}^{2m}\abs{\widecheck{\ind{X}}(x)}^2} \brac{ \bbe_{x}\mu_B\circ \mu_B(x)\abs{\widecheck{\ind{X}}(x)}^2}^{2m-1}\]
is at least $(\tfrac{1}{4}\delta \abs{X}\abs{\Lambda})^{2m}$. 

By Lemma~\ref{lemma:fourierbohr} if we let $\mu=\mu_{B_{1+\rho}}\ast \mu_{B'}$ then $\mu_B\leq 2\mu$, provided $\rho \leq c/d$ for some sufficiently small constant $c>0$, and hence, 
\begin{align*}
\bbe_{x}\mu_B\circ \mu_B(x)\Abs{\widecheck{\ind{X}}(x)}^2
&\leq 4\bbe_{x}\mu\circ \mu(x)\Abs{\widecheck{\ind{X}}(x)}^2\\
&= 4\Inn{\ind{X}\circ \ind{X}, \abs{\widehat{\mu_{B_{1+L\rho}}}}^2\abs{\widehat{\mu_{B'}}}^{2}}\\
&\leq 4\Inn{\ind{X}\circ \ind{X}, \abs{\widehat{\mu_{B'}}}^2}\\
&\leq 8\abs{X}.
\end{align*}
Trivially bounding $\Abs{\widecheck{\ind{X}}}\leq \abs{X}$ we therefore obtain the lower bound 
\[\abs{X}^{-1}(\tfrac{1}{32}\delta \abs{\Lambda})^{2m}\leq \bbe_{x}\mu_B\circ \mu_B(x)\Abs{\widecheck{\ind{\Lambda}}(x)}^{2m}=\Inn{\ind{\Lambda}^{(m)}\circ \ind{\Lambda}^{(m)}, \abs{\widehat{\mu_B}}^2}.\]
By Lemma~\ref{lemma:fourierbohr} if we let $\mu=\mu_{B_{1+L\rho}}\ast \mu_{B'}^{(L)}$ then $\mu_B\leq 2\mu$, provided $\rho \leq c/Ld$ for some sufficiently small constant $c>0$, and hence
\[\abs{X}^{-1}(\tfrac{1}{32}\delta \abs{\Lambda})^{2m}\ll \Inn{\ind{\Lambda}^{(m)}\circ \ind{\Lambda}^{(m)}, \abs{\widehat{\mu_{B'}}}^{2L}}.\]
We can discard the contribution to this inner product from where $\abs{\widehat{\mu_{B'}}}< 1/2$, provided we choose $L \geq C(m\log(2/\delta)+\log \abs{X})$ for some large $C>0$, so that
\[E_{2m}(\Lambda;\Gamma') \geq \frac{1}{2\abs{X}}(\tfrac{1}{32}\delta\abs{\Lambda})^{2m}.\]
By the $\Gamma'$-dissociativity of $\Lambda$ and Lemma~\ref{lemma:dimenergy}, however, the left-hand side is at most $(C'm\abs{\Lambda})^m$ for some large constant $C'>0$. If we choose $m=\lceil C\log \abs{X}\rceil $ for some sufficiently large $C$, then this in particular implies that $\abs{\Lambda} \ll \delta^{-2}\log\abs{X}$, and the lemma follows.

\end{proof}

\subsection*{Additive non-smoothing}

We conclude this section by giving the formal definition of what it means for a set to be additively non-smoothing, as discussed in Section~\ref{section:sketch}. Recall that, heuristically speaking, we say that a set $\Delta$ is additively non-smoothing if, when $\tau$ is such that $E_4(\Delta)=\tau\abs{\Delta}^3$, we have $E_8(\Delta) \ll \tau^3\abs{\Delta}^7$. For our formal definition, we need to both quantify precisely the upper bound for $E_8(\Delta)$ and also `relativise' this notion, so that we can handle additive energies modulo a given set $\Gamma$. 

To allow enough `room' in our proofs we need to in fact bound the higher energies not relative to the same $\Gamma$, but relative to some larger $\Gamma'$, where $\Gamma$ and $\Gamma'$ have some additive framework constructed between them. Furthermore, we need to control something slightly more general -- instead of just assuming an upper bound for
\[E_8(\Delta;\Gamma')=\ind{\Delta}^{(4)}\circ \ind{\Delta}^{(4)}\circ \ind{\Gamma'}(0),\]
we will need to assume upper bounds for the $L^\infty$ norm of the function on the right, and not only for the 4-fold convolution but also for the 2 and 3-fold convolutions. In the non-relative case when $\Gamma=\Gamma'=\{0\}$ (or, indeed, if $\Gamma=\Gamma'$ is any subgroup) none of this is required, since $\ind{\Delta}^{(4)}\circ \ind{\Delta}^{(4)}(\gamma)$ is always maximised at $\gamma=0$, and an upper bound for $E_8(\Delta)$ immediately implies suitable upper bounds for $E_4(\Delta)$ and $E_6(\Delta)$ by H\"{o}lder's inequality. In the more general setting we are working in, where $\Gamma$ may not be a subgroup but is just the spectrum of some Bohr set, which is only approximately structured, we need to be more careful, and it is easiest to insist upon appropriate control on all the various quantities required at the outset, resulting in the precise technical definition below.

\begin{definition}[Additively non-smoothing]\label{def2-ans}
We say that $\Delta$ is $(\tau,k)$-additively non-smoothing\label{def-ans} relative to an additive framework $\widetilde{\Gamma}$ if
\begin{enumerate}
\item $\Delta$ is $\Gamma_{\mathrm{top}}$-orthogonal, 
\item if $\Delta'\subset \Delta$ then
\[E_4(\Delta'; \Gamma_{\mathrm{bottom}}) \geq \frac{\tau}{\abs{\Delta}} \abs{\Delta'}^4,\]
\item \[\norm{\ind{\Delta}^{(2)}\circ \ind{\Delta}^{(2)}\circ \ind{\Gamma_{\mathrm{top}}}}_\infty \leq \tau^{1-1/k} \abs{\Delta}^3,\]
\item \[\norm{\ind{\Delta}^{(3)}\circ \ind{\Delta}^{(3)}\circ \ind{\Gamma_{\mathrm{top}}}}_\infty \leq \tau^{2-1/k} \abs{\Delta}^5,\]
\item 
\[\norm{\ind{\Delta}^{(4)}\circ \ind{\Delta}^{(4)}\circ \ind{\Gamma_{\mathrm{top}}}}_\infty \leq \tau^{3-1/k} \abs{\Delta}^7,\]
\end{enumerate}
and, furthermore, $\log(\tau^{-1}) \leq \tau^{-1/k}$.\footnote{This last condition is not at all essential; it just makes for less technical statements elsewhere.}
\end{definition}

In some sense this definition is already robust, in that large subsets of additively non-smoothing sets remain additively non-smoothing with only a small change in the parameters. To allow for particularly clean statements and proofs later on, however, it is convenient to introduce the following definition.

\begin{definition}[Robustly additively non-smoothing]
We say that $\Delta$ is $\kappa$-robustly $(\tau,k)$-additively non-smoothing relative to an additive framework $\widetilde{\Gamma}$ if every subset $\Delta'\subset \Delta$ of size $\abs{\Delta'}\geq \kappa\abs{\Delta}$ is $(\tau,k)$-additively non-smoothing relative to $\widetilde{\Gamma}$. 
\end{definition}
\section{From lack of progressions to large spectra}\label{section:lack}

We now have assembled enough tools to begin our proof in earnest. Recall that our goal is to show that either a set of a given density has many progressions, or else it has a suitable density increment. Our goal in this section is to prove Proposition~\ref{prop:big}, which roughly states that either our set $A$
\begin{enumerate}
\item is dense, or 
\item has many progressions, or
\item has a suitably density increment (either large or small), or
\item the large spectrum of (a smoothed version) of the Fourier transform of $A$ is both  large and additively non-smoothing. 
\end{enumerate}
The precise statement is technically quite cumbersome, largely because of the need to work with several different layers of Bohr sets at once. In particular, we are unable to give a lower bound for $T(A)$ directly, and instead must assume that one of the elements of our progression lies inside some narrower Bohr set. To this end, let
\[T(A,A',A)=\bbe_{x,y} \ind{A}(x)\ind{A}(y)\ind{2\cdot A'}(x+y).\]
We will address the question of how to pass from such a count to $T(A)$ itself in Section~\ref{section:concluding}.

\begin{proposition}\label{prop:big}
There is a constant $c>0$ such that the following holds. Let $k,h,t\geq 20$ be some parameters.

Let $B$ be a regular Bohr set of rank $d$ and suppose $A\subset B$ has density $\alpha$. Let $B' = B_{\rho}$ be a regular Bohr set, where $\rho\leq c\alpha^2/d$, and assume that $A'\subset B'$ has density $\alpha'$ satisfying $\alpha/2\leq \alpha'\leq 2\alpha$. Either
\begin{enumerate}
\item(large density) $\alpha \gg 1/k^2$, or 
\item (many progressions) $T(A,A',A)\gg \alpha^3\mu(B)\mu(B')$, or
\item $A$ has a density increment of strength either 
\begin{enumerate}
\item (small increment)  $[1,\alpha^{-1/k};\tilde{O}_\alpha(h\log t)]$ or
\item (large increment)  $[\alpha^{-1/k},\alpha^{-1+1/k};\tilde{O}_\alpha(h\log t)]$ 
\end{enumerate}
relative to $B'$, or
\item (non-smoothing large spectrum) there exists a set $\Delta$ and three quantities $\rho_{\mathrm{top}},\rho_{\mathrm{bottom}},\rho'\in (0,1)$ satisfying

\[\rho_{\mathrm{top}}\gg \alpha^{O(1)}(c/td)^{O(h)},\quad\rho_{\mathrm{bottom}} \gg (\alpha/d)^{O(1)},\quad\textrm{and}\quad \rho'\gg (\alpha/d)^{O(1)},\]
such that
\begin{enumerate}
\item 
\[\alpha^{-3+O(1/k)}\ll \abs{\Delta}\ls_\alpha \alpha^{-3},\]
\item 
there is an additive framework $\widetilde{\Gamma}$ of height $h$ and tolerance $t$ between $\Gamma_{\mathrm{top}} \coloneqq \Delta_{1/2}(2\cdot B'_{\rho_{\mathrm{top}}})$ and $\Gamma_{\mathrm{bottom}} \coloneqq \Delta_{1/2}(2\cdot B'_{\rho_{\mathrm{bottom}}})$, 
\item $\Delta$ is $\tfrac{1}{4}$-robustly $(\tau,k')$-additively non-smoothing relative to $\widetilde{\Gamma}$ for some $\alpha^{2-O(1/k)}\gg \tau \gg \alpha^{2+O(1/k)}$ and $k\geq k'\gg k$, and
\item if we let $B''=(2\cdot B'_{\rho_{\mathrm{top}}})_{\rho'}$ then for all $\gamma\in\Delta+\Gamma_{\mathrm{top}}$
\[\abs{\widehat{\bal{A}{B}}}^2\circ \abs{\widehat{\mu_{B''}}}^2(\gamma)\gg \alpha^{2+O(1/k)}\mu(B)^{-1},\]
and
\item \[\norm{\ind{\Delta}\circ \abs{\widehat{\mu_{B''}}}^2}_\infty \leq 2.\]
\end{enumerate}
\end{enumerate}
\end{proposition}
Condition 4(d) here is saying that $\Delta+\Gamma_{\mathrm{top}}$ behaves like an `analytically smoothed' spectrum of $\bal{A}{B}$ at level $\alpha^{1+O(1/k)}$. Condition 4(e) is saying that $\Delta$ satisfies a strong orthogonality condition with respect to the smoothing factor used in 4(d).

With Proposition~\ref{prop:big} in hand, we will then combine our structural result for additively non-smoothing sets with a spectral boosting argument to turn the final conclusion into a density increment in its own right. This will be done in Proposition~\ref{prop:specboost}. Finally, the proof of Proposition~\ref{mainprop} (and hence the proof of Theorem~\ref{mainthm}) will be concluded in Section~\ref{section:concluding} by combining Propositions~\ref{prop:big} and \ref{prop:specboost}.

In the rest of this section, we will prove Proposition~\ref{prop:big}. The proof of this is quite involved, as the statement would suggest, and will use almost all of the tools we have developed so far. To help orient the reader, we have provided a dependency graph for Proposition~\ref{prop:big} in Figure~\ref{figa}.

\begin{figure}[h]
\centering
\includegraphics[width=6in]{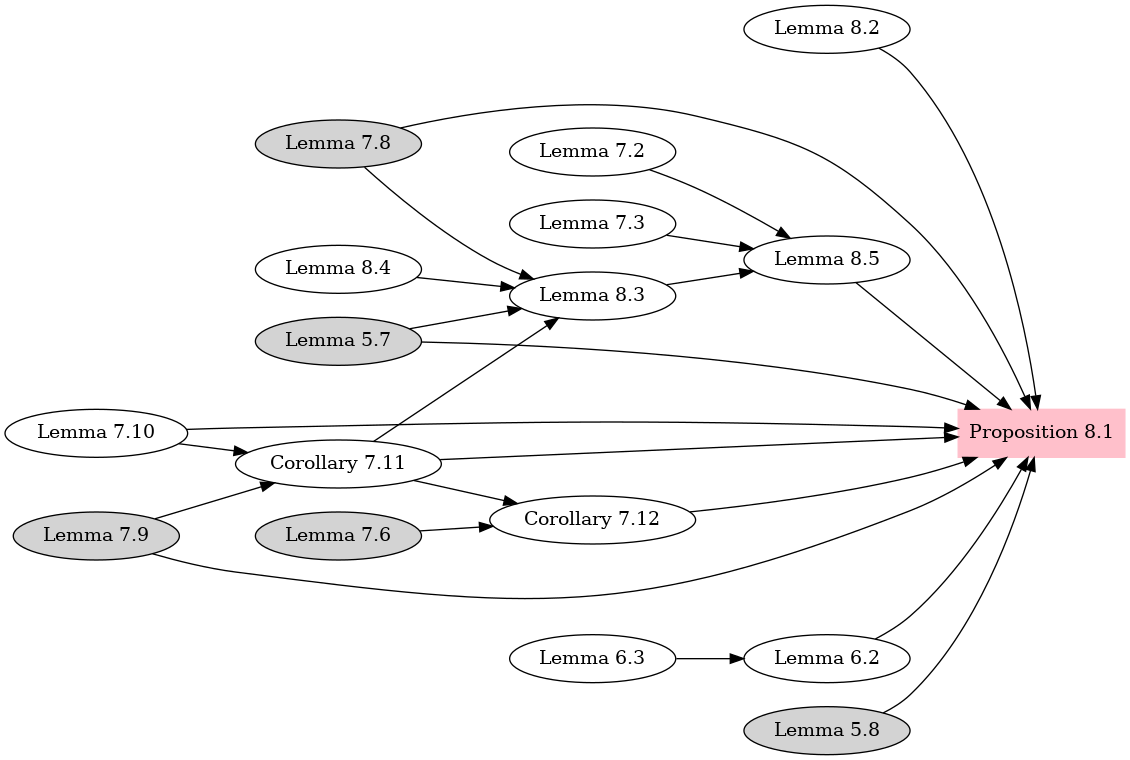}
\caption{Dependency chart for the proof of Proposition~\ref{prop:big}. Dependencies on lemmas from Section 4 are not shown. Those lemmas in grey are also used in the proof of Proposition~\ref{prop:specboost}.}
\label{figa}
\end{figure}

We begin by showing that either $A$ has many progressions, or else it has many large Fourier coefficients. There are two technical caveats here, however: the first is that, since Bohr sets are not closed under addition, we are unable to give a lower bound for $T(A)$ directly, and instead must work with $T(A,A',A)$ as described above.

The second technical issue is that, unlike in the model case of $\bbf_3^n$, we are unable to say directly that the a spectrum is large as a set, but only that some spectrum supports a large proportion of the $L^2$ mass of the Fourier transform of the balanced function $\bal{A}{B}$.
\begin{lemma}\label{lemma:progtofourier}
There is a constant $c>0$ such that the following holds. Let $B$ be a regular Bohr set of rank $d$ and suppose $A\subset B$ has density $\alpha$. Let $B' \subset B_\rho$ be a regular Bohr set with $\rho \leq c\alpha/d$, and suppose that $A' \subset B'$ has density $\alpha'$. Then either 
\begin{enumerate}
\item (many progressions) $T(A, A', A) \geq \tfrac{1}{2}\alpha^2 \alpha' \mu(B) \mu(B')$, or
\item (large $L^2$ mass on a spectrum) there is some $\eta \gg \alpha$ such that
\[ \sum_{\gamma \in \Delta_\eta(2\cdot A')} \Abs{\widehat{\bal{A}{B}}(\gamma)}^2 \gs_\alpha \eta^{-1}\mu(B)^{-1}. \]
\end{enumerate}
\end{lemma}
\begin{proof}
We have
\[ T(A, A', A) = \Inn{ \ind{A}\ast \ind{A}, \ind{2\cdot A'} } = \Inn{ \mu_A*\mu_A, \mu_{2\cdot A'} }\, \alpha^2 \alpha' \mu(B)^2 \mu(B'). \]
Replacing both copies of $\mu_A$ with their balanced functions $\bal{A}{B}=\mu_A-\mu_B$, we have
\[ \Inn{ \mu_A*\mu_A, \mu_{2\cdot A'} } = \Inn{ \bal{A}{B}\ast \bal{A}{B}, \mu_{2\cdot A'} }  + 2 \Inn{ \mu_A*\mu_B, \mu_{2\cdot A'} } - \Inn{ \mu_B*\mu_B, \mu_{2\cdot A'} }, \]
and we deal with the latter two inner products using regularity in an essentially identical manner to the computation in Lemma \ref{lemma:L2inc}, where we use the fact that $2\cdot A'\subset B'+B'\subset B_{2\rho}$. Thus, by Lemma \ref{lemma:regConv},
\begin{align*} 
\Inn{ \mu_A*\mu_B, \mu_{2\cdot A'} } 
&= \Inn{ \mu_A,\mu_B \ast \mu_{2\cdot A'} } \\
&= \Inn{ \mu_A, \mu_B } + O(\rho d \alpha^{-1} \mu(B)^{-1})\\
& = \mu(B)^{-1} + O(\rho d \alpha^{-1}\mu(B)^{-1}), 
\end{align*}
and similarly
\[ \Inn{ \mu_B*\mu_B, \mu_{2\cdot A'} } = \mu(B)^{-1} + O(\rho d \mu(B)^{-1}). \]
Thus
\[
T(A, A', A) =\alpha^2 \alpha' \mu(B) \mu(B')\brac{\Inn{ \bal{A}{B}\ast \bal{A}{B}, \mu_{2\cdot A'} }\mu(B) + 1
+ O(\rho d \alpha^{-1})}.\]
If we are not in the first case of the conclusion then, provided $\rho$ is small enough, this implies that
\[ \Inn{\bal{A}{B}\ast \bal{A}{B}, \mu_{2\cdot A'} } \leq -\tfrac{1}{4} \mu(B)^{-1}. \]
By Parseval's identity and the triangle inequality, this implies that
\[ \sum_{\gamma \in \widehat{G}} \Abs{ \widehat{\bal{A}{B}}(\gamma) }^2 \Abs{ \widehat{\mu_{2\cdot A'}}(\gamma) } \geq \tfrac{1}{4}\mu(B)^{-1}. \]
Since $\norm{\bal{A}{B}}_2^2 \leq \alpha^{-1} \mu(B)^{-1}$ the contribution to this sum from any terms with $\Abs{ \widehat{\mu_{2\cdot A'}}(\gamma) }$ less than $\tfrac{1}{8}\alpha$ is negligible compared to the right-hand side; thus
\[ \sum_{\gamma \in \Delta_{\alpha/8}(2\cdot A')} \Abs{ \widehat{\bal{A}{B}}(\gamma) }^2 \Abs{ \widehat{\mu_{2\cdot A'}}(\gamma) } \gg \mu(B)^{-1}. \]
The dyadic pigeonhole principle then gives that there is some $\eta \gg \alpha$ such that
\[ \sum_{\gamma \in \Delta_\eta(2\cdot A')\backslash \Delta_{2\eta}(2\cdot A')} \Abs{\widehat{\bal{A}{B}}(\gamma)}^2 \Abs{ \widehat{\mu_{2\cdot A'}}(\gamma) } \gs_\alpha \mu(B)^{-1},\]
and the result follows.
\end{proof}

In much of the previous work on Roth's theorem in the integers, the next natural step after a result like Lemma~\ref{lemma:progtofourier} is to work with the function $\Abs{\widehat{\bal{A}{B}}}^2$ restricted to $\Delta_\eta(2\cdot A')$, and prove structural results about this function. This is the approach taken in \cite{Bl:2016}, for example, and in an analytic sense would be the natural way to proceed. 

Unfortunately, the structural result for additively non-smoothing sets that is vital to our argument relies on a delicate combinatorial argument, which operates on sets, and does not seem to adapt well to `weighted indicator functions', like the function in the previous paragraph. In the model setting of $\mathbb{F}_3^n$ it is straightforward to pass from this weighted indicator function to a genuine indicator function. When $G=\bbz/N\bbz$, however, this is much harder to accomplish, largely due to the fact that the set $2\cdot A'$ whose spectrum we are restricting to is supported on $2\cdot B'$, which is a much smaller set than the $B$ occurring in the weight $\abs{\widehat{\bal{A}{B}}}^2$. 

We will therefore need to take somewhat of a technical detour that allows us to pass from a large `mass' of a weighted indicator function to a large set, which is suitable for applying our structural result to. The first step is the following lemma, which extracts a large set with good spectral properties from a large $L^2$ Fourier mass (at least, it either does so, or else we have found a strong density increment for $A$).

In parsing the following statement and proof, one should think of $\delta$ as $\alpha^{c_1}$ and $K$ as $\alpha^{-c_2}$ for some small constants $c_1,c_2$ satisfying $0<c_1<c_2/10$. In particular, in our application, $K^{1/2}\delta$ is larger than $1$, and the upper bound $K\leq \alpha^{-2}$ in the hypotheses will be easily satisfied. Furthermore, in our application we will have $\abs{\Delta}=\alpha^{-O(1)}$, and so logarithmic losses in the size of $\Delta$ are absorbed into the logarithmic losses in $\alpha$.
\begin{lemma}\label{lemma:bigspec}
There is a constant $c>0$ such that the following holds.  Let $B$ and $B^*$ be regular Bohr sets, both of rank $d$, and suppose $A\subset B$ has density $\alpha$. Let $\delta\in(0,1)$ and $K\geq 1$ be some parameters such that $K\leq \alpha^{-2}$. Suppose that $\Delta$ is such that 
\[\sum_{\gamma\in \Delta+\Delta_{1/2}(B^*)}\abs{\widehat{\bal{A}{B}}(\gamma)}^2 \geq  \delta\alpha^{-1}\mu(B)^{-1}.\]
Finally, suppose that $B^*\subset B_\rho$, where $\rho \leq c\alpha^4\delta^2/Kd\log \abs{\Delta}$.
Either
\begin{enumerate}
\item there exists some $\Delta'\subset \Delta$ such that
\begin{enumerate}
\item $\abs{\Delta'}\gs_{\delta/\abs{\Delta}} K^{-1}\alpha^{-3}$ and
\item if $\gamma\in \Delta'$ then 
\[\abs{\widehat{\bal{A}{B}}}^2\circ \abs{\widehat{\mu_{B^*}}}^2(\gamma) \gg \frac{\delta\alpha^{-1}}{\abs{\Delta}}\mu(B)^{-1},\]
\end{enumerate}
or
\item $A$ has a density increment of strength $[1,0;\tilde{O}_{\alpha\delta}(1)]$ relative to $B^*$, or
\item $A$ has a density increment of strength 
\[[\nu K^{1/2}\delta,\nu^{-1}K^{-1/2}\delta^{-1}\alpha^{-1};\tilde{O}_{\alpha\delta/\abs{\Delta}}(1)],\]
where $\nu=\min(1,\delta\alpha^{-1})$,  relative to $B^*$.
\end{enumerate}
\end{lemma}
\begin{proof}
Let $\Gamma=\Delta_{1/2}(B^*)$. By assumption, 
\[\Inn{\abs{\widehat{\bal{A}{B}}}^2 , \ind{\Delta+\Gamma}}\geq \delta\alpha^{-1}\mu(B)^{-1}.\]
We decompose $\Delta+\Gamma$ as a disjoint union\footnote{The existence of such a disjoint union easily follows from a greedy algorithm, for example.} $\sqcup_{\gamma\in\Delta}(\gamma+\Gamma(\gamma))$, where $\Gamma(\gamma)\subset \Gamma$, so that
\[\sum_{\gamma\in\Delta}\sum_{\lambda\in \Gamma(\gamma)}\abs{\widehat{\bal{A}{B}}(\gamma+\lambda)}^2\geq \delta\alpha^{-1}\mu(B)^{-1}.\]
Let the inner sum be denoted by $F(\gamma)$, and note that trivially 
\[F(\gamma)\leq \abs{\widehat{\bal{A}{B}}}^2\circ \ind{\Gamma}(\gamma)\ll \abs{\widehat{\bal{A}{B}}}^2\circ \abs{\widehat{\mu_{B^*}}}^2(\gamma).\]
We similarly have $F(\gamma)\leq \norm{\bal{A}{B}}_2^2\leq \alpha^{-1}\mu(B)^{-1}$. By dyadic pigeonholing there exists some $\alpha^{-1}\geq \kappa\gg \delta\alpha^{-1}/\abs{\Delta}$ and $\Delta'\subset \Delta$ such that
\[\abs{\Delta'}\gs_{\delta/\abs{\Delta}} \kappa^{-1}\delta\alpha^{-1}\quad\textrm{ and }\quad\textrm{ if }\gamma\in\Delta'\textrm{ then }F(\gamma)\gg \kappa\mu(B)^{-1}.\]
For brevity, we adopt the convention that all logarithmic losses in the $\gs$ and $\tilde{O}(\cdot)$ notation for the remainder of this proof will be logarithmic in $\alpha\delta/\abs{\Delta}$.

If this occurs for some $\kappa \leq K\delta\alpha^2$ then we have $\abs{\Delta'}\gs K^{-1}\alpha^{-3}$, and we are in the first case of the lemma. We will therefore henceforth assume that $\kappa > K\delta\alpha^2$. Discarding elements if necessary, we may assume that $\abs{\Delta'}\ls \kappa^{-1}\delta\alpha^{-1}$, while a similar lower bound also holds. We next note that, by construction of $\Delta'$,
\begin{equation}\label{eq:fromhere1}\Inn{\abs{\widehat{\bal{A}{B}}}^2\circ \abs{\widehat{\mu_{B^*}}}^2, \ind{\Delta'}}\gg \kappa\mu(B)^{-1}\abs{\Delta'}.
\end{equation}
We would now like to deduce that $\Abs{\widehat{\bal{A}{B}}}$ is reasonably large on some translate of $\Delta'$. The simplest way to proceed would be to note that $\Norm{\abs{\widehat{\mu_{B^*}}}^2}_1=\mu(B^*)^{-1}$, so that \eqref{eq:fromhere1} immediately implies that
\[\norm{\abs{\widehat{\bal{A}{B}}}^2\circ\ind{\Delta'}}_\infty\gg \kappa \abs{\Delta'}\mu_B(B^*) \gs \delta\alpha^{-1}\mu_B(B^*).\]
When $B=B^*$ (which will be the case in our application when $G=\bbf_p^n$, for example, where dilates of a Bohr set do not change the underlying set) this simple argument will suffice. In general, however, $\mu_B(B^*)$ will be exponentially small in the rank of $B$, which is a loss that we cannot afford.

Instead, we require a more technical argument, which exploits the regularity of Bohr sets to `remove' the $\abs{\widehat{\mu_{B^*}}}^2$ from \eqref{eq:fromhere1} with only a loss of a $\mu(B)^{-1}$ factor, hence avoiding the exponential loss in rank. The trade-off is that we must replace $\bal{A}{B}$ with the balanced function of $A$ over some smaller Bohr set. We defer the precise statement and its proof to the subsequent Lemma~\ref{lemma:phys_inc}.

Applying Lemma~\ref{lemma:phys_inc} with $\omega=\ind{\Delta'}$ we deduce that 
\begin{equation}\label{eq:fromm}
\Inn{ \Abs{\widehat{f}}^2, \ind{\Delta'}} \gg \kappa \abs{\Delta'}\gs\delta \alpha^{-1},
\end{equation}
where $f=(\alpha^{-1}\ind{A}-\ind{B})\mu_{B^{(1)}+z}$, where $B^{(1)}=B^*_{\rho^*}$ is a regular Bohr set, and $1\geq  \rho^*\gg \alpha^4\delta/d$. Let $f'(x)=f(x+z)$, so that $f'$ is supported on $B^{(1)}$. 

If $\norm{f'}_1\geq 3$, say, then $A$ has a density increment of strength $[1,0;O(1)]$ relative to $B^{(1)}$. Since $B^{(1)}$ has the same rank as $B^*$, and $\abs{B^{(1)}}\geq (\delta \alpha/2d)^{O(d)}\abs{B^*}$, it follows that $A$ has a density increment of strength $[1,0;\tilde{O}(1)]$ relative to $B^*$, and we are in the second case.

We therefore suppose that $\norm{f'}_1\leq 3$. It follows that $\norm{f'}_2^2 \leq \norm{f'}_\infty\norm{f}_1\ll \alpha^{-1}\mu(B^{(1)})^{-1}$. Furthermore, from \eqref{eq:fromm} we trivially have \[\|f'\|_1^2\geq \|\widehat{f'}\|_\infty^2\gtrsim \kappa>K\delta\alpha^2.\]
Importantly, the quantity $\norm{f'}_1^2\norm{f'}_2^{-2}\mu(B^{(1)})^{-1}$ that is used in Corollary~\ref{cor:massdim} is $\gs \delta\alpha^3$, so that any logarithmic losses incurred in this parameter are $\tilde{O}(1)$. 

We now apply the dyadic pigeonhole principle to  \eqref{eq:fromm} and conclude that, since $\norm{f'}_1\ll 1$ and $\abs{\Delta'}\ls \kappa^{-1}\delta\alpha^{-1}$, there exists some $1\geq \epsilon \gs \kappa^{1/2}>(K\delta)^{1/2}\alpha$ such that 
\[\abs{\Delta_\epsilon(f')\cap \Delta'}\gs\epsilon^{-2}\delta\alpha^{-1}.\]
We now apply Corollary~\ref{cor:massdim} to $\omega=\ind{\Delta_{\epsilon}(f')\cap \Delta'}$ with $B'$ replaced by $B^{(2)}=B^{(1)}_{\rho^{(1)}}$ where $\rho^{(1)}=c^{(1)}\alpha\delta/d$, say, for some sufficiently small constant $c^{(1)}>0$. This produces some $\Delta''\subset \Delta'$ of size $\abs{\Delta''}\gs \nu\epsilon^{-1}\delta\alpha^{-1}$ and dimension $\dim(\Delta'';\Delta_{1/2}(B^{(2)}))\ls \epsilon^{-1}$. If necessary, we remove elements from $\Delta''$ to ensure that the upper bound $\abs{\Delta''}\ls \epsilon^{-1}\delta\alpha^{-1}$ also holds (note that this will not affect our dimension bound). We then remove $\Delta''$ from $\Delta'$ and, provided at least half of $\Delta'$ remains, repeat the above argument from equation \eqref{eq:fromhere1}.

Continuing in this fashion, we obtain some $\Delta''_1,\ldots,\Delta''_r$, disjoint subsets of $\Delta'$, such that 
\[\sum_{i=1}^r \abs{\Delta''_i}\geq \tfrac{1}{2}\abs{\Delta'}\gs \kappa^{-1}\delta\alpha^{-1}\]
along with associated $1\geq \epsilon_i\gs\kappa^{1/2}$ such that for $1\leq i\leq r$ we have 
\[\epsilon_i^{-1}\delta\alpha^{-1}\gs \abs{\Delta''_i}\gs \nu \epsilon_i^{-1}\delta\alpha^{-1}\]
and $\dim(\Delta''_i; \Delta_{1/2}(B^{(2)}))\ls \epsilon_i^{-1}$.

By the dyadic pigeonhole principle we may find some $I\subset \{1,\ldots,r\}$ such that there exists some $1\geq \epsilon \gs\kappa^{1/2}$ such that for all $i\in I$ we have $2\epsilon > \epsilon_i\geq \epsilon$, and 
\[\sum_{i\in I} \abs{\Delta_i''}\gs \abs{\Delta'}.\]
In particular, 
\[\nu^{-1}\epsilon \kappa^{-1}\gs \abs{I}\gs \epsilon\kappa^{-1}.\]
We now let $r'=\lceil K^{1/2}\alpha\abs{I}\rceil$ (so that certainly $1\leq r'\leq \abs{I}$), and let $\tilde{\Delta}$ be the union of any $r'$ of the $\Delta''_i$ for $i\in I$. In particular,
\[\Abs{\tilde{\Delta}}\gs \nu\kappa^{-1}K^{1/2}\delta \]
and, recalling that we are assuming $\kappa> K\delta\alpha^2$, since $\epsilon \gg \kappa^{1/2}$,
\[\dim(\tilde{\Delta}; \Delta_{1/2}(B^{(2)}))\ls \max\brac{ \epsilon^{-1}, K^{1/2}\kappa^{-1}\alpha}\ls K^{-1/2}\delta^{-1}\alpha^{-1}.\]
By Lemma~\ref{lemma:dimcovering} the set $\tilde{\Delta}$ is $\tilde{O}(K^{-1/2}\delta^{-1}\alpha^{-1})$-covered by $\Delta_{1/2}(B^{(2)})$, and hence certainly $\tilde{\Delta}+\Gamma$ is $\tilde{O}(K^{-1/2}\delta^{-1}\alpha^{-1})$-covered by $\Gamma+\Delta_{1/2}(B^{(2)})$. If we choose $B^{(3)}=B^{(2)}_{\rho^{(2)}}$ with $\rho^{(2)}=c^{(2)}/d$ for some suitable constant $c^{(2)}>0$, chosen in particular such that $B^{(3)}$ is regular, then by Lemma~\ref{lemma:bohrspectra}
\[\Gamma+\Delta_{1/2}(B^{(2)})\subset \Delta_{1/2}(B^{(3)}),\]
and so $\tilde{\Delta}+\Gamma$ is $\tilde{O}(K^{-1/2}\delta^{-1}\alpha^{-1})$-covered by $\Delta_{1/2}(B^{(3)})$. Observe that since $\tilde{\Delta}\subset \Delta$ we have that $\sqcup_{\gamma\in\tilde{\Delta}}(\gamma+\Gamma(\gamma))\subset \tilde{\Delta}+\Gamma$, and so in particular
\[\sum_{\tilde{\Delta}+\Gamma}\abs{\widehat{\bal{A}{B}}(\gamma)}^2 \geq\sum_{\gamma\in\tilde{\Delta}}F(\gamma).\]
It follows that
\[\sum_{\tilde{\Delta}+\Gamma}\abs{\widehat{\bal{A}{B}}(\gamma)}^2 \gs \kappa\mu(B)^{-1}\Abs{\tilde{\Delta}}\gs  \nu K^{1/2}\delta\mu(B)^{-1},\]
and hence by Lemma~\ref{lemma:L2inc} $A$ has a density increment of strength 
\[[\nu K^{1/2}\delta,\nu^{-1}K^{-1/2}\delta^{-1}\alpha^{-1};\tilde{O}(1)]\]
relative to $B^{(3)}$. Using the fact that $B^{(3)}=B^*_{\tilde{\rho}}$ for some $\tilde{\rho}\geq (c\alpha\delta/d)^{O(1)}$, we deduce that $A$ has a density increment of the same strength (possibly with a worse constant in the third parameter) relative to $B^*$, and we are in the third case.
\end{proof}

We now finish the proof of Lemma~\ref{lemma:bigspec} by proving the following technical result, which allows us to `quotient out' by the spectrum of a narrower Bohr set without losing too much. To help parse the statement, one should note that if $B'=B$ then a similar conclusion with $f'$ replaced by $\bal{A}{B}$ follows immediately by averaging, using the fact that $\Norm{\widehat{\mu_{B'}}^2}_1=\mu(B')$.

\begin{lemma}\label{lemma:phys_inc}
There is a constant $c>0$ such that the following holds. Let $B$ and $B'$ be regular Bohr sets with $B'\subset B$, both of rank $d$, and suppose that $A\subset B$ has density $\alpha$. Let $\omega:\widehat{G}\to\bbr_{\geq 0}$ be a non-negative function. If $\kappa>0$ is such that
\[\Inn{\abs{\widehat{\bal{A}{B}}}^2\circ \abs{\widehat{\mu_{B'}}}^2,\omega}\geq \kappa \mu(B)^{-1}\norm{\omega}_1\]
then there is a regular Bohr set $B''=B'_{\rho'}$ for some $\rho'$ satisfying $1\geq \rho' \gg \alpha^2\kappa/d$ and some $z\in G$ such that 
\[\Inn{\Abs{\widehat{f}}^2,\omega}\gg \kappa \norm{\omega}_1,\]
where $f=(\alpha^{-1}1_A-1_B)\mu_{B''+z}$.
\end{lemma}
\begin{proof}
By Parseval's identity
\[\Inn{ (\bal{A}{B}\circ \bal{A}{B})\overline{\widecheck{\omega}}, \mu_{B'}\circ \mu_{B'}} \geq \kappa \mu(B)^{-1}\norm{\omega}_1.\]
Averaging over $B'$ there is some $v\in B'$ such that 
\[\abs{\Inn{\bal{A}{B}, (\widecheck{\omega}\mu_{B'+v})\ast \bal{A}{B}}} = \abs{\Inn{ (\bal{A}{B}\circ \bal{A}{B})\overline{\widecheck{\omega}}, \mu_{B'+v}}} \geq \kappa \mu(B)^{-1}\norm{\omega}_1.\]
Let $\rho=c\alpha^2\kappa/d$, where $c$ is some small absolute constant to be chosen later, but chosen in particular such that $B'_\rho$ is regular. We will now replace $\bal{A}{B}$, which is supported on $B$, by some function supported on a translate of $B'_\rho$. To this end, we use the regularity of $B$ and Lemma~\ref{lemma:regConv} to bound the error
\[\abs{\Inn{ \bal{A}{B}, (\widecheck{\omega}\mu_{B'+v})\ast \bal{A}{B}} - \Inn{\bal{A}{B}(1_B\ast \mu_{B'_\rho}),(\widecheck{\omega}\mu_{B'+v})\ast \bal{A}{B}}}\]
by
\[\ll \rho d\mu(B)\norm{\bal{A}{B}((\widecheck{\omega}\mu_{B'+v})\ast \bal{A}{B})}_\infty\ll \rho d\alpha^{-2}\mu(B)^{-1}\norm{\omega}_1\] 
using the bounds $\norm{\bal{A}{B}}_\infty\ll \alpha^{-1}\mu(B)^{-1}$ and
\begin{align*}
\norm{\widecheck{\omega}\mu_{B'+v}\ast \bal{A}{B}}_\infty
&\leq \norm{\bal{A}{B}}_\infty\norm{\widecheck{\omega}}_\infty\\
&\ll \alpha^{-1}\mu(B)^{-1}\norm{\omega}_1.
\end{align*}
In particular, provided $\rho \leq c\alpha^2\kappa/d$ for some sufficiently small $c$, this error is negligible, so that
\[\abs{\Inn{\bal{A}{B}(\ind{B}\ast \mu_{B'_\rho}),(\widecheck{\omega}\mu_{B'+v})\ast \bal{A}{B}}} \geq \tfrac{1}{2}\kappa \mu(B)^{-1}\norm{\omega}_1.\]
Averaging over $B$ and recalling $\bal{A}{B}=(\alpha^{-1}\ind{A}-\ind{B})\mu(B)^{-1}$, we deduce there is some $w\in B$ such that 
\[\abs{\Inn{(\alpha^{-1}\ind{A}-\ind{B})\mu_{B'_\rho+w}\circ \bal{A}{B}, \widecheck{\omega}\mu_{B'+v}}} \geq \tfrac{1}{2}\kappa \mu(B)^{-1}\norm{\omega}_1.\]
It follows that, if $f_0=\alpha^{-1}1_A-1_B$, then
\[\mu(B')^{-1}\abs{\Inn{f_0\mu_{B'_\rho+w}\circ f_01_{B'+B'_\rho+w-v}, \widecheck{\omega}1_{B'+v}}} \geq\tfrac{1}{2}\kappa \norm{\omega}_1.\]

Observe that $f_0\mu_{B'_\rho+w}\circ f_01_{B'+B'_\rho+w-v}$ is supported on $B'+2B'_\rho+v$. By the regularity of $B'$ we can replace the $1_{B'+v}$ by $1_{B'+2B'_\rho+v}$ in the inner product with error bounded by
\[\mu(B')\rho d\norm{\widecheck{\omega}(f_0\mu_{B'_\rho+w}\circ f_01_{B'+B'_\rho+w-v})}_\infty\ll \rho d\alpha^{-2}\mu(B')\norm{\omega}_1,\]
using $\norm{f_0}_\infty\ll \alpha^{-1}$. Again, provided $\rho$ is small enough this error is negligible, and so
\[\mu(B')^{-1}\abs{\Inn{ f_0\mu_{B'_\rho+w}\circ f_0\ind{B'+B'_\rho+w-v}, \widecheck{\omega}}}\geq \tfrac{1}{4}\kappa \norm{\omega}_1.\]
Finally, by one more application of regularity, we can replace $\ind{B'+B'_\rho+w-v}$ by $\ind{B'+w-v}$, and so 
\[\abs{\Inn{ f_0\mu_{B'_\rho+w}\circ f_0\mu_{B'+w-v}, \widecheck{\omega}}}\geq \tfrac{1}{8}\kappa \norm{\omega}_1.\]
By Parseval's identity and the triangle inequality, with $f_1= f_0\mu_{B'_\rho+w}$ and $f_2=f_0\mu_{B'+w-v}$, 
\[\sum_\gamma \abs{\widehat{f_1}(\gamma)}\abs{\widehat{f_2}(\gamma)}\omega(\gamma)\geq \tfrac{1}{8}\kappa \norm{\omega}_1.\]
By the Cauchy--Schwarz inequality there is some $i\in \{1,2\}$ such that
\[\sum_\gamma \abs{\widehat{f_i}(\gamma)}^2\omega(\gamma)\gg\kappa \norm{\omega}_1\]
which concludes the proof (with $B''$ being either $B'_\rho$ or $B'$). 
\end{proof}

The set produced by Lemma~\ref{lemma:bigspec} has large size and good spectral properties, but it is vital for both the structural result on additively non-smoothing sets, and also the spectral boosting arguments that follow, that it also have strong orthogonality properties. (Recall that a set $\Delta$ being $\Gamma$-orthogonal means that the translates $(\gamma+\Gamma)_{\gamma\in \Delta}$ are all disjoint.) We therefore now prove the following, which applies Lemma~\ref{lemma:bigspec} in an iterative fashion to find such orthogonality.

Again, in parsing the following statement and the proof, it is helpful to have a rough idea of what the parameters will be in our application: one may think of $\eta=c_1\alpha$, $\delta=\alpha^{c_2}$ and $K=\alpha^{-c_3}$ for some small constants $c_1,c_2,c_3>0$ with $c_2< c_3/10$, say.
\begin{lemma}\label{lemma:newone}
There is a constant $c>0$ such that the following holds. Let $B$ and $B'$ be regular Bohr sets of rank $d$. Suppose that $A\subset B$ has density $\alpha$ and $A'\subset B'$ has density $\alpha'$ such that $\alpha/2\leq \alpha'\leq 2\alpha$. Suppose that $B'\subset B_\rho$ where $\rho \leq c/d$. 

Let $\eta,\delta\in(0,1)$ and $\Delta\subset \Delta_\eta(A')$ be such that
\[\sum_{\gamma\in \Delta}\abs{\widehat{\bal{A}{B}}(\gamma)}^2 \geq \delta \alpha^{-1}\mu(B)^{-1}.\]
Let $1\leq K\leq \alpha^{-2/3}$ be some parameter, and suppose that $B^{(1)}$ and $B^{(2)}$ are regular Bohr sets, each of rank $d$, such that $B^{(1)}\subset B'_{\rho'}$ and $B^{(2)}\subset B^{(1)}_{\rho_1}$ where $\rho',\rho_1 \leq c\alpha^5\delta^2\eta^3/d$.

Then either
\begin{enumerate}
\item there exists a subset $\Delta'\subset \Delta$ such that
\begin{enumerate}
\item $\Delta'$ is $\Delta_{1/2}(B^{(1)})$-orthogonal,
\item $\abs{\Delta'}\gs_{\alpha\delta\eta} K^{-3}\alpha^{-3}$, 
\item $\norm{\ind{\Delta'}\ast \abs{\widehat{\mu_{B^{(2)}}}}^2}_\infty \leq 2$, and 
\item for all $\gamma\in \Delta'+\Delta_{1/2}(B^{(1)})$
\[\abs{\widehat{\bal{A}{B}}}^2\circ \abs{\widehat{\mu_{B^{(2)}}}}^2(\gamma)\gs_{\alpha\delta\eta} K^{-1}\delta\eta^4\alpha^{-2}\mu(B)^{-1},\]
\end{enumerate}
or
\item $A$ has a density increment of strength $[1,0;\tilde{O}_{\alpha\delta\eta}(1)]$ relative to $B^{(2)}$, or 
\item $A$ has a density increment of strength 
\[[K^{1/2}\delta, K^{-1/2}\delta^{-1}\alpha^{-1}; \tilde{O}_{\alpha\delta\eta}(1)]\]
relative to $B^{(2)}$, where $\nu=\min(1,\delta\alpha^{-1})$, or
\item $A$ has a density increment of strength 
\[[\nu' K^{1/2}\delta\eta^2\alpha^{-2},(\nu')^{-1} K^{-1/2}\delta^{-1}\eta^{-2}\alpha; \tilde{O}_{\alpha\delta\eta}(1)]\]
relative to $B^{(2)}$, where $\nu'=\min(1,\delta\eta^2\alpha^{-3}/K)$.
\end{enumerate}
\end{lemma}
\begin{proof}
Let $B^{(3)}=B^{(2)}_{\rho_2}$ and $B^{(4)}=B^{(3)}_{\rho_3}$ where $\rho_2=c_2\eta^3\alpha^2/d$ and $\rho_3=c_3/d$ for some small absolute constants $c_2,c_3>0$ chosen later, in particular chosen such that both Bohr sets are regular. Let $\Gamma_i=\Delta_{1/2}(B^{(i)})$. 

Provided $\rho_3\leq c/d$ for some sufficiently small constant $c>0$, Lemma~\ref{lemma:bohrspectra} implies that $\Gamma_3-\Gamma_3\subset \Gamma_4$. Similarly, supposing $\rho_2\leq c/d$, we can ensure that $\Gamma_1\subset \Gamma_3$. Let $L\geq 1$ be some parameter to be chosen later. All logarithmic losses in this proof are with respect to $\alpha\eta\delta$. For convenience, for $\Lambda\subset\widehat{G}$ let
\[F(\Lambda) = \sum_{\gamma\in \Lambda}\abs{\widehat{\bal{A}{B}}(\gamma)}^2.\]

We will iteratively construct a sequence of $\tilde{\Delta}_j$ and $\Delta_j$, both subsets of $\Delta$, as follows. We begin with $\tilde{\Delta}_0=\Delta$, and $\Delta_0=\emptyset$.\footnote{We are using the convention that $\emptyset+\Delta=\emptyset$ for any set $\Delta$.} Suppose then that we have constructed $\tilde{\Delta}_j$ and $\Delta_j$ for some $j\geq 0$. If
\begin{equation}\label{eq:endcond}
F(\Delta_j+\Gamma_2)\geq \frac{1}{L}\delta \alpha^{-1}\mu(B)^{-1}
\end{equation}
then we exit the iterative procedure -- we will return to this case below.

Otherwise, we let $\tilde{\Delta}_{j+1}=\Delta\backslash (\sqcup_{k=0}^j\Delta_j+\Gamma_2)$ and let $\tilde{\Delta}_{j+1}'$ be a maximal $\Gamma_3$-orthogonal subset of $\tilde{\Delta}_{j+1}$. Since $\tilde{\Delta}_{j+1}\subset \Delta\subset \Delta_\eta(A')$, by Lemma~\ref{lemma:orthbessel}, provided $\rho'\leq c\delta\eta/d$ for some sufficiently small constant $c>0$, we have $\Abs{\tilde{\Delta}_{j+1}'}\ls \eta^{-2}\alpha^{-1}$ (note in particular that this upper bound is independent of $j$).

Observe that, by induction, we have
\[F\brac{\tilde{\Delta}_{j+1}}\geq \brac{1-\frac{j+1}{L}}\delta\alpha^{-1}\mu(B)^{-1}.\]
In particular, if $j+1\leq 2L$ then this is at least $\tfrac{1}{2}\delta\alpha^{-1}\mu(B)^{-1}$. By maximality and  Lemma~\ref{lemma:trivorth} we have that $\tilde{\Delta}_{j+1}\subset \tilde{\Delta}_{j+1}'+\Gamma_3-\Gamma_3\subset \tilde{\Delta}_{j+1}'+\Gamma_4$. In particular,
\[F\brac{\tilde{\Delta}_{j+1}'+\Gamma_4}\geq \tfrac{1}{2}\delta\alpha^{-1}\mu(B)^{-1}.\]
By Lemma~\ref{lemma:bigspec}, applied with $B^*=B^{(4)}$, either we are in one of the two density increment cases, or there is some $\Delta_{j+1}\subset \tilde{\Delta}_{j+1}'$ such that $\abs{\Delta_{j+1}}\gs K^{-1}\alpha^{-3}$ (note that this lower bound is also independent of $j$). We then continue this iterative construction until either $j+1>2L$ or \eqref{eq:endcond} holds.

Suppose that we have performed this iterative construction at least $2L$ times. This means we have produced at least $2L$ disjoint $\Delta_1,\ldots,\Delta_{2L}$, each of which is $\Gamma_3$-orthogonal, and moreover such that $\Delta_i$ is disjoint from $\Delta_j+\Gamma_2$ for all $j<i$, and $\abs{\Delta_i}\gs K^{-1}\alpha^{-3}$ for $1\leq i\leq 2L$. 

If we let $\Delta'=\sqcup_{i=1}^{2L} \Delta_i$ then we claim that $\Delta'$ is $\Gamma_1$-orthogonal. Indeed, if not, then there exist two distinct $\gamma_1,\gamma_2\in \Delta'$ and some $\lambda_1,\lambda_2\in \Gamma_1$ such that
\[\gamma_1+\lambda_1=\gamma_2+\lambda_2.\]
If there is some $1\leq i\leq 2L$ such that $\gamma_1,\gamma_2\in \Delta_i$ then this contradicts the $\Gamma_3$-orthogonality of $\Delta_i$, since $\Gamma_1\subset \Gamma_3$. Otherwise, without loss of generality, there is some $i<j$ such that $\gamma_1\in \Delta_i$ and $\gamma_2\in \Delta_j$. Since
\[\gamma_2=\gamma_1+\lambda_1-\lambda_2\in \Delta_i+\Gamma_1-\Gamma_1\subset \Delta_i+\Gamma_2\]
this contradicts the fact that $\Delta_j$ is disjoint from $\Delta_i+\Gamma_2$. 

We have thus shown that $\Delta'$ is $\Gamma_1$-orthogonal. In particular, since $\Delta'\subset \Delta_\eta(A')$, we have that $\abs{\Delta'}\ls \eta^{-2}\alpha^{-1}$ by Lemma~\ref{lemma:orthbessel}. By construction, however,
\[\abs{\Delta'}\gs \frac{L}{K}\alpha^{-3},\]
and hence $L\ls K\alpha^2\eta^{-2}$. In particular, if we choose $L$ sufficiently large (but still satisfying $L\ls K\alpha^2\eta^{-2}$) then this is a contradiction, and hence the above iterative procedure must exit with \eqref{eq:endcond} for some $j<2L$.

 Suppose then that 
\[F(\Delta_j+\Gamma_2) \geq \tfrac{1}{L}\delta\alpha^{-1}\mu(B)^{-1}\]
for some $j$ (fixed for the remainder of the proof). In this case, we apply Lemma~\ref{lemma:bigspec} once again, this time with $B^*=B^{(2)}$, and with $K$ replaced by $K^3$. Thus either we have a suitable density increment, or else there is some $\Delta'\subset \Delta_j$ such that $\Abs{\Delta'}\gs K^{-3}\alpha^{-3}$ and if $\gamma\in \Delta'$ then 
\[\abs{\widehat{\bal{A}{B}}}^2\circ \abs{\widehat{\mu_{B^{(2)}}}}^2(\gamma) \gg \frac{\delta \alpha^{-1}}{L\abs{\Delta_j}}\mu(B)^{-1}\gs K^{-1}\delta\eta^4\alpha^{-2}\mu(B)^{-1}.\]
We claim that this is a suitable $\Delta'\subset \Delta$ for the first case of the lemma to hold. We have already seen that property (b) holds. 

To verify property (d), we note that if $\gamma\in \Gamma_1$ then, by the argument in the proof of Lemma~\ref{lemma:bohrspectra}, for any $\lambda\in \widehat{G}$
\[ \widehat{\mu_{B^{(2)}}}(\gamma+\lambda)= \widehat{\mu_{B^{(2)}}}(\lambda)+O(\rho_1d).\]
In particular, if $\gamma=\gamma_1+\gamma_2$ with $\gamma_1\in \Delta'$ and $\gamma_2\in \Gamma_1$, then 
\begin{align*}
\abs{\widehat{\bal{A}{B}}}^2\circ \abs{\widehat{\mu_{B^{(2)}}}}^2(\gamma)
&=\sum_{\lambda}\abs{\widehat{\bal{A}{B}}(\lambda)}^2 \abs{\widehat{\mu_{ B^{(2)}}}(\gamma+\lambda)}^2\\
&=\abs{\widehat{\bal{A}{B}}}^2\circ \abs{\widehat{\mu_{B^{(2)}}}}^2(\gamma_1)+O(\rho_1d\alpha^{-1}\mu(B)^{-1}).
\end{align*}
Provided $\rho_1\leq c\alpha^2\delta^2\eta^3/d$, for some sufficiently small constant $c>0$, we therefore have, for any $\gamma\in \Delta'+\Gamma_1$, 
\[\abs{\widehat{\bal{A}{B}}}^2\circ \abs{\widehat{\mu_{B^{(2)}}}}^2(\gamma) \gs K^{-1}\delta\eta^4\alpha^{-2}\mu(B)^{-1},\]
and property (d) is satisfied.

To verify property (c), suppose that this fails and there exists some $\gamma$ such that
\[\ind{\Delta'}\ast \abs{\widehat{\mu_{B^{(2)}}}}^2(\gamma)=\sum_{\lambda\in \Delta'}\abs{\widehat{\mu_{B^{(2)}}}(\gamma-\lambda)}^2>2.\]
It follows by the pigeonhole principle (and recalling that $\abs{\Delta_j}\ls \eta^{-2}\alpha^{-1}$) that there must exist at least two $\lambda_1,\lambda_2\in \Delta'$ such that both $\gamma-\lambda_1$ and $\gamma-\lambda_2$ lie in $\Delta_\epsilon(B^{(2)})$ for some $\epsilon\gs \eta\alpha^{1/2}$, and in particular, 
\[\lambda_1-\lambda_2\in \Delta_\epsilon(B^{(2)})- \Delta_\epsilon(B^{(2)}).\]
Provided $\rho_2\leq c\eta^3\alpha^2/d$, however, for some sufficiently small constant $c>0$, the right-hand side is contained inside $\Gamma_3$, which contradicts the $\Gamma_3$-orthogonality of $\Delta_j$. This contradiction shows that property (c) must hold.

Finally, to verify property (a), we note that trivially $\Delta'$ is $\Gamma_1$-orthogonal (since $\Delta'\subset \Delta_j$ which is $\Gamma_3$-orthogonal). The proof is complete, and we are in the first case of the lemma.
\end{proof}

Finally, we can combine the lemmas to prove Proposition~\ref{prop:big}, as stated at the beginning of this section. 
\begin{proof}[Proof of Proposition~\ref{prop:big}]
We will need a few levels of nested Bohr sets in this proof, and introduce the following shorthand. We write $B^{(0)}$ for $B'$. For higher indices, $B^{(i+1)}=B^{(i)}_{\rho_{i}}$ for some $\rho_{i}$ which will be described on its introduction, but in particular always chosen so that $B^{(i+1)}$ is regular. For brevity, we also use $K$ to denote $\alpha^{-1/k}$. We also note now that if $B=\mathrm{Bohr}_{\nu}(\Gamma)$ is a regular Bohr set then $2\cdot B$ (as a set) is the Bohr set $\mathrm{Bohr}_{\nu'}(\tfrac{1}{2}\cdot \Gamma)$, where $\nu'(\tfrac{1}{2}\gamma)=\nu(\gamma)$.\footnote{By $\tfrac{1}{2}\gamma$ we mean the character which maps $x\mapsto \gamma(x/2)$, which is a character since $G$ has odd order.} In particular it has the same size and rank, and for any $\lambda\in(0,1)$, we have $2\cdot B_\lambda=(2\cdot B)_\lambda$. Furthermore, if $B'\subset B_\rho$, then $2\cdot B'\subset B_{2\rho}$ by the triangle inequality.

By Lemma~\ref{lemma:progtofourier} either we are in the second case or there is some $\eta\gg \alpha$ such that
\[\sum_{\gamma\in\Delta_{\eta}(2\cdot A')}\abs{\widehat{\bal{A}{B}}(\gamma)}^2\gtrsim_\alpha \eta^{-1}\mu(B)^{-1}.\]

Suppose first that this is true for some $\eta \geq \tfrac{1}{2} K^{-1}$. In this case, we apply Corollary~\ref{cor:massdim} with $2\cdot B'$ in place of $B$, the function $f$ chosen to be $1_{2\cdot A'}$, with $2\cdot B^{(1)}$ in place of $B'$, with $\rho_0=c\alpha^2/d$ for some suitably small constant $c>0$, and the weight function $\omega$ given by $\omega=\abs{\widehat{\bal{A}{B}}}^2$ restricted to $\Delta_{\eta}(2\cdot A')$. This produces some $\Delta$ of dimension $\dim(\Delta;\Delta_{1/2}(2\cdot B^{(1)}))\ls_\alpha K$ such that
\[\sum_{\gamma\in\Delta} \abs{\widehat{\bal{A}{B}}(\gamma)}^2 \gtrsim_\alpha \mu(B)^{-1}.\]
By Lemma~\ref{lemma:dimcovering} the set $\Delta$ is $\tilde{O}_\alpha(K)$-covered by $\Delta_{1/2}(2\cdot B^{(1)})$.  
Lemma~\ref{lemma:L2inc} then implies that $A$ has a density increment of strength $[1,K;\tilde{O}_\alpha(1)]$ relative to $2\cdot B^{(1)}$. Since $2\cdot B^{(1)}$ has the same rank as $B'$, and $\Abs{2\cdot B^{(1)}}=\Abs{B^{(1)}}\geq (\rho_0/4)^{d}\abs{B'}$ by Lemma~\ref{lemma:bohrsiz}, this implies that $A$ has a density increment of strength $[1,K;\tilde{O}_\alpha(1)]$ relative to $B'$, and we are in the third case.

Suppose now that $\tfrac{1}{2} K^{-1}\geq \eta\geq K^2\alpha$. By Corollary~\ref{cor:massdimboot} with $\delta=\eta K$, and otherwise the same inputs as in the previous case, there is a set $\Delta$ such that $\dim(\Delta;\Delta_{1/2}(2\cdot B_1))\lesssim_\alpha \eta^{-1}K\leq K^{-1}\alpha^{-1}$ and 
\[\sum_{\gamma\in\Delta}\abs{\widehat{\bal{A}{B}}(\gamma)}^2\gtrsim_\alpha K\mu(B)^{-1}.\]
As in the previous case, an application of Lemma~\ref{lemma:dimcovering} followed by Lemma~\ref{lemma:L2inc} implies that $A$ has a density increment of strength $[K,K^{-1}\alpha^{-1};\tilde{O}_\alpha(1)]$ relative to $2\cdot B_1$, and hence again relative to $B'$, and we are in the third case.

We may therefore assume that $\alpha\ll \eta \leq K^2\alpha$, and so in particular,
\[\sum_{\gamma\in\tilde{\Delta}} \abs{\widehat{\bal{A}{B}}(\gamma)}^2\gtrsim_\alpha K^{-2}\alpha^{-1}\mu(B)^{-1},\]
where $\tilde{\Delta} = \Delta_{c\alpha}(2\cdot A')$ for some absolute constant $c>0$.

We fix $\rho_1=c_1\alpha/d$ and let $\rho_2$ be chosen such that $B^{(2)}$ is a regular Bohr set and between $\Gamma_{\mathrm{top}}=\Delta_{1/2}(2\cdot B^{(2)})$ and $\Gamma_{\mathrm{bottom}}=\Delta_{1/2}(2\cdot B^{(1)})$ there is an additive framework of height $h$ and tolerance $t$, as produced by Lemma~\ref{lemma:AF1}, which ensures that such a $\rho_1$ can be chosen satisfying $\rho_1\geq (c_1/td)^{4h}$ for some absolute constant $c_1>0$. 

In this case we apply Lemma~\ref{lemma:newone} with $\Delta$ replaced by $\tilde{\Delta}$, $\eta$ being some suitably small constant multiple of $\alpha$, $A'$ being replaced by $2\cdot A'$, $B'$ being replaced by $2\cdot B'$, $\delta$ satisfying $\delta \gs_\alpha K^{-2}$, the $K$ in the statement of Lemma~\ref{lemma:newone} being replaced by $K^6$, $B^{(1)}$ being replaced by $2\cdot B^{(2)}$, and $B^{(2)}$ being replaced by $2\cdot B^{(3)}$, where $\rho_2$ is some small constant multiple of $\alpha^7/K^{10}d$. Note in particular that $\delta \gs \alpha K^6$, whence the $\nu$ and $\nu'$ in the conclusion of Lemma~\ref{lemma:newone} are both $\gs 1$. 

This either produces a density increment for $A$ such that the third case of the proposition holds, or else produces some $\Delta$ which satisfies most of the conditions of the final case of the lemma (with $B''$ being $2\cdot B^{(3)}$). It remains to show that we can assume that $\Delta$ is $\tfrac{1}{4}$-robustly $(\tau,k')$-additively non-smoothing relative to $\widetilde{\Gamma}$ for some suitable $\tau$ and $k'$.

Let $\Delta'\subset \Delta$ be of size $\abs{\Delta'}\geq \tfrac{1}{4}\abs{\Delta}$. We will show that, with a suitable choice of $\tau$ and $k'$ (independent of $\Delta'$) the set $\Delta'$ is $(\tau,k')$-additively non-smoothing. By construction, $\Delta$ is $\Gamma_{\mathrm{top}}$-orthogonal, and hence certainly $\Delta'$ is. Furthermore, by Lemma~\ref{lemma:energylower}, provided $\rho_0\leq c_1\alpha/d$ for some absolute constant $c_1>0$, if $\Delta''\subset \Delta'\subset \Delta_{c\alpha}(2\cdot A')$ then 
\begin{equation}\label{eq:enlower}
E_4( \Delta''; \Gamma_{\mathrm{bottom}}) \gg \alpha^5 \abs{\Delta''}^4\gs_\alpha \frac{\alpha^{2+O(1/k)} }{\abs{\Delta'}}\abs{\Delta''}^4.
\end{equation}

In particular, we can choose $\tau=\alpha^{2+C/k}$ for some absolute constant $C>0$ such that the second condition of additive non-smoothing is met (note that we can assume that $\log(1/\alpha)\leq \alpha^{-1/k}$ or else we are in the first case of the proposition). Let
\[e_{2m}=\norm{\ind{\Delta'}^{(m)}\circ \ind{\Delta'}^{(m)}\circ \abs{\widehat{\mu_{2\cdot B^{(3)}}}}^2}_\infty \abs{\Delta'}^{1-2m}.\]
We will show that we can assume that 
\[e_4 \leq \alpha^{2-O(1/k)},\quad e_6\leq \alpha^{4-O(1/k)},\quad\textrm{and}\quad e_8\leq \alpha^{6-O(1/k)}.\]
In particular, since $\Gamma_{\mathrm{top}}\subset \Delta_{1/2}(B^{(3)})$, it follows that for $n\in \{2,3,4\}$ we have
\[\norm{\ind{\Delta'}^{(n)}\circ \ind{\Delta'}^{(n)}\circ \ind{\Gamma_{\mathrm{top}}}}_\infty\ll \alpha^{2(n-1)-O(1/k)} \abs{\Delta'}^{2n-1}.\]
It follows from this that there exists some $k\geq k'\gg k$ such that, with $\tau$ chosen as above, the upper bounds on the energies required by Definition~\ref{def2-ans} are all satisfied, and hence $\Delta'$ is $(\tau,k')$-additively non-smoothing as required. (Note that we can assume that $\log(1/\tau)\leq \tau^{-1/k}$ or else we are in the first case of the proposition.)

It remains to prove the stated upper bounds for $e_4$, $e_6$, and $e_8$. We will show that if these fail then we have a density increment for $A$. We note first that for any $\lambda$,
\[\ind{\Delta'}^{(m)}\circ \ind{\Delta'}^{(m)}\circ \abs{\widehat{\mu_{2\cdot B^{(3)}}}}^2(\lambda) = \bbe_x \mu_{2\cdot B^{(3)}}\circ \mu_{2\cdot B^{(3)}}(x)\abs{\widecheck{\ind{\Delta'}}(x)}^{2m}\lambda(x).\] 
In particular, by the triangle inequality, the function on the left-hand side is maximised when $\lambda$ is the trivial character, and so
\[e_{2m}\abs{\Delta'}^{2m-1} = \bbe_x \mu_{2\cdot B^{(3)}}\circ \mu_{2\cdot B^{(3)}}(x)\abs{\widecheck{\ind{\Delta'}}(x)}^{2m}.\]
In particular,
\[e_2\abs{\Delta'} =\Inn{\ind{\Delta'}\circ \ind{\Delta'}, \abs{\widehat{\mu_{2\cdot B^{(3)}}}}^2}\leq\abs{\Delta'} \norm{\ind{\Delta'}\circ \abs{\widehat{\mu_{2\cdot B^{(3)}}}}^2}_\infty\leq 2\abs{\Delta'}.\]
We now use H\"{o}lder's inequality to see that for any measure $\mu$ and $f:G\to \bbc$, and $1\leq n\leq m$, 
\[\bbe_x \mu(x)\abs{f(x)}^{2n}\leq \brac{ \bbe_x \mu(x)\abs{f(x)}^2}^{\frac{m-n}{m-1}}\brac{ \bbe_x \mu(x) \abs{f(x)}^{2m}}^{\frac{n-1}{m-1}}.\]
In particular, combined with the fact that
\[\abs{\Delta'}^{2n-1} =  \abs{\Delta'}^{\frac{m-n}{m-1}}\brac{\abs{\Delta'}^{2m-1}}^{\frac{n-1}{m-1}},\]
it follows that
\[e_{2n}^{m-1} \leq e_2^{m-n}e_{2m}^{n-1}.\]
Let $L$ be some parameter to be chosen later (but which will be of the shape $L=\alpha^{-C'/k}$ for some absolute constant $C'>0$). Suppose that any of $e_4\geq L\alpha^2$, or $e_6\geq L\alpha^4$, or $e_8\geq L\alpha^6$ hold. Using the fact that $e_2\leq 2$, any of these imply that, for any $m\geq 5$,
\[e_{2m}\geq (\tfrac{1}{2}L^{\frac{1}{3}}\alpha^2)^{m-1}.\]
In particular, using the fact that $\abs{\Delta'}\ls_\alpha \alpha^{-3}$, 
\[E_{2m}(\Delta';\abs{\widehat{\mu_{2\cdot B^{(3)}}}}^2) \gs_\alpha \alpha^{3} (\tfrac{1}{2}L^{\frac{1}{3}}\alpha^2)^{m-1}\abs{\Delta'}^{2m}.\]
Let $T$ be some parameter to be chosen later, and $\mu=\mu_{(2\cdot B^{(3)})_{1+T\rho_3}}\ast \mu_{2\cdot B^{(4)}}^{(T)}$, where $\rho_3=c/Td$ for some small constant $c>0$. By Lemma~\ref{lemma:fourierbohr} we have $\mu_{2\cdot B^{(3)}}\leq 2 \mu$, and hence 
\[E_{2m}(\Delta';\abs{\widehat{\mu_{2\cdot B^{(4)}}}}^{2T}) \gs_\alpha \alpha^{3} (\tfrac{1}{2}L^{\frac{1}{3}}\alpha^2)^{m-1}\abs{\Delta'}^{2m}.\]
In particular, provided we choose $T=Cm\lceil \log(L/\alpha)\rceil$ for some suitably large constant $C>0$, we have \[E_{2m}(\Delta';\Delta_{1/2}(2\cdot B^{(4)})) \gs_\alpha \alpha^{3} (\tfrac{1}{2}L^{\frac{1}{3}}\alpha^2)^{m-1}\abs{\Delta'}^{2m}.\]

We will now apply Lemma~\ref{lemma:energytodimension} to $\omega=\ind{\Delta'}$ and $\Gamma=\Delta_{1/2}(2\cdot B^{(4)})$, with $m=C\lceil \log(2/\alpha)\rceil$ and $\ell=CL^{-\frac{1}{6}}\alpha^{-1}$ for some large absolute constants $C>0$ (note that since we will choose $L=\alpha^{-O(1/k)}$ we have $\norm{\omega}_1\gs \ell$). In particular, provided we choose $C$ sufficiently large, the second case of Lemma~\ref{lemma:energytodimension} cannot hold, and hence there exists some $\Delta''\subset \Delta'$ such that
\[\abs{\Delta''}\gs_\alpha L^{1/6}\alpha^{-2+O(1/k)}\textrm{ and }\dim(\Delta'';\Delta_{1/2}(2\cdot B^{(4)}))\ls_\alpha L^{-1/6}\alpha^{-1}.\]
We now apply Lemma~\ref{lemma:lowdiminc} which produces a density increment of strength 
\[[L^{1/6}\alpha^{O(1/k)},L^{-1/6}\alpha^{-1};\tilde{O}_\alpha(1)]\]
relative to $B^{(4)}$. If we choose $L=\alpha^{-C'/k}$ for some suitably large absolute constant $C'$ then this in particular produces a density increment of strength $[\alpha^{-1/k},\alpha^{-1+1/k};\tilde{O}_\alpha(h\log t)]$ relative to $B'$. Thus either we have produced a large density increment, or else $e_4\leq L\alpha^2$, $e_6\leq L\alpha^4$, and $e_8\leq L\alpha^6$. In particular, the required upper bounds for the energies of $\Delta'$ hold, and so $\Delta'$ is $(\tau,k')$-additively non-smoothing for some $k\geq k'\gg k$, and the proof is complete.
\end{proof}

\section{Structure of non-smoothing sets}\label{section:structure1} 

In this section we prove the key structural result about sets with near-optimal relationships between their additive energies, so-called additively non-smoothing sets. Before the statement we give an intuitive idea of the kind of result we are after. The first such structural result was proved by Bateman and Katz~\cite{BaKa:2012}, and both the philosophy and methods of that paper have heavily influenced our approach.

\subsection*{Additively non-smoothing sets}
Consider a set $\Delta$ (in any finite abelian group, for this sketch). Suppose that 
\[E_4(\Delta)=\Inn{\ind{\Delta}\circ \ind{\Delta},\ind{\Delta}\circ \ind{\Delta}}=\tau \abs{\Delta}^3.\]
Two applications of the Cauchy--Schwarz inequality\footnote{The first applied to $\Inn{\ind{\Delta}\ast\ind{\Delta}\circ \ind{\Delta},\ind{\Delta}}$ and the second to $\Inn{\ind{\Delta}^{(2)}\circ \ind{\Delta}^{(2)},\ind{\Delta}\circ\ind{\Delta}}$.} imply that\[
E_8(\Delta)=\Inn{\ind{\Delta}^{(4)}, \ind{\Delta}^{(4)}}
\geq \tau^3\abs{\Delta}^7.\]
A structural result for additively non-smoothing sets gives structural information about sets where this lower bound is (almost) sharp. Note that we are making no assumptions about the size of $\tau$ relative to $\Delta$ (aside from the trivial bounds $1\geq \tau \geq \abs{\Delta}^{-1}$). This is in contrast to much of additive combinatorics, which tends to work in the regime where $\tau\gg 1$. For our application, when $\Delta$ is a subset of a spectrum, we are not in this regime, since then we expect $\tau \approx \abs{\Delta}^{-2/3}$. The power of the non-smoothing approach of Bateman and Katz is that it allows us to make strong structural statements, even when the additive energy is very small.

To see what kind of conclusion we expect, consider the following two examples of sets:
\[\Delta_1=H\oplus D\quad\textrm{and}\quad \Delta_2=\bigsqcup_{i=1}^L H_i,\]
where $H$ and $H_i$ are subgroups (and the $H_i$ are all the same size, say $\abs{H_i}\approx K$), and $D$ is `$H$-dissociated' in the sense that there are no non-trivial additive relations between elements of $D$ and $H$. For $\Delta_1$, we have $\abs{\Delta_1}\approx \abs{H}\abs{D}$ and we expect that $\Delta_1-\Delta_1=H+ D- D$ and so $\abs{\Delta_1-\Delta_1}\approx \abs{\Delta_1}\abs{D}$. On $H$, we have $\ind{\Delta_1}\circ \ind{\Delta_1}\approx \abs{\Delta_1}$, and on the rest of $\Delta_1-\Delta_1$, we have $\ind{\Delta_1}\circ \ind{\Delta_1}\approx \abs{H}$. Therefore
\[E_4(\Delta_1) \approx \abs{H}\abs{\Delta_1}^2 + \abs{\Delta_1-\Delta_1}\abs{H}^2\approx \frac{1}{\abs{D}}\abs{\Delta_1}^3.\]
In particular, if $\abs{H}\approx \tau\abs{\Delta_1}$ then $\abs{D}\approx \tau^{-1}$ and $\Delta_1$ has $E_4(\Delta_1)\approx \tau\abs{\Delta_1}^3$. Moreover, a similar calculation shows that
\[E_8(\Delta_1)\approx \abs{H}(\abs{H}^3\abs{D}^2)^2 + \abs{\Delta_1+\Delta_1-\Delta_1-\Delta_1}\abs{H}^6\approx \tau^3\abs{\Delta_1}^7.\]

In particular, $\Delta_1$ is additively non-smoothing. For $\Delta_2$, on the other hand, assuming the $H_i$ are `spread out' enough that they do not additively interact much with each other, 
\[\Delta_2-\Delta_2 = \bigsqcup_{i=1}^L H_i \cup \bigsqcup_{1\leq i\neq j\leq L} (H_i-H_j).\]
On the first part, which has size $\abs{\Delta_2}$, we have $\ind{\Delta_2}\circ \ind{\Delta_2}\approx K$. On the second part, which has size $\approx L^2K^2\approx \abs{\Delta_2}^2$, we have $\ind{\Delta_2}\circ \ind{\Delta_2}\approx 1$. Therefore
\[E_4(\Delta_2)\approx K^2 \abs{\Delta_2} +\abs{\Delta_2}^2.\]
In particular, if $K\approx \tau^{1/2}\abs{\Delta_2}$ and $L\approx \tau^{-1/2}$, the second term is negligible, and then $E_4(\Delta_2)\approx \tau\abs{\Delta_2}^3$. Similarly, 
\[E_8(\Delta_2)\approx K^6 \abs{\Delta_2} +\abs{\Delta_2}^4\approx \tau^3\abs{\Delta_2}^7,\]
and hence $\Delta_2$ is also additively non-smoothing.

Note that $\Delta_1$ and $\Delta_2$, although highly structured sets, have qualitatively different kinds of structure. The former is the union of $\approx \tau^{-1}$ many cosets, each of which is a translate of the same subgroup, while the latter is the union of $\approx \tau^{-1/2}$ many cosets, each of which comes from a different subgroup, which do not interact much.

The philosophy behind the structural results for additive non-smoothing sets is that these two kinds of structure (and natural interpolations between the two) are the only ways that a set can be additively non-smoothing. This applies, quite crucially, whatever the size of $\tau$. The important thing is that the ratio $(E_8(\Delta)/\abs{\Delta}^7)/(E_4(\Delta)/\abs{\Delta}^3)^3$ is small, not the size of the energies themselves.

To motivate the form our structural theorem takes, note that in both $\Delta_1$ and $\Delta_2$ there is a set $X\subset \Delta_i$ and a subgroup $H\subset \Delta_i$ such that $\abs{X}\abs{H}\approx \tau \abs{\Delta_i}^2$ and $E(X,H)\gg \abs{X}\abs{H}^2$. Indeed, for $\Delta_1$ we take $X=\Delta_1$ and $H$ to be the $H$ in its construction, and for $\Delta_2$ we take $X=H=H_i$ for some arbitrary $1\leq i\leq L$. 

This is the type of structural result we will prove: we show that for any additively non-smoothing set we can find such $X$ and $H$. It is possible to then apply some further techniques of additive combinatorics, such as the asymmetric Balog--Szemer\'{e}di--Gowers lemma, and use the fact that $E(X,H)$ is near-maximal, to deduce further, more rigid, structural properties of $\Delta$. Indeed, this is the approach taken by Bateman and Katz. We do not pursue this in this paper, as this simpler energy form is more flexible, and in particular it gives a conclusion which naturally leads to the spectral boosting method.

\subsection*{The statement of the structural theorem}

The following is a precise form of the structural result described above. Since we cannot obtain information on the actual additive energy of spectra, but only on their energy relative to spectra of Bohr sets, we need a structural result that is flexible enough to apply to relative energies. For this we need the objects we are taking energy relative to to themselves be highly structured, which is captured by the definition of an additive framework (Definition~\ref{not-frame}). 

\begin{theorem}\label{th:structure}
There is a constant $C>0$ such that the following holds. Let $h,t,k\geq 2$ and $\tau\leq 1/2$ be some parameters, and let $\widetilde{\Gamma}$ be an additive framework of height $h$ and tolerance $t$. Suppose further that $h\leq C^{-1}\log\log k/\log\log\log k$ and $t\geq C\log k$.

If $\Delta$ is any set which is $\tfrac{1}{2}$-robustly $(\tau,k)$-additively non-smoothing relative to $\widetilde{\Gamma}$, then there are $X,H\subset \Delta$ and some $\delta$ with $1\geq \delta \gg \tau$ such that 
\begin{enumerate}
\item 
\[\abs{H}\asymp \delta \abs{\Delta}\quad\textrm{and}\quad\abs{X}\asymp \tau \delta^{-1}\abs{\Delta},\]
\item 
\[\Inn{ \ind{X}\circ \ind{X}, \ind{H}\circ \ind {H}\circ \ind{\Gamma_{\mathrm{top}}}}\gg \abs{H}^2\abs{X},\]
\listintertext{and}
\item for some $z$ 
\[H+z\subset \{ x : \ind{X}\circ \ind{X}\circ \ind{\Gamma_{\mathrm{top}}}(x) \gg \abs{X}\}.\]
\end{enumerate}
All implicit error bounds are polynomial in $t^h2^{k}\tau^{-\frac{1}{\log\log k}-\frac{1}{h}}$. 
\end{theorem}

The third component of Theorem~\ref{th:structure} is somewhat superfluous, in the sense that it can be easily deduced from the other parts of the theorem, but we have found it convenient to include it all in the same theorem statement. 

We stress that Theorem~\ref{th:structure} is valid for sets within any finite abelian group (although we apply it on the dual side to subsets of $\widehat{G}$). The proof is entirely elementary and `physical' -- that is, at no point do we apply the Fourier transform or its inverse, and instead repeatedly apply the Cauchy--Schwarz inequality and the pigeonhole principle. 

We will apply this theorem with $h\asymp \log\log k/\log\log \log k$ and $t\asymp \log k$, and so the implicit constants are polynomial in $2^k \tau^{-\frac{\log\log\log k}{\log\log k}}$. Part of the poor quality of this bound is due to the need to work with the approximate nature of an additive framework. If one carried out the proof that follows in the non-relative case, when the additive framework is trivial and all the associated $\Gamma^{(i)}$ are just $\{0\}$, then the constants would be polynomial in $k\tau^{-\frac{1}{\log k}}$. This is unlikely to be optimal -- it is natural to conjecture that the result should hold with constants bounded by a polynomial in $k\tau^{-1/k}$. This would have little effect on our final result, however, without further refinements to the rest of the methods used in this paper.

The proof of this theorem is quite delicate, even without the need to work with relative energies throughout. We therefore begin with a sketch of the argument in the non-relative case (when all involved $\Gamma^{(i)}$ are $=\{0\}$) for orientation purposes.
 
\subsection*{A sketch proof of the non-relative case}
As this is a sketch, we will be deliberately vague with notation, and make liberal use of the $\ll$ and $\approx$ notation to hide various constants and logarithmic factors. We begin with a set $\Delta$ with energy
\[E_4(\Delta)=\Inn{\ind{\Delta}\circ \ind{\Delta},\ind{\Delta}\circ \ind{\Delta}}\approx \tau \abs{\Delta}^3,\]
for which the higher energy $E_8$ is almost as small as possible, relative to the $E_4$ energy, so that $E_8(\Delta)\ll \tau^3\abs{\Delta}^7$. Our goal is to find some $X,H\subset \Delta$ such that $\abs{X}\abs{H}\approx \tau\abs{\Delta}^2$ and $\Inn{\ind{X}\circ \ind{X},\ind{H}\circ \ind{H}}\gg \abs{X}\abs{H}^2$.

Our proof naturally separates into two steps:
\begin{enumerate}
\item (Finding a structured piece) Find, for some $\delta$ with $1\geq \delta\gg \tau$, some $\Delta'\subset \Delta$ of size $\abs{\Delta'}\approx \delta\abs{\Delta}$ such that there are $\gg \tau\delta^{-1}\abs{\Delta}$ many $x\in \Delta-\Delta$ such that $\ind{\Delta'}\circ \ind{\Delta}(x)\approx \abs{\Delta'}$. Note that this is much stronger than just the lower bound $\ind{\Delta}\circ\ind{\Delta}(x)\gg \delta\abs{\Delta}$.
\item (Gluing structured pieces together) Extract many disjoint such structured $\Delta'$ and `glue' them together to find the requisite structure -- in the end, our $X$ will be some suitable union of these structured pieces, and $H$ will be a structured subset of one of these pieces. 
\end{enumerate}

The non-smoothing assumption that $E_8(\Delta)\ll \tau^3\abs{\Delta}^7$ plays an important and necessary role in \emph{both} steps of the argument, and even the first step can fail dramatically for sets which are additively smoothing. For example, if $\Delta$ is just a random subset of $G$ of density $\approx \tau$ (so its energy is $\approx \tau\abs{\Delta}^3$ as required) then, with high probability, for any reasonably sized $\Delta'\subset \Delta$ we have $\ind{\Delta'}\circ \ind{\Delta}(x) \approx \tau\abs{\Delta'}$ for all $x\in G\backslash \{0\}$.

Our main tool for locating `structured subsets' will be to pass from a set $\Delta$ to an intersection $\Delta\cap (\Delta+a)$ for some suitable translate $a$. This is a common manouevre in additive combinatorics -- for example, it is a key step in Schoen's proof \cite{Sc:2015} of the current best bounds in the Balog--Szemer\'{e}di--Gowers lemma.

\subsection*{Step one: Finding a structured piece} 
Our goal here is to find some `structured piece' $\Delta'\subset \Delta$ of size $\abs{\Delta'}\approx \delta\abs{\Delta}$ such that there are $\gg \tau\delta^{-1}\abs{\Delta}$ many $x\in \Delta-\Delta$ such that $\ind{\Delta'}\circ \ind{\Delta}(x)\approx \abs{\Delta'}$. As mentioned above, this $\Delta'$ will take the form $\Delta'=\Delta\cap (\Delta-a)$ for some suitably chosen $a$. The strategy is to begin by pigeonholing to find many $a\in \Delta-\Delta$ where $\Delta\cap(\Delta-a)$ is of some controlled size. Either one of these $a$ yields a suitable set, or else the non-smoothing hypothesis combined with the Cauchy--Schwarz inequality allows us to find an even larger subset of $\Delta-\Delta$ with suitable properties. This procedure cannot continue forever, and so at some point we can exit with a suitable $a$. At this point it is best to let the calculations speak for themselves.

We first note that by dyadic pigeonholing on $E_4(\Delta)=\sum_{x} \ind{\Delta}\circ \ind{\Delta}(x)^2$ there exists some $\delta$ with $1\geq \delta\gg \tau$ such that $\abs{S}\approx \tau\delta^{-2}\abs{\Delta}$, where $S=\{ x: \ind{\Delta}\circ \ind{\Delta}(x)\approx \delta\abs{\Delta}\}$. It follows that
\[\Inn{\ind{\Delta}, \ind{S}\ast \ind{\Delta}}=\Inn{\ind{S},\ind{\Delta}\circ\ind{\Delta}}\gg \delta\abs{\Delta}\abs{S}.\]
By the Cauchy--Schwarz inequality,
\[\Inn{ \ind{\Delta}, (\ind{S}\ast \ind{\Delta})^2}\gg \delta^2\abs{\Delta}\abs{S}^2.\]
Expanding out the left-hand side gives
\[\sum_{a,b\in S}\sum_{x\in \Delta}\ind{\Delta}(x-a)\ind{\Delta}(x-b).\]
We now let $F(a,b)$ denote the inner sum and dyadically pigeonhole again over the pairs $(a,b)$, so that we find some $\eta$ with $1\geq \eta \gg \delta^2$ and $G\subset S\times S$ on which $F(a,b)\approx \eta \abs{\Delta}$, and $\abs{G}\approx \eta^{-1}\delta^2\abs{S}^2$. Note that since $F(a,b)\leq \ind{\Delta}\circ\ind{\Delta}(a)$ and $a\in S$ we actually have $\eta\ll \delta$. 

If $\eta$ is significantly smaller than $\delta$, then we consider $D=\{ a-b : (a,b) \in G\}$, and note that, since $F(a,b)\leq \ind{\Delta}\circ \ind{\Delta}(a-b)$, if $x\in D$ then $\ind{\Delta}\circ \ind{\Delta}(x)\gg \eta \abs{\Delta}$. Furthermore,
\[\Inn{\ind{S}\circ \ind{S}, \ind{D}}\geq \abs{G}\gg \eta^{-1}\delta^2\abs{S}^2.\]
By the definition of $S$ and the fact that $E_8(\Delta)\ll \tau^3\abs{\Delta}^7$ we have
\[\Inn{\ind{S}\circ \ind{S},\ind{S}\circ \ind{S}}\ll (\delta\abs{\Delta})^{-4}E_8(\Delta)\ll \tau^3\delta^{-4}\abs{\Delta}^3,\]
and so, by the Cauchy--Schwarz inequality, $\abs{D}\gg \tau \eta^{-2}\abs{\Delta}$. Using the upper bound $E_4(\Delta)\ll \tau\abs{\Delta}^3$ it follows that, on at least half of $D$, say, we must have $\ind{\Delta}\circ \ind{\Delta}(x) \ll \eta \abs{\Delta}$. We have therefore found some $D$ such that $\abs{D}\approx \tau \eta^{-2}\abs{\Delta}$ and if $x\in D$ then $\ind{\Delta}\circ \ind{\Delta}(x)\approx \eta\abs{\Delta}$.

This is the same kind of data that we began the argument with, with $D$ replacing $S$, except that $\delta$ has been replaced by the smaller $\eta$. We now iterate the entire argument from the beginning, until we find some $\delta$ where the corresponding $\eta$ satisfies $\eta\approx \delta$. This must happen eventually, and the argument must terminate after a reasonable number of steps, since the $\delta$ parameter cannot significantly decrease indefinitely. Indeed, since we also have the trivial estimate
\[\delta\abs{\Delta}\abs{S}\ll \Inn{\ind{S}, \ind{\Delta}\circ \ind{\Delta}}\leq \abs{\Delta}^2,\]
which implies the upper bound $\abs{S}\ll \delta^{-1}\abs{\Delta}$, the fact that $\abs{S}\approx \tau\delta^{-2}\abs{\Delta}$ at each stage of the argument means that we always have $\delta \gg \tau$. 

In the exit case when $\delta\approx \eta$ we have some $G\subset S\times S$ of size $\abs{G}\gg \delta\abs{S}^2$ such that $F(a,b)\approx \delta\abs{\Delta}$ if $(a,b)\in G$. In particular there exists some $a\in S$ and $S'\subset S$ of size $\abs{S'}\approx \tau\delta^{-1}\abs{\Delta}$ such that, if we let $\Delta'=\Delta\cap(\Delta+a)$, then $\abs{\Delta'}\approx \ind{\Delta}\circ \ind{\Delta}(a)\approx \delta\abs{\Delta}$ and, since $\ind{\Delta'}\circ \ind{\Delta}(b)=F(a,b)\approx \delta\abs{\Delta}$,
\[\Inn{\ind{S'},\ind{\Delta'}\circ \ind{\Delta}}\approx\abs{\Delta'}\abs{S'}\approx \tau\abs{\Delta}^2.\]
In particular, for $\gg \abs{S'}\gg \tau\delta^{-1}\abs{\Delta}$ many $x$ we have $\ind{\Delta'}\circ \ind{\Delta}(x)\gg \abs{\Delta'}$ as required.

\subsection*{Step two: Gluing structured pieces together}

The next step is to remove this structured piece $\Delta'$ from $\Delta$, and repeat the argument to find another structured piece in $\Delta\backslash \Delta'$ and so on, until we have some collection of structured pieces $\Delta_i$ whose (disjoint) union has size $\gg \abs{\Delta}$. The robustness of our initial energy assumptions implies that the relevant hypotheses continue to hold (with possibly slightly smaller constants), and so this is possible. We will then use the Cauchy--Schwarz inequality combined with the non-smoothing assumption to explore how these pieces interact with each other. A pigeonhole argument will then show that there is some suitable union of the structured pieces, say $X=\sqcup_{i\in I}\Delta_i$, and some choice of translates, say $a_i$ for $i\in I$, such that with $H=\Delta_j\cap (\Delta_j+a_j)$ for some $j$, we have $\abs{X}\abs{H}\approx \tau\abs{\Delta}^2$ and $E(X,H)\gg \abs{X}\abs{H}^2$, as required.

We begin by transforming the `structural information' about $\Delta'$ obtained in the previous step, namely
\[\Inn{\ind{S'},\ind{\Delta'}\circ \ind{\Delta}}\approx\abs{\Delta'}\abs{S'}\approx \tau\abs{\Delta}^2,\]
into a different form. By the Cauchy--Schwarz inequality,
\[\sum_{a,b\in \Delta'} \sum_{x\in S'} \ind{\Delta}(x-a)\ind{\Delta}(x-b)\gg \abs{\Delta'}^2\abs{S'}.\]
Simple pigeonholing shows that this sum must be concentrated where the inner sum is $\approx \abs{S'}\approx \tau \delta^{-1}\abs{\Delta}$, and thus since the inner sum is also at most $\ind{\Delta}\circ \ind{\Delta}(a-b)$ we have shown that
\[\Inn{\ind{\Delta'}\circ\ind{\Delta'}, \ind{T}}\gg \abs{\Delta'}^2,\]
where 
\[T = \{ x : \ind{\Delta}\circ \ind{\Delta}(x) \gg \tau\delta^{-1}\abs{\Delta}\}.\]
This is the form of structure we will use in our gluing procedure. (The reader may reasonably ask at this point why we did not use this as our benchmark for `structure' in the previous step, since it is all this is used in what follows. The answer is that the notion of structure produced in the first step is more natural and easy to describe, even more so when we come to the full proof in the relative case, and so it provides a much more sensible point at which to break the proof into two separate stages.)

We now remove $\Delta'$ from $\Delta$ and repeat the entire argument with $\Delta\backslash \Delta'$, provided this is at least half of $\Delta$ still. Continuing in this manner, we arrive at disjoint $\Delta_1,\ldots,\Delta_K\subset \Delta$, each of size $\abs{\Delta_i}\approx \delta_i\abs{\Delta}$, and an associated $T_i$ such that $T_i\subset \{ x : \ind{\Delta}\circ \ind{\Delta}(x) \gg \tau\delta_i^{-1}\abs{\Delta}\}$ and 
\[\Inn{\ind{\Delta_i}\circ \ind{\Delta_i},\ind{T_i}}\gg \abs{\Delta_i}^2.\]
By dyadic pigeonholing yet again, we can assume that there is some $\delta$ along with $\gg \delta^{-1}$ many $i$ for which $\delta_i\approx \delta$, and thus all the $T_i$ are the same set, say $T$.

Now, for any such $i$, 
\[\Inn{\ind{T}\ast \ind{\Delta_i},\ind{\Delta_i}}=\Inn{ \ind{T}, \ind{\Delta_i}\circ \ind{\Delta_i}}\gg \delta^2\abs{\Delta}^2,\]
by the Cauchy--Schwarz inequality
\[\sum_{x,y\in T}\sum_{a\in \Delta_i}\ind{\Delta_i}(a-x)\ind{\Delta_i}(a-y)\gg \delta^3\abs{\Delta}^3.\]
The left-hand side is
\[\sum_{x\in T}\Inn{ \ind{T}, \ind{\Delta_{i,x}}\circ \ind{\Delta_i}},\]
say, where $\Delta_{i,x}=\Delta_i\cap (\Delta_i+x)$. We now sum this over all $\delta^{-1}$ many $i$, so that
\[\sum_{x\in T}\Inn{ \ind{T},\sum_i \ind{\Delta_{i,x}}\circ \ind{\Delta_i}}\gg \delta^2\abs{\Delta}^3.\]
Let $F(x)=\sum_i \abs{\Delta_{i,x}}\ll \ind{\Delta}\circ \ind{\Delta}(x)$. The inner product is bounded above by $\delta\abs{\Delta}F(x)$. 

We claim that the contribution from those $x$ such that $F(x) \gg \tau\delta^{-1}\abs{\Delta}$ is negligible. Indeed, if the contribution from those $x$ such that $F(x) \geq C\tau\delta^{-1}\abs{\Delta}$ is $\gg \delta^2\abs{\Delta}^2$ then 
\[E_4(\Delta)\geq \sum_{x\in T} F(x)^2 \gg C \tau \abs{\Delta}^3,\]
which is a contradiction for large enough $C$. Thus we can restrict the sum to those $x\in T$ such that 
\[\Inn{ \ind{T},\sum_i \ind{\Delta_{i,x}}\circ \ind{\Delta_i}}\ll \tau\abs{\Delta}.\]
In fact, we can obtain a similar lower bound, using the fact that $\abs{T}\ll \tau^{-1}\delta^2\abs{\Delta}$ (which again follows from the fact that $E_4(\Delta)\ll \tau\abs{\Delta}^3$). Therefore we can restrict the sum to those $x$ such that the inner product is $\approx \tau\abs{\Delta}^2$. That is, we have $T'\subset T$ such that $\abs{T'}\approx \tau^{-1}\delta^2\abs{\Delta}^3$ and if $x\in T'$ then
\[\Inn{\ind{T}, \sum_i\ind{\Delta_{i,x}}\circ \ind{\Delta_i}}\approx \tau \abs{\Delta}^2.\]
By dyadic pigeonholing we can find, for any $x\in T'$, some $T_x\subset T$ and $\eta_x$ such that $\abs{T_x}\approx \delta \eta_x^{-1}\abs{\Delta}$ and
\[\sum_i\ind{\Delta_{i,x}}\circ \ind{\Delta_i}(y)\approx \eta_x \tau \delta^{-1}\abs{\Delta}\]
for $y\in T_x$. By another application of dyadic pigeonholing we can assume that all the $\eta_x$ are roughly equal, say $\eta_x\approx \eta$. 

We will now use the non-smoothing assumption to show that in fact $\eta \approx 1$. If $D$ is the set of all $x-y$ where $x\in T'$ and $y\in T_x$ then
\[\Inn{ \ind{T}\circ \ind{T}, \ind{D}}\gg \sum_{x\in T}\abs{T_x}\gg \eta^{-1}\tau^{-1}\delta^3\abs{\Delta}^2.\]
Note that if $x-y\in D$ then
\[\sum_i\ind{\Delta_i}\circ \ind{\Delta_i}(x-y)\geq \sum_i \ind{\Delta_{i,x}}\circ \ind{\Delta_i}(y)\gg \eta \tau\delta^{-1}\abs{\Delta},\]
and hence summing over all $x-y\in D$ (and using $\sum_i \abs{\Delta_i}^2 \ll \delta\abs{\Delta}$) yields 
\[\abs{D}\ll \eta^{-1}\tau^{-1}\delta^2\abs{\Delta}.\]
 and hence by the Cauchy--Schwarz inequality
\[E_4(T) \gg \eta^{-1}\tau^{-1}\delta^4\abs{\Delta}^3.\]
By definition of $T$, however,
\[E_4(T) \ll (\tau \delta^{-1}\abs{\Delta})^{-4}E_8(\Delta),\]
and so the non-smoothing assumption $E_8(\Delta)\ll \tau^3\abs{\Delta}^7$ forces $\eta \gg 1$ as claimed.

Using the fact that $\eta\approx 1$ we can, in particular, fix some $x\in T$ with an associated $T'\subset T$ of size $\abs{T'}\approx \delta\abs{\Delta}$ such that, with $\Delta_i'=\Delta_{i,x}$, for all $y\in T'$, 
\[\sum_i \ind{\Delta_i'}\circ \ind{\Delta_{i}}(y)\approx \tau \delta^{-1}\abs{\Delta}.\]
Note that, by the above, $\sum_i \abs{\Delta_i'}=F(x)\approx \tau \delta^{-1}\abs{\Delta}$. By the pigeonhole principle (and relabelling if necessary) there exists some $M$ such that for all $1\leq i\leq M$, 
\[\Inn{\ind{T'}, \ind{\Delta_i'}\circ \ind{\Delta_i}}\approx M^{-1}\tau \abs{\Delta}^2\]
and $\abs{\Delta_i'}\approx M^{-1}\tau\delta^{-1}\abs{\Delta}$. Summing over all $1\leq i\leq M$ and applying the Cauchy--Schwarz inequality,
\[\sum_{1\leq i,j\leq M} \Inn{\ind{\Delta_i}\ast \ind{\Delta_j'},\ind{\Delta_j}\ast \ind{\Delta_i'}}\gg \tau^2\delta^{-1}\abs{\Delta}^3.\]
In particular, using the trivial fact that 
\[\Inn{\ind{\Delta_i}\ast \ind{\Delta_j'},\ind{\Delta_j}\ast \ind{\Delta_i'}}\leq \abs{\Delta_i}\abs{\Delta_i'}\abs{\Delta_j'}\ll M^{-2}\tau^2\delta^{-1}\abs{\Delta}^3\]
for any $1\leq i,j\leq M$, there exists some $1\leq j\leq M$ and $\gg M$ many $i$ such that
\[\Inn{\ind{\Delta_i}\ast \ind{\Delta_j'},\ind{\Delta_j}\ast \ind{\Delta_i'}}\gg M^{-2}\tau^2\delta^{-1}\abs{\Delta}^3.\]
By the Cauchy--Schwarz inequality,
\[\Inn{\ind{\Delta_i}\ast \ind{\Delta_j'},\ind{\Delta_i}\ast \ind{\Delta_j'}}\gg M^{-2}\tau^2\delta^{-1}\abs{\Delta}^3.\]
Taking the union of all the $\Delta_i$, we have found some $H$ (namely $\Delta_j'$) and $X$ (the union of the $\Delta_i$) such that $\abs{H}\approx M^{-1}\tau \delta^{-1}\abs{\Delta}$, $\abs{X}\approx M\delta \abs{\Delta}$, so that $\abs{X}\abs{H}\approx \tau\abs{\Delta}$, and $E(X,H)\gg \abs{X}\abs{H}^2$. We have produced $X$ and $H$ as required, and the (sketch) proof of the structural result for additively non-smoothing sets is complete.
\subsection*{The proof of relative structure}
We begin the proof of Theorem~\ref{th:structure} by establishing the following simple, but crucial, lemma, which is a generalisation of the observation that for any set $A$
\[\sum_{x\in A} f(a+x)g(b+x)\leq f\circ g(a-b),\]
which we made repeated use of in the sketch proof. Recall that $A$ being $\Gamma'$-orthogonal is simply requiring that the translates $(a+\Gamma')_{a\in A}$ are all disjoint.

\begin{lemma}\label{lemma:collapser}
Let $\Gamma, \Gamma'$ be arbitrary sets, and suppose that $A$ is any $\Gamma'$-orthogonal set. Then for any $a,b$ and non-negative functions $f, g$
\[\sum_{x\in A}f\circ \ind{\Gamma}(a+x)g\circ \ind{\Gamma'}(b+x)\leq (f\circ g)\circ \ind{\Gamma-\Gamma'}(a-b).\]
\end{lemma}
\begin{proof}
The left-hand side we write as
\[\sum_{x\in A}\sum_{v\in \Gamma'}\sum_{u\in \Gamma}f(a+x+u)g(b+x+v).\]
For $t\in A+\Gamma'$ let $v_t$ be the $v$ such that $t=x+v$ with $x\in A$ and $v\in \Gamma'$, which is unique by orthogonality. We can then write the above as 
\[\sum_{t\in A+\Gamma'} \sum_{w\in \Gamma-v_t}f(a+t+w)g(b+t)\]
which is at most 
\begin{align*}
\sum_{t}\sum_{w\in \Gamma-\Gamma'}f(a+t+w)g(b+t)
&=\sum_{w \in \Gamma-\Gamma'} f \circ g(w+a-b)\\
&= (f \circ g)\circ 1_{\Gamma-\Gamma'}(a-b)
\end{align*}
as required.
\end{proof}

The following lemma allows us to find a subset of an additively non-smoothing set $\Delta$ with some structure. This corresponds to the first part of the sketch above. 

\begin{lemma}\label{lemma:structpiece}
There is a constant $C>0$ such that the following holds. Let $h,t,k\geq 2$ and $\tau\leq 1/2$ be some parameters. Suppose that $\widetilde{\Gamma}$ is an additive framework of height $h$ and tolerance $t$, and that $\Delta$ is $(\tau,k)$-additively non-smoothing relative to $\widetilde{\Gamma}$, where $h\leq C^{-1}\log\log k/\log\log\log k$ and $t\geq C\log k$.

There is some $\delta\gg \tau^2$ and $\Delta'\subset \Delta$ and $S$ such that
\[S\subset \{ x : \ind{\Delta}\circ \ind{\Delta+\Gamma}(x) \geq\delta \abs{\Delta}\},\]
\[\abs{S} \asymp \tau \delta^{-1}\abs{\Delta}\abs{\Gamma},\]
\[\abs{\Delta'}\ll\delta\abs{\Delta},\]
and
\[\Inn{\ind{\Delta'},\ind{S}\ast \ind{\Delta+\Gamma}}\gg \tau \abs{\Delta}^2\abs{\Gamma},\]
where $\Gamma=\Gamma^{(i)}+ \ell_1 \Gamma^{(i+1)}+\cdots \ell_j\Gamma^{(i+j)}$ for some $2\leq i\leq h$ and $0\leq j\leq h-i$, where $\sum\ell_r\leq t$. Here all the implied constants are polynomial in 
\[2^{k}\tau^{-\frac{1}{\log\log k}-\frac{1}{h}}.\]
\end{lemma}

The proof is iterative and quite delicate (and is where the requirements for an additive framework come from), and we defer it to the following section. For now, we show how the structural Theorem~\ref{th:structure} follows from Lemma~\ref{lemma:structpiece}.

All bounds implicit in the $\gg$ and $\asymp$ notation will be, for the remainder of this section only, up to polynomial losses in $t^h2^{k}\tau^{-\frac{1}{\log\log k}-\frac{1}{h}}$.

\begin{proof}[Proof of Theorem~\ref{th:structure}]
We begin by applying Lemma~\ref{lemma:structpiece}, which produces some $\Delta_1$ with associated $S_1$, $\delta_1$, and $\Gamma_1$. If $\abs{\Delta_1}\geq\tfrac{1}{2}\abs{\Delta}$ then we stop. Otherwise we apply Lemma~\ref{lemma:structpiece} to $\Delta\backslash \Delta_1$, and repeat. Importantly, the fact that $\Delta$ is $\tfrac{1}{2}$-robustly $(\tau,k)$-additively non-smoothing ensures that $\Delta'\subset \Delta$ remains $(\tau,k)$-additively non-smoothing provided $\abs{\Delta'}\geq\tfrac{1}{2}\abs{\Delta}$.  We may therefore continue to apply Lemma~\ref{lemma:structpiece} until we find disjoint $\Delta_1,\ldots,\Delta_K$ (with associated $S_i$, $\delta_i$, and $\Gamma_i$ satisfying the conclusions of Lemma~\ref{lemma:structpiece}) such that $\abs{\cup \Delta_i}\geq \tfrac{1}{2}\abs{\Delta}$.

By dyadic pigeonholing, we may find $\gg K$ many $1\leq i\leq K$ such that the associated $\delta_i$ all lie in the same dyadic range, say $2\delta > \delta_i\geq \delta$, and furthermore the sum of all the corresponding $\abs{\Delta_i}$ is $\gg \abs{\Delta}$. In particular, since $\abs{\Delta_i}\ll \delta_i\abs{\Delta}\ll \delta\abs{\Delta}$ for all such $i$, we have $K\gg \delta^{-1}$. By a further dyadic pigeonholing we may assume that all the associated $\Gamma_i$ are the same, say $\Gamma'$ (since there are $O(ht^h)$ many possible $\Gamma_i$, the implicit loss is only a factor of $O(1)$ according to our conventions in this section). For brevity, let $\Gamma=\Gamma^{(1)}$. Reducing $K$ if necessary, we will henceforth assume that all $\delta_i$ and $\Gamma_i$ satisfy these restrictions.

Observe that it is an immediate consequence of our definitions of additive framework and additively non-smoothing that $\Delta$ (and hence any subset of $\Delta$) is, in particular, both $\Gamma'$ and $\Gamma$-orthogonal. This means that, for example, $\ind{\Delta+\Gamma}=\ind{\Delta}\ast \ind{\Gamma}$, and so we will freely interchange between these two functions as is convenient.

For the moment, fix some such $1\leq i\leq K$, and note that
\[\Inn{\ind{\Delta_i}, \ind{S_i}\ast\ind{\Delta+\Gamma'}}\gg \tau \abs{\Delta}^2\abs{\Gamma'},\]
where $\abs{S_i}\asymp \tau \delta^{-1}\abs{\Delta}\abs{\Gamma'}$. By the popularity principle (using the fact that $\abs{\Delta_i}\ll \delta\abs{\Delta}$), there is a set $\Delta_i'\subset \Delta_i$ on which $\ind{S_i}\ast \ind{\Delta+\Gamma'}\asymp \abs{S_i}$, with $\abs{\Delta_i'}\gg \abs{\Delta_i}$, and by the popularity principle again, there is $S_i'\subset S_i$ on which $\ind{\Delta_i'}\circ \ind{\Delta+\Gamma'} \gg \delta \abs{\Delta}$, and $\abs{S_i'}\gg \abs{S_i}$.

Let $S_i''\subset S_i'$ be a maximal $\Gamma$-orthogonal subset, so that 
\[S_i'\subset S_i''+\Gamma-\Gamma\subset S_i''+\Gamma_{\mathrm{top}}.\]
In particular
\[\Inn{\ind{\Delta_i'+\Gamma'}\circ \ind{\Delta}, \ind{S_i''}\ast \ind{\Gamma_{\mathrm{top}}}}\geq \Inn{\ind{\Delta_i'+\Gamma'}\circ \ind{\Delta}, \ind{S_i'}}\gg \tau \abs{\Delta}^2\abs{\Gamma'},\]
and so $\abs{S_i''}\gg \tau \delta^{-1}\abs{\Delta}$, where we have used the $\Gamma_{\mathrm{top}}$-orthogonality of $\Delta$ to bound $\norm{\ind{\Delta}\ast\ind{\Gamma_{\mathrm{top}}}}_\infty \leq 1$, and the fact that $\abs{\Delta_i}\ll \delta\abs{\Delta}$. Let $S_i'''\subset S_i''$ be some set satisfying $\abs{S_i'''}\asymp \tau \delta^{-1}\abs{\Delta}$. Since $S_i'''\subset S_i'$ we have
\[\Inn{\ind{\Delta_i'+\Gamma'}\circ \ind{\Delta},\ind{S_i'''}}\gg \delta\abs{\Delta}\abs{S_i'''}\gg \tau\abs{\Delta}^2.\]
Now, using the fact that $\ind{\Gamma'}\ll \abs{\Gamma}^{-1}\ind{\Gamma}\circ \ind{\Gamma}$ (since $\Gamma'\subset \Gamma^{(2)}-\Gamma^{(2)}$), combined with the $\Gamma$-orthogonality of both $\Delta_i'$ and $S_i'''$, we have
\begin{align*}
\Inn{\ind{\Delta_i'+\Gamma}\circ \ind{S_i'''+\Gamma}, \ind{\Delta}}
&\gg \Inn{\ind{\Delta_i'}\circ \ind{S_i'''}\ast \ind{\Gamma}\circ\ind{\Gamma},\ind{\Delta}}\\
&\gg \abs{\Gamma}\Inn{\ind{\Delta_i'}\circ \ind{S_i'''}\ast \ind{\Gamma'},\ind{\Delta}}\\
&= \abs{\Gamma}\Inn{\ind{\Delta_i'+\Gamma'}\circ \ind{S_i'''},\ind{\Delta}}\\
&\gg  \tau \abs{\Delta}^2\abs{\Gamma}.
\end{align*}
By the popularity principle, since $\abs{S_i'''+\Gamma}\leq \abs{S_i'''}\abs{\Gamma}\ll \tau\delta^{-1}\abs{\Delta}\abs{\Gamma}$, there is a subset $\tilde{S_i}\subset S_i'''+\Gamma$ on which $\ind{\Delta_i'+\Gamma}\circ \ind{\Delta}\gg \delta \abs{\Delta}$ and $\Abs{\tilde{S_i}}\gg \tau \delta^{-1}\abs{\Delta}\abs{\Gamma}$. Discarding elements if necessary, we will henceforth assume that $\Abs{\tilde{S_i}}\asymp \tau\delta^{-1}\abs{\Delta}\abs{\Gamma}$, and note that
\[\Inn{\ind{\Delta_i'+\Gamma}\circ \ind{\Delta}, \ind{\tilde{S_i}}}\gg \tau \abs{\Delta}^2\abs{\Gamma}.\]
(The point of the above manoeuvre is that we have gone from some $S_i$ of size $\asymp \tau \delta^{-1}\abs{\Delta}\abs{\Gamma'}$ on which $\ind{\Delta_i'}\circ \ind{\Delta+\Gamma'}\gg \delta\abs{\Delta}$ to some $\tilde{S_i}$ of size $\asymp \tau \delta^{-1}\abs{\Delta}\abs{\Gamma}$ on which $\ind{\Delta_i'}\circ \ind{\Delta+\Gamma}\gg \delta\abs{\Delta}$. This replacement of $\Gamma'$ by $\Gamma$, which sits higher up on the levels of the additive framework, is important in the calculations which follow, since $\abs{\Gamma'-\Gamma}\approx \abs{\Gamma}$ but $\abs{\Gamma'-\Gamma'}$ may be much larger than $\abs{\Gamma'}$.)

 By dyadic pigeonholing, there exists some $1\geq \eta_i\gg \tau$ and some $\tilde{\Delta_i}\subset \Delta_i'+\Gamma$ on which $\ind{\tilde{S_i}}\ast \ind{\Delta}\asymp \eta_i \abs{\Delta}$, say, and $\Abs{\tilde{\Delta_i}}\gg \tau\eta_i^{-1}\abs{\Delta}\abs{\Gamma}$. In particular, 
 \[\sum_{x\in \Delta_i+\Gamma} \ind{\tilde{S_i}}\ast \ind{\Delta}(x)^2 \gg \eta_i \tau\abs{\Delta}^3\abs{\Gamma}.\]

We carry out the above procedure for each $1\leq i\leq K$, obtaining an associated $\eta_i$. By a further dyadic pigeonholing, reducing $K$ by a factor of $O(1)$ if necessary, we can assume that all the $\eta_i$ are in the same dyadic range, so that $2\eta >\eta_i\geq \eta$, say. We therefore have, summing over all such $i$, using the fact that the $\Delta_i$ are disjoint subsets of $\Delta$, which is $\Gamma$-orthogonal, and that $\tilde{S_i}\subset S = \{ x : \ind{\Delta+\Gamma}\circ \ind{\Delta}(x)\gg \delta\abs{\Delta}\}$, 
 \[\sum_i\sum_{x\in \Delta_i+\Gamma} \ind{S}\ast \ind{\Delta}(x)^2 \gg \eta\delta^{-1}\tau\abs{\Delta}^3\abs{\Gamma}.\]
Since the $\Delta_i$, and hence (by orthogonality) the $\Delta_i+\Gamma$ are disjoint, the left-hand side is at most 
\[E(S,\Delta)= \sum_{x} \ind{S}\ast \ind{\Delta}(x)^2,\] which is by the non-smoothing property, 
\[\ll (\delta\abs{\Delta})^{-2}\Inn{ \ind{\Delta+\Gamma}\ast \ind{\Delta}^{(2)},\ind{\Delta+\Gamma}\ast \ind{\Delta}^{(2)}}\ll \delta^{-2}\tau^2\abs{\Delta}^3\abs{\Gamma}.\]
It follows that $\eta \ll \tau \delta^{-1}$. Since $\abs{\Delta_i'}\ll \delta\abs{\Delta}$ and $\tilde{\Delta_i}\subset \Delta_i'+\Gamma$ satisfies $\Abs{\tilde{\Delta_i}}\gg \tau\eta^{-1}\abs{\Delta}\abs{\Gamma}$, we also have $\eta \gg \tau \delta^{-1}$, so that we may henceforth assume that $\eta$ (and in particular each $\eta_i$) is $\asymp \tau \delta^{-1}$. 

Now let $G_i$ be the set of pairs  $(a,b)\in \Delta_i'\times \Gamma$ such that $a+b\in \tilde{\Delta_i}$. Since we have
\[\sum_{x\in \tilde{S_i}}\sum_{(a,b)\in G_i}\ind{\Delta}(a+b-x)=\Inn{\ind{\tilde{S_i}}, \ind{\tilde{\Delta_i}}\circ\ind{\Delta} }\gg \eta\abs{\Delta}\Abs{\tilde{\Delta_i}}\gg\tau \abs{\Delta}^2\abs{\Gamma},\]
by the Cauchy--Schwarz inequality,
\[\sum_{\substack{(a_1,b_1)\in G_i\\ (a_2,b_2)\in G_i}}\sum_{x\in \tilde{S_i}} \ind{\Delta}(a_1+b_1-x)\ind{\Delta}(a_2+b_2-x)\gg \tau \delta\abs{\Delta}^3\abs{\Gamma},\]
and hence in particular,
\[\sum_{a_1,a_2\in \Delta_i'}\sum_{b_1\in \Gamma}1_{(a_1,b_1)\in G_i} \brac{\sum_{x\in \tilde{S_i}} \ind{\Delta}(a_1+b_1-x)\ind{\Delta+\Gamma}(a_2-x)}\gg \tau \delta\abs{\Delta}^3\abs{\Gamma}.\]
The inner bracketed sum is $\ll \ind{\tilde{S_i}}\ast \ind{\Delta}(a_1+b_1)$, and hence by our choice of $\tilde{\Delta_i}$ is $\ll \tau \delta^{-1}\abs{\Delta}$. Furthermore, since $\abs{\Delta_i'}\ll \delta\abs{\Delta}$, we can also further restrict the summation to those pairs $(a_1,a_2)$ such that the bracketed sum is $\gg \tau \delta^{-1}\abs{\Delta}$, losing only a constant factor on the right-hand side. Since the bracketed sum is also $\ll \ind{\Delta}\circ \ind{\Delta+\Gamma}(a_1-a_2+b_1)$, this shows that
\[\Inn{ \ind{\Delta_i'}\circ \ind{\Delta_i'+\Gamma}, \ind{T}}\gg \delta^2\abs{\Delta}^2 \abs{\Gamma},\]
where
\[T = \{ x : \ind{\Delta}\circ \ind{\Delta+\Gamma}(x) \gg \tau \delta^{-1}\abs{\Delta}\}.\]
By the popularity principle, there is some $\Delta_i''\subset \Delta_i'$ on which $\ind{T}\ast \ind{\Delta_i'+\Gamma}\gg \delta \abs{\Delta}\abs{\Gamma}$ such that $\abs{\Delta_i''}\gg \delta\abs{\Delta}$. 

We perform a similar manoeuvre, now beginning with the inequality
\[\Inn{\ind{\Delta_i''},\ind{S_i}\ast \ind{\Delta+\Gamma'}}\gg \tau \abs{\Delta}^2\abs{\Gamma'},\]
which holds since $\Delta_i''\subset \Delta_i'$, and $\Delta_i'$ was constructed so that $\ind{S_i}\ast \ind{\Delta+\Gamma'}\gg \tau \delta^{-1}\abs{\Delta}\abs{\Gamma'}$ pointwise on $\Delta_i'$. As above, by dyadic pigeonholing, we can find some $1\geq \eta_i\gg \tau$ and some $\tilde{\Delta_i}''\subset \Delta_i''+\Gamma'$ such that $\ind{S_i}\ast \ind{\Delta}\approx \eta_i\abs{\Delta}$, and $\Abs{\tilde{\Delta_i}''}\gg \tau\eta_i^{-1}\abs{\Delta}\abs{\Gamma'}$. Once again considering this over all $i$, reducing $K$ by dyadic pigeonholing again if necessary, and using the fact that $S_i\subset \{ x: \ind{\Delta+\Gamma'}\circ \ind{\Delta}\gg \delta\abs{\Delta}\}$, we may assume that $\eta_i\asymp \tau\delta^{-1}$ for all $i$.

If we let $G_i''\subset \Delta_i''\times \Gamma'$ be the set of pairs such that $a+b\in \tilde{\Delta_i}''$ then again, by the Cauchy--Schwarz inequality, 
\[\sum_{a_1,a_2\in \Delta_i''}\sum_{b_1\in \Gamma'}1_{(a_1,b_1)\in G_i''} \brac{\sum_{x\in S_i} \ind{\Delta}(a_1+b_1-x)\ind{\Delta+\Gamma'}(a_2-x)}\gg \tau \delta\abs{\Delta}^3\abs{\Gamma'}.\]
As above, since the inner sum is bounded above by $\ind{S_i}\ast \ind{\Delta}(a_1+b_1)\ll \tau\delta^{-1}\abs{\Delta}$, it follows that 
\[\Inn{\ind{\Delta_i''}\circ \ind{\Delta_i''+\Gamma'}, \ind{T'}}\gg \delta^2\abs{\Delta}^2\abs{\Gamma'},\]
where
\[T' = \{ x : \ind{\Delta}\circ \ind{\Delta+\Gamma'}(x) \gg \tau \delta^{-1}\abs{\Delta}\}.\]
In particular, recalling the condition placed on $\Delta_i''$, we have
\[\Inn{ \ind{\Delta_i''}, (\ind{T'}\ast \ind{\Delta_i''+\Gamma'})(\ind{T}\ast \ind{\Delta_i'+\Gamma})}\gg \delta^3\abs{\Delta}^3\abs{\Gamma}\abs{\Gamma'},\]
and so, in particular, using $\Delta_i',\Delta_i''\subset \Delta_i$,
\[\Inn{ \ind{\Delta_i}, (\ind{T'}\ast \ind{\Delta_i+\Gamma'})(\ind{T}\ast \ind{\Delta_i+\Gamma})}\gg \delta^3\abs{\Delta}^3\abs{\Gamma}\abs{\Gamma'}.\]
The above is true for $\gg \delta^{-1}$ many $1\leq i\leq K$. 

Changing the order of summation, we can write this as
\[\sum_{x\in T'} \langle \ind{T}, \ind{\Delta_{i,x}}\circ \ind{\Delta_i+\Gamma}\rangle \gg \delta^3\abs{\Delta}^3\abs{\Gamma}\abs{\Gamma'},\]
where $\Delta_{i,x}=\Delta_i\cap(\Delta_i+\Gamma'+x)$, so that $\sum_i \abs{\Delta_{i,x}}=\sum_i\ind{\Delta_i+\Gamma'}\circ \ind{\Delta_i}(x)=F(x)$, say. We now sum over all $\gg \delta^{-1}$ many $i$. It follows that
\begin{equation}\label{struc2}
\sum_{x\in T'}\Inn{ \ind{T},\sum_i\ind{\Delta_{i,x}}\circ \ind{\Delta_i+\Gamma}} \gg \delta^2\abs{\Delta}^3\abs{\Gamma}\abs{\Gamma'}.
\end{equation}
The inner product is, for fixed $x\in T'$, bounded above by $\ll \delta\abs{\Delta}\abs{\Gamma}F(x)$. We now claim that the contribution to \eqref{struc2} from those $x$ such that $F(x)\geq C \tau \delta^{-1}\abs{\Delta}$ for some sufficiently large $C$ (which still satisfies $C\ll 1$) is negligible. Indeed, let $T_2\subset T'$ be the set of those $x\in T'$ such that $F(x) \leq C\tau\delta^{-1}\abs{\Delta}$. If
\[\sum_{x\not\in T_2}\Inn{ \ind{T},\sum_i \ind{\Delta_{i,x}}\circ \ind{\Delta_i+\Gamma}}\geq \frac{1}{2}\sum_{x\in T'}\Inn{ \ind{T},\sum_i \ind{\Delta_{i,x}}\circ \ind{\Delta_i+\Gamma}}\]
then 
\[\sum_{x\not\in T_2}F(x) \gg \delta \abs{\Delta}^2\abs{\Gamma'},\]
and so
\[\sum_{x\not\in T_2}F(x)^2 \gg C\tau \abs{\Delta}^3\abs{\Gamma'}.\]
For $C$ sufficiently large, since $F(x)\leq \ind{\Delta+\Gamma'}\circ \ind{\Delta}(x)$, this contradicts the fact that 
\[\sum_{x}F(x)^2 \ll \Inn{\ind{\Delta}\circ \ind{\Delta+\Gamma'}, \ind{\Delta}\circ \ind{\Delta+\Gamma'}}\ll \tau \abs{\Delta}^3\abs{\Gamma'}.\]
Therefore
\[\sum_{x\in T_2}\Inn{ \ind{T},\sum_i \ind{\Delta_{i,x}}\circ \ind{\Delta_i+\Gamma}}\gg \delta^2\abs{\Delta}^3\abs{\Gamma}\abs{\Gamma'}.\]
Furthermore, since
\[\Inn{\ind{\Delta}\circ \ind{\Delta+\Gamma'}, \ind{\Delta}\circ \ind{\Delta+\Gamma'}}\ll \tau \abs{\Delta}^3\abs{\Gamma'},\]
we have $\abs{T_2}\leq\abs{T'}\ll \tau^{-1}\delta^2\abs{\Delta}\abs{\Gamma'}$. In particular, by the popularity principle, there is $T_3\subset T_2$ on which  the inner product is $\gg \tau\abs{\Delta}^2\abs{\Gamma}$, such that 
\begin{equation}\label{eq:struccy}\sum_{x\in T_3}\Inn{ \ind{T},\sum_i \ind{\Delta_{i,x}}\circ \ind{\Delta_i+\Gamma}}\gg \delta^2\abs{\Delta}^3\abs{\Gamma}\abs{\Gamma'}.
\end{equation}
In particular, for $x\in T_3$, we have $F(x)\asymp \tau \delta^{-1} \abs{\Delta}$ and 
\[
\Inn{ \ind{T},\sum_i \ind{\Delta_{i,x}}\circ \ind{\Delta_i+\Gamma}}\asymp \tau\abs{\Delta}^2\abs{\Gamma}.
\]
For each fixed $x\in T_3$ we perform a dyadic pigeonholing to find some $1\geq \kappa_x\gg \tau$ and $T_x\subset T$ such that
\[\abs{T_x}\gg \kappa_x^{-1}\delta\abs{\Delta}\abs{\Gamma}\]
and
\[\sum_i \ind{\Delta_{i,x}}\circ \ind{\Delta_i+\Gamma}(y)\approx  \kappa_x \tau \delta^{-1}\abs{\Delta}\textrm{ for all }y\in T_x.\]
We then dyadically pigeonhole yet again to ensure that the contribution to \eqref{eq:struccy} is dominated by $T_4$, say, which is the set of those $x\in T_3$ such that $2 \kappa > \kappa_x\geq  \kappa$ for some $1\geq  \kappa\gg \tau$. Therefore,
\begin{equation}\label{eq:struct3}
\sum_{x\in T_4}\sum_{y\in T_x}\brac{\sum_i\ind{\Delta_{i,x}}\circ \ind{\Delta_i+\Gamma}(y)}^{1/2}\gg \kappa^{-1/2}\delta^{5/2}\tau^{-1/2}\abs{\Delta}^{5/2}\abs{\Gamma}\abs{\Gamma'}.
\end{equation}
By Lemma~\ref{lemma:collapser} we have 
\begin{align*}
\ind{\Delta_{i,x}}\circ \ind{\Delta_i+\Gamma}(y)
&= \sum_{a\in \Delta_i}\ind{\Delta_i+\Gamma}(a-y)\ind{\Delta_i+\Gamma'}(a-x)\\
&\leq \ind{\Delta_i}\circ \ind{\Delta_i+\Gamma-\Gamma'}(x-y).
\end{align*}
 Furthermore,
\[\sum_i\sum_z \ind{\Delta_i}\circ \ind{\Delta_i+\Gamma-\Gamma'}(z)\ll \sum_i \abs{\Delta_i}^2\abs{\Gamma-\Gamma'}\ll \delta\abs{\Delta}^2\abs{\Gamma},\]
Therefore, by the Cauchy--Schwarz inequality on \eqref{eq:struct3},
\[\Inn{\ind{T}\circ \ind{T_4}, \ind{T}\circ \ind{T_4}}\delta\abs{\Delta}^2\abs{\Gamma}\gg \kappa^{-1}\delta^{5}\tau^{-1}\abs{\Delta}^{5}\abs{\Gamma}^2\abs{\Gamma'}^2.\]
By the additive non-smoothing upper bound on the higher additive energy, however, recalling the definitions of $T$ and $T'$, since $T_4\subset T'$,
\begin{align*}
\Inn{\ind{T}\circ \ind{T_4}, \ind{T}\circ \ind{T_4}}
&\ll (\tau\delta^{-1} \abs{\Delta})^{-4}\Inn{ \ind{\Delta}^{(4)}\ast \ind{\Gamma}\ast \ind{\Gamma'}, \ind{\Delta}^{(4)}\ast \ind{\Gamma}\ast \ind{\Gamma'}}\\
&\ll \tau^{-1}\delta^4\abs{\Delta}^3\abs{\Gamma'}^2\abs{\Gamma},
\end{align*}
and hence $\kappa \gg 1$.

We now fix some $x\in T_4$, and choose $\tilde{T}\subset T_x$ with $\Abs{\tilde{T}}\asymp \delta\abs{\Delta}\abs{\Gamma}$, and let $\Delta_i'=\Delta_{i,x}$, so that $F(x)=\sum_i \abs{\Delta_i'}\asymp \tau \delta^{-1}\abs{\Delta}$ and
\[ \sum_i\Inn{ \ind{\tilde{T}}, \ind{\Delta_i'}\circ \ind{\Delta_i+\Gamma}} \asymp \tau\abs{\Delta}^2\abs{\Gamma}.\]
By the dyadic pigeonhole principle we may choose some $1\leq M\ll \delta^{-1}$ such that (after relabelling) for all $1\leq i\leq M$ we have
\[\langle \ind{\tilde{T}},  \ind{\Delta_i'}\circ \ind{\Delta_i+\Gamma}\rangle \asymp M^{-1}\tau \abs{\Delta}^2\abs{\Gamma}.\]
This trivially implies a lower bound of $\abs{\Delta_i'}\gg M^{-1}\tau \delta^{-1}\abs{\Delta}$, and by pigeonholing further if necessary we can also assume that $\abs{\Delta_i'}\asymp M^{-1}\tau \delta^{-1}\abs{\Delta}$ for $1\leq i\leq M$. By the Cauchy--Schwarz inequality
\[\sum_{1\leq i,j\leq M} \langle \ind{\Delta_i+\Gamma}\ast \ind{\Delta_j'},\ind{\Delta_j+\Gamma}\ast \ind{\Delta_i'}\rangle \gg \delta^{-1}\tau^2\abs{\Delta}^3\abs{\Gamma}.\]
By averaging, there exists some $1\leq j\leq M$ and $\gg M$ many $i$ such that
\[\langle \ind{\Delta_i+\Gamma}\ast \ind{\Delta_j'}, \ind{\Delta_j+\Gamma}\ast \ind{\Delta_i'}\rangle \gg M^{-2}\delta^{-1}\tau^2\abs{\Delta}^3\abs{\Gamma}.\]
In particular, there is $\Delta'\subset \Delta$ with $\abs{\Delta'}\approx M^{-1}\tau \delta^{-1}\abs{\Delta}$ such that, applying the Cauchy--Schwarz inequality once again, for $\gg M$ many $i$, 
\[\langle \ind{\Delta'}\circ \ind{\Delta'}, \ind{\Delta_i+\Gamma}\circ \ind{\Delta_i+\Gamma}\rangle \gg \delta\abs{\Delta}\abs{\Delta'}^2\abs{\Gamma}.\]
We now let $X=\sqcup \Delta_i$, so that $\abs{X}\asymp M\delta\abs{\Delta}$ and 
\[\langle \ind{X+\Gamma}\circ \ind{X+\Gamma}, \ind{\Delta'}\circ \ind{\Delta'}\rangle \gg \abs{X}\abs{\Delta'}^2\abs{\Gamma}.\]
In particular, there is some translate of $\Delta'$, say $\Delta'+z$, such that
\[\Inn{ \ind{X+\Gamma}\circ \ind{X+\Gamma}, \ind{\Delta'+z}}\gg \abs{X}\abs{\Delta'}\abs{\Gamma}.\]
By the popularity principle there exists $H\subset \Delta'$ such that $\ind{X+\Gamma}\circ \ind{X+\Gamma}(x)\gg \abs{X}\abs{\Gamma}$ for all $x\in H+z$, and $\abs{H}\gg \abs{\Delta'}$. Since $\ind{\Gamma}\circ \ind{\Gamma} \ll \abs{\Gamma}\ind{\Gamma_{\mathrm{top}}}$, we in particular have
\[H+z \subset \{ x : \ind{X}\circ \ind{X+\Gamma_{\mathrm{top}}} \gg \abs{X}\}.\]
Finally, by the Cauchy--Schwarz inequality, we have 
\[\Inn{ \ind{X+\Gamma}\circ \ind{X+\Gamma}, \ind{H}\circ \ind{H}}\gg \abs{X}\abs{H}^2\abs{\Gamma},\]
and the second part of the structural theorem follows by again using $\ind{\Gamma}\circ \ind{\Gamma} \ll \abs{\Gamma}\ind{\Gamma_{\mathrm{top}}}$ (and replacing $\delta$ by $M^{-1}\tau\delta^{-1}$).
\end{proof}

\section{Finding a structured piece}\label{section:structure2} 

In this section we prove Lemma~\ref{lemma:structpiece}. We have already given a sketch of how to proceed in the previous section, but since we must work within an additive framework which is only approximately structured, the iteration takes some care. We will need to work between multiple levels of the framework $\widetilde{\Gamma}$. 

All sets in this section are assumed to be subsets of some fixed finite abelian group (which for our application we will take to be $\widehat{G}$). All of the lemmas proved in this section will be applied with the parameters $h,t,k,\tau,\Delta,\widetilde{\Gamma}$ being as given in the statement of Lemma~\ref{lemma:structpiece}. 

For each individual lemma the full list of assumptions of an additive framework is not required, but it is simpler to have a single global definition that captures all of the auxiliary assumptions we will need along the way.

To help structure the argument, we introduce the notion of `viscosity'. To provide some motivation, we note that if $\Delta$ has $E_4(\Delta)\approx \tau \abs{\Delta}^3$, then if $S_\delta = \{ x: \ind{\Delta}\circ \ind{\Delta}(x)\approx \delta\abs{\Delta}\}$, we must have $\abs{S_\delta}\ll \tau\delta^{-2}\abs{\Delta}$. We refer to $S_\delta$ as a symmetry set at `depth' $\delta$. If $\abs{S_\delta}$ is close to this maximum size, then this is some kind of `thickness' at depth $\delta$, and so we refer informally to the ratio $\abs{S_i}/\tau \delta^{-2}\abs{\Delta}$ as the `viscosity' at depth $\delta$. By the dyadic pigeonhole principle we can be sure of finding some depth $1\geq \delta \gg \tau$ with high viscosity (that is, $\gs_\tau 1$). It does not matter much at what depth this occurs.  Indeed, we cannot hope to control at which depth a high viscosity occurs, as can be seen by considering the examples of structured sets given in the previous section. For those examples, it can be checked that $\Delta_1$ has high viscosity at depths $1$ and $\tau$, and $\Delta_2$ has high viscosity at depth $\tau^{1/2}$. 

We will require a relative version of viscosity that operates on multiple levels of an additive framework simultaneously. To this end, we introduce the following definition. Let $\mathcal{S}$ denote the collection of symmetric sets that contain $0$.\label{not-symm}

\begin{definition}[Multiscale viscosity]
Let $\epsilon\in [0,1]$, $\vec{\delta}\in [0,1]^n$ and $\vec{\Gamma} \in \mathcal{S}^n$ for some $n \geq 1$. We say that $\Delta$ has viscosity\label{def-vis} $\epsilon$ at depths $(\vec{\delta}, \vec{\Gamma})$ if there exist
\[\Delta=\Delta_0\supset \Delta_1 \supset \cdots \supset \Delta_n\]
such that, for $1\leq i\leq n$, 
\[\abs{\Delta_i}\geq \epsilon\abs{\Delta},\]
and the sets 
\[S_i = \{ x : 2\delta_i\abs{\Delta}> \ind{\Delta_i}\circ \ind{\Delta_i+\Gamma_{i}}(x) \geq \delta_i\abs{\Delta}\},\]
satisfy
\[\abs{S_i}\geq \epsilon \tau \delta_i^{-2}\abs{\Delta}\Abs{\Gamma_{i}}, \]
and, for $1<i\leq n$ and $x\in \Delta_{i}$, 
\[\ind{S_{i-1}}\ast \ind{\Delta_{i-1}+\Gamma_{i-1}}(x) \geq \tfrac{1}{2} \delta_{i-1}\abs{S_{i-1}}.\]
\end{definition}
Note that it follows from the trivial bound $\sum_{x \in S_i} \ind{\Delta_i}\circ \ind{\Delta_i+\Gamma_{i}}(x) \leq \abs{\Delta_i}^2 \Abs{\Gamma_{i}}$ that $\delta_i \geq \epsilon \tau$ for each $i$.

The reader should think of $\epsilon$ as being $\gg 1$, as it will remain throughout the proof (up to polynomial losses in $2^k\tau^{-\frac{1}{\log\log k}-\frac{1}{h}}$). The most important role that $\epsilon$ plays is in giving a lower bound for the size of $S_i$ -- its dual role in lower bounding the size of $\Delta_i$ is far less important, and we use the same parameter to control both largely for notational simplicity.

The sets $\Gamma_i$ need not be the same as the levels $\Gamma^{(i)}$ of the additive framework, but will be closely related to these. In fact, we begin by showing that we have high viscosity at some depth with the level $\Gamma^{(1)}$.

\begin{lemma}\label{lemma:visc1}
Let $h,t,k\geq 2$ and $\tau\in (0,1)$. Let $\tilde{\Gamma}$ be an additive framework of height $h$ and tolerance $t$ and let $\Delta$ be $(\tau,k)$-additively non-smoothing relative to $\tilde{\Gamma}$.

There exists some $\delta_1\in [\tau/4,1]$ such that $\Delta$ has viscosity $\tau^{O(1/k)}$ at depth $\left(\delta_1,\Gamma^{(1)}\right)$.
\end{lemma}
\begin{proof}
By the definition of additive non-smoothing,
\[ \Inn{1_\Delta\circ 1_\Delta,\ 1_\Delta\circ 1_\Delta \circ 1_{\Gamma_{\mathrm{bottom}}}} \geq \tau \abs{\Delta}^3. \]
Since $1_{\Gamma^{(1)}} \circ 1_{\Gamma^{(1)}} \geq \tfrac{1}{2} \Abs{\Gamma^{(1)}}\, 1_{\Gamma_{\mathrm{bottom}}}$ by the definition of additive framework, this implies that
\[ \Inn{1_\Delta\circ 1_{\Delta+\Gamma^{(1)}},\ 1_\Delta\circ 1_{\Delta +\Gamma^{(1)}}} \geq \tfrac{1}{2}\tau \abs{\Delta}^3 \Abs{\Gamma^{(1)}}. \]
(Recall that $\Delta$ is $\Gamma^{(1)}$-orthogonal, since it is $\Gamma_{\mathrm{top}}$-orthogonal.) By dyadic pigeonholing, we get some $1 \geq \eta \geq \tau/4$ and a set $S$ of size 
\[ \abs{S} \gs_\tau \tau \eta^{-2} \abs{\Delta}\Abs{\Gamma^{(1)}} \]
such that
\[ 2\eta\abs{\Delta} > 1_{\Delta}\circ 1_{\Delta + \Gamma^{(1)}}(x) \geq \eta \abs{\Delta} \text{ for all $x \in S$}. \]
This immediately implies the conclusion, with $\Delta_1 = \Delta$ and $\delta_1 = \eta$ (note that the final condition of multiscale viscosity is vacuously true when $n=1$), since by assumption $\log(1/\tau)\leq \tau^{-1/k}$. 
\end{proof}

The following lemma allows us to extend the number of scales on which we have high viscosity, using the levels of the additive framework. We recall that $\mathcal{S}$ denotes the collection of symmetric sets which contain $0$.

\begin{lemma}\label{lemma:visc2}
Let $h,t,k\geq 2$ and $\tau\in (0,1)$. Let $\tilde{\Gamma}$ be an additive framework of height $h$ and tolerance $t$ and let $\Delta$ be $(\tau,k)$-additively non-smoothing relative to $\tilde{\Gamma}$.

Let $1\leq n<h$. Suppose that $\Delta$ has viscosity $\epsilon$ at depths $(\vec{\delta},\vec{\Gamma})$, where $\vec{\delta} \in [0,1]^n$ and $\vec{\Gamma} \in \mathcal{S}^n$, and $\epsilon \leq \min(\tfrac{1}{2},\tau^{1/k})$, where $\Gamma_n - \Gamma_n \subset \Gamma_{\mathrm{top}}$. Then there exists $\delta_{n+1}$ with $1\geq \delta_{n+1}\geq \epsilon^{O(1)} \tau$ such that $\Delta$ has viscosity $\epsilon^{O(1)}$ at depths $\big((\vec{\delta},\delta_{n+1}),\, (\vec{\Gamma}, \Gamma^{(n+1)})\big)$.
\end{lemma}
\begin{proof}
By definition
\[\Inn{\ind{\Delta_n},\ind{S_n}\ast \ind{\Delta_n+\Gamma_{n}}}=\Inn{\ind{S_n},\ind{\Delta_n}\circ \ind{\Delta_n+\Gamma_{n}}}\geq \delta_n\abs{S_n}\abs{\Delta}.\]
By dyadic pigeonholing there exists some $\eta\geq \delta_n/2$ and some $\Delta_{n+1}\subset \Delta_n$ such that if $x\in \Delta_{n+1}$ then
\[\ind{S_n}\ast \ind{\Delta_n+\Gamma_{n}}(x)\geq \eta \abs{S_n},\]
and $\abs{\Delta_{n+1}}\gs_{\delta_n} \eta^{-1}\delta_n\abs{\Delta}$.

In particular,
\begin{align*}
\sum_{a,b\in S_n}\sum_{x\in \Delta_{n+1}} \ind{\Delta_n+\Gamma_{n}}(x-a) \ind{\Delta_n+\Gamma_{n}}(x-b) 
&= \Inn{\ind{\Delta_{n+1}},(\ind{S_n}\ast \ind{\Delta_n+\Gamma_{n}})^2}\\
&\geq \eta^2\abs{S_n}^2\abs{\Delta_{n+1}}.
\end{align*}
The innermost sum is at most $\ind{\Delta_n}\circ \ind{\Delta_n}\circ 1_{\Gamma_{n}-\Gamma_{n}}(a-b)$ by Lemma \ref{lemma:collapser} and the $\Gamma_{n}$-orthogonality of $\Delta_{n+1}$ (which is guaranteed since $\Gamma_n\subset \Gamma_{\mathrm{top}}$ and $\Delta$ itself is $\Gamma_{\mathrm{top}}$-orthogonal), and hence
\[\Inn{\ind{S_n}\circ \ind{S_n},\ind{\Delta_n}\circ \ind{\Delta_n}\circ \ind{\Gamma_{n}-\Gamma_{n}}}\geq \eta^2\abs{S_n}^2\abs{\Delta_{n+1}}.\]
By the definition of $S_n$ it follows that
\[\Inn{\ind{\Delta}^{(3)}\circ \ind{\Delta}^{(3)}\ast \ind{\Gamma_{n}}\ast \ind{\Gamma_{n}},\ind{\Gamma_{n}-\Gamma_{n}}}\geq \delta_n^2\eta^2\abs{S_n}^2\abs{\Delta}^2\abs{\Delta_{n+1}}.\]
Since $\Gamma_n-\Gamma_n \subset \Gamma_{\mathrm{top}}$, the non-smoothing assumption gives that the left-hand side is at most $\tau^{2-1/k}\abs{\Delta}^5\Abs{\Gamma_{n}}^2$, and hence, using the various bounds we have on the sizes involved,
\[\tau^{2-1/k}\abs{\Delta}^5\Abs{\Gamma_{n}}^2\gs_{\delta_n} \delta_n^2\eta^2 \cdot \epsilon^2\tau^2\delta_n^{-4}\abs{\Delta}^2\Abs{\Gamma_{n}}^2 \cdot \abs{\Delta}^2 \cdot \eta^{-1}\delta_n\abs{\Delta}\]
and so, after simplifying,
\[\eta \ls_{\delta_n} \tau^{-1/k} \epsilon^{-2}\delta_n.\]
By the robust energy lower bound in the definition of non-smoothing, and the fact that $1_{\Gamma^{(n+1)}} \circ 1_{\Gamma^{(n+1)}} \geq \tfrac{1}{2} \Abs{\Gamma^{(n+1)}} 1_{\Gamma_{\mathrm{bottom}}}$, writing $\nu = \abs{\Delta_{n+1}}/\abs{\Delta}$ we see that 
\[\Inn{\ind{\Delta_{n+1}}\circ \ind{\Delta_{n+1}+\Gamma^{(n+1)}}, \ind{\Delta_{n+1}}\circ \ind{\Delta_{n+1}+\Gamma^{(n+1)}}}\geq \tfrac{1}{2}\tau\nu^4\abs{\Delta}^3\Abs{\Gamma^{(n+1)}}.\]
By dyadic pigeonholing, there exists some $\delta_{n+1}$ with $1\geq \delta_{n+1}\geq \tau\nu^2/4$ such that if
\[S_{n+1}=\{ x : 2\delta_{n+1}\abs{\Delta}> \ind{\Delta_{n+1}}\circ \ind{\Delta_{n+1}+\Gamma^{(n+1)}}(x)\geq  \delta_{n+1}\abs{\Delta}\}\]
then
\[\abs{S_{n+1}}\gs_{\tau\nu}\tau \nu^4\delta_{n+1}^{-2}\abs{\Delta}\Abs{\Gamma^{(n+1)}}.\]
The conclusion now follows since $\log(\tau^{-1}) \leq \tau^{-1/k}$, and
\[\nu\gs_{\delta_n} \eta^{-1}\delta_n \gs_{\delta_n} \epsilon^2\tau^{1/k},\]
and so all implicit constants are at worst polynomial in $\epsilon$.
\end{proof}

The following lemma either produces a new viscosity vector at some depth vector which is smaller (lexicographically in $\vec{\delta}$), or else finds a large structured piece of $\Delta$.

\begin{lemma}\label{lemma:visc3}
Let $h,t,k\geq 2$ and $\tau\in (0,1)$. Let $\tilde{\Gamma}$ be an additive framework of height $h$ and tolerance $t$ and let $\Delta$ be $(\tau,k)$-additively non-smoothing relative to $\tilde{\Gamma}$.

Let $\nu \in [0,1]$ be a parameter, and suppose that $n\geq 2$ and $\vec{\Gamma} \in \mathcal{S}^n$ satisfies $\Gamma_i - \Gamma_{i+1} \subset \Gamma_{\mathrm{top}}$ and $\abs{\Gamma_i-\Gamma_{i+1}}\leq 2\abs{\Gamma_i}$ for all $1\leq i<n$. Suppose further that $ \epsilon \leq \min(\tfrac{1}{2},\tau^{1/k})$. 

If $\Delta$ has viscosity $\epsilon$ at depths $(\vec{\delta}, \vec{\Gamma})$ then either
\begin{enumerate}
\item for some $1\leq i<n$ there is $\delta_i'\leq \nu \delta_i$ such that $\Delta$ has viscosity $\epsilon^{O(1)}$ at depths $(\vec{\delta}',\vec{\Gamma}')$ with
\[\vec{\delta}' = (\delta_1,\ldots,\delta_{i-1},\delta_i')\]
and
\[\vec{\Gamma}' = (\Gamma_1,\ldots,\Gamma_{i-1},\Gamma_i+\Gamma_{i+1}),\]
\listintertext{or}
\item for every $1\leq i<n$ we have
\[\delta_{i+1}\gg \epsilon^{O(1)}\nu \delta_i\]
and there is $\Delta_i'\subset \Delta$ with $\abs{\Delta_i'}\ll \delta_i\abs{\Delta}$ and $S_{i+1}'\subset S_{i+1}$ with
\[\tau \delta_{i+1}^{-1}\abs{\Delta}\abs{\Gamma_{i+1}}\geq\abs{S_{i+1}'}\geq \epsilon^{O(1)}\tau \delta_{i+1}^{-1}\abs{\Delta}\abs{\Gamma_{i+1}}\]
such that 
\[\Inn{\ind{\Delta_i'}\circ \ind{\Delta_{i+1}+\Gamma_{i+1}}, \ind{S_{i+1}'}}\geq \epsilon^{O(1)}\nu\delta_i\abs{\Delta}\abs{S_{i+1}'}.\]
\end{enumerate}
\end{lemma}
Note that in the first case of the conclusion, we keep only the first $i$ components of the depth vectors and discard the rest; the point is that the depth $\delta_i$ has decreased at scale $i$, so the overall depth vector has decreased lexicographically.
\begin{proof}
Fix some $1\leq i<n$. By construction,
\[\Inn{\ind{\Delta_{i+1}}, \ind{S_{i+1}}\ast \ind{\Delta_{i+1}+\Gamma_{i+1}}}=\Inn{\ind{S_{i+1}},\ind{\Delta_{i+1}}\circ \ind{\Delta_{i+1}+\Gamma_{i+1}}}\geq \delta_{i+1}\abs{\Delta}\abs{S_{i+1}}.\]
It follows that, since $\ind{S_i}\ast \ind{\Delta_i+\Gamma_i}(x) \gg \delta_i\abs{S_i}$ for $x\in \Delta_{i+1}$,
\[\Inn{\ind{\Delta_{i+1}},(\ind{S_i}\ast \ind{\Delta_i+\Gamma_i})(\ind{S_{i+1}}\ast \ind{\Delta_{i+1}+\Gamma_{i+1}})}\gg \delta_i\delta_{i+1}\abs{S_i}\abs{S_{i+1}}\abs{\Delta}.\]
The left-hand side can be expanded as
\[\sum_{a\in S_i}\sum_{b\in S_{i+1}}\sum_{x\in \Delta_{i+1}}\ind{\Delta_i+\Gamma_i}(x-a)\ind{\Delta_{i+1}+\Gamma_{i+1}}(x-b).\]
Let the innermost sum be denoted by $F_i(a,b)$. By the $\Gamma_{i+1}$-orthogonality of $\Delta_{i+1}$ and Lemma~\ref{lemma:collapser},
\[F_i(a,b)\leq\ind{\Delta_i}\circ \ind{\Delta_{i+1}}\circ \ind{\Gamma_i-\Gamma_{i+1}}(b-a).\]
By dyadic pigeonholing there exists some $\eta_i$ with $1\geq  \eta_i\gg \delta_i\delta_{i+1}$ and $G_i\subset S_i\times S_{i+1}$ such that $\abs{G_i}\gs_{\epsilon\tau} \eta_i^{-1}\delta_i\delta_{i+1}\abs{S_i}\abs{S_{i+1}}$ and if $(a,b)\in G_i$ then $2\eta_i\abs{\Delta}> F_i(a,b)\geq \eta_i \abs{\Delta}$. 

Let $D=\{ a-b : (a,b)\in G_i\}$, so that
\[\Inn{\ind{S_i}\circ \ind{S_{i+1}}, \ind{D}} =\abs{G_i}\gs_{\epsilon\tau} \eta_i^{-1} \delta_i\delta_{i+1}\abs{S_i}\abs{S_{i+1}}.\]
By assumption, and using the upper bound from additive non-smoothing, and the fact that $\Gamma_i\subset \Gamma_{\mathrm{top}}$,
\begin{align*}
\Inn{ \ind{S_i}\circ \ind{S_{i+1}},\ind{S_i}\circ \ind{S_{i+1}}}
&\leq (\delta_i\delta_{i+1}\abs{\Delta}^2)^{-2}\norm{\ind{\Delta_i}\circ\ind{\Delta_i+\Gamma_i}\ast \ind{\Delta_{i+1}}\circ \ind{\Delta_{i+1}+\Gamma_{i+1}}}_2^2\\
&\leq \tau^{3-1/k}\delta_i^{-2}\delta_{i+1}^{-2}\abs{\Delta}^3\abs{\Gamma_{i+1}}^2\abs{\Gamma_i}.
\end{align*}
By the Cauchy--Schwarz inequality, therefore,
\[\abs{D}\gs_{\epsilon\tau}\frac{\eta_i^{-2}\delta_i^2\delta_{i+1}^2\abs{S_i}^2\abs{S_{i+1}}^2}{\tau^{3-1/k}\delta_i^{-2}\delta_{i+1}^{-2}\abs{\Delta}^3\abs{\Gamma_i}\abs{\Gamma_{i+1}}^2}\]
and so, after simplifying and recalling the lower bounds on the sizes of $S_i$ and $S_{i+1}$,
\[\abs{D}\gs_{\epsilon\tau} \epsilon^4\tau^{1+1/k} \eta_i^{-2}\abs{\Delta}\abs{\Gamma_i}\gs_{\epsilon\tau} \epsilon^4\tau^{1+1/k}\eta_i^{-2}\abs{\Delta}\abs{\Gamma_i-\Gamma_{i+1}},\]
since $\abs{\Gamma_i}\gg \abs{\Gamma_i-\Gamma_{i+1}}$. We claim that, if $\eta_i$ is sufficiently small, then this means we are in the first case of the lemma. Indeed, since $\Gamma_i-\Gamma_{i+1}\subset \Gamma_{\mathrm{top}}$, the upper bound on energies from additive non-smoothing gives
\[\Inn{\ind{\Delta}\circ \ind{\Delta+{\Gamma_i-\Gamma_{i+1}}},\ind{\Delta}\circ \ind{\Delta+\Gamma_i-\Gamma_{i+1}}}\leq \tau^{1-1/k}\abs{\Delta}^3\abs{\Gamma_i-\Gamma_{i+1}}.\]
It follows from the lower bound on the size of $D$ that there exists a $C\ls_{\epsilon\tau}\epsilon^{-2}\tau^{-1/k}$ such that, if 
\[D' = \{ x \in D : \ind{\Delta_i}\circ \ind{\Delta_i+\Gamma_i-\Gamma_{i+1}}(x) \leq C\eta_i \abs{\Delta}\},\]
then $\abs{D'}\geq \tfrac{1}{2}\abs{D}$.  Furthermore, if $x\in D$ then 
\[\ind{\Delta_{i}}\circ \ind{\Delta_i+\Gamma_i-\Gamma_{i+1}}(x)\geq \ind{\Delta_{i+1}}\circ \ind{\Delta_i+\Gamma_i-\Gamma_{i+1}}(x) \geq \eta_i\abs{\Delta}.\]
In particular, by dyadic pigeonholing, there exists some $\delta_i'$ with $C\eta_i \geq \delta_i'\geq \eta_i$ such that the set
\[S_i'=\{ x : 2\delta_i'\abs{\Delta}> \ind{\Delta_{i}}\circ \ind{\Delta_i+\Gamma_i-\Gamma_{i+1}}(x)\geq \delta_i'\abs{\Delta}\}\]
has size $\gs_{\epsilon\tau} \epsilon^4\tau^{1/k}\tau\eta_i^{-2}\abs{\Delta}\abs{\Gamma_i-\Gamma_{i+1}}$. 

If $\eta_i\leq \tfrac{\nu}{C}\delta_i$ for some $1\leq i<n$, therefore, we are in the first case of the lemma, with the set $S_i$ being replaced by $S_i'$ and all the auxiliary $\Delta_j$ for $j\leq i$ and $S_j$ for $j<i$ remaining the same. 

Otherwise, we have that $\eta_i \gg \epsilon^{O(1)} \nu \delta_i$ for all $1\leq i<n$. Note that, since $F_i(a,b)\leq \ind{\Delta_{i+1}}\circ \ind{\Delta_{i+1}+\Gamma_{i+1}}(b)\ll \delta_{i+1}\abs{\Delta}$ for all $(a,b)\in S_i\times S_{i+1}$, this in particular implies that $\epsilon^{O(1)}\nu \delta_i \ll \delta_{i+1}$ for all $1\leq i<n$. Furthermore, since $a\in S_i$, we have
\[F_i(a,b)\leq \ind{\Delta_{i+1}}\circ \ind{\Delta_i+\Gamma_i}(a)\ll \delta_i\abs{\Delta},\]
and so $\eta_i \ll \delta_i$. Therefore, with $G_i\subset S_i\times S_{i+1}$ as above, we have
\[\abs{G_i}\gs_{\epsilon\tau} \delta_{i+1}\abs{S_i}\abs{S_{i+1}}\]
and if $(a,b)\in G_i$ then 
\[F_i(a,b) = \ind{\Delta_{i,a}}\circ \ind{\Delta_{i+1}+\Gamma_{i+1}}(b)\gg \epsilon^{O(1)} \nu \delta_i\abs{\Delta},\]
where $\Delta_{i,a} = \Delta_{i+1}\cap(\Delta_i+\Gamma_i+a)$. Note that, since $a\in S_i$, we have $\abs{\Delta_{i,a}}\ll \delta_i\abs{\Delta}$. 

Applying the pigeonhole principle to $G_i$ there must exist some $a\in S_i$ which appears in many pairs $(a,b)\in G_i$ -- that is, there is some $S_{i+1}'\subset S_{i+1}$ such that 
\[\abs{S_{i+1}'}\gs_{\epsilon\tau} \delta_{i+1}\abs{S_{i+1}}\gs_\tau \epsilon\tau \delta_{i+1}\abs{\Delta}\abs{\Gamma_{i+1}}\]
 with $(a,b)\in G_i$ for all $b\in S_{i+1}'$. Discarding elements if necessary, we may suppose that
 \[\abs{S_{i+1}'}\leq \tau \delta_{i+1}\abs{\Delta}\abs{\Gamma_{i+1}}\]
 also. We let $\Delta_i'=\Delta_{i,a}$, so that
\[\Inn{\ind{S_{i+1}'}, \ind{\Delta_i'}\circ \ind{\Delta_{i+1}+\Gamma_{i+1}}}\gg \epsilon^{O(1)}\nu \delta_i\abs{\Delta}\abs{S_{i+1}'},\]
and we are in the second case of the lemma.
\end{proof}

We now couple this with Lemma~\ref{lemma:visc2} and a pigeonholing argument to obtain the following.

\begin{lemma}\label{lemmatoit}
There is a constant $C>0$ such that the following holds. Let $h,t,k\geq 2$ and $\tau\in (0,1)$. Let $\tilde{\Gamma}$ be an additive framework of height $h$ and tolerance $t$ and let $\Delta$ be $(\tau,k)$-additively non-smoothing relative to $\tilde{\Gamma}$.

Let $\nu\in[0,1]$ be some parameter, and suppose that $\vec{\Gamma}\in\mathcal{S}^h$ satisfies $2\Gamma_i-2\Gamma_{i+1}\subset \Gamma_{\mathrm{top}}$ and  $\abs{\Gamma_i-\Gamma_{i+1}}\leq 2\abs{\Gamma_i}$ for all $1\leq i<h$. Suppose further that $\epsilon \leq \min(\tfrac{1}{2},\tau^{1/k})$. 

If $\Delta$ has viscosity $\epsilon$ at depths $(\vec{\delta}, \vec{\Gamma})$ then either
\begin{enumerate}
\item there exists $1\leq i<h$ and $(\vec{\delta}',\vec{\Gamma}')\in [0,1]^h\times \mathcal{S}^h$ such that $\Delta$ has viscosity $\epsilon^{C^h}$ at depths $(\vec{\delta}',\vec{\Gamma}')$ with 
\[\delta_j'=\delta_j\textrm{ for }1\leq j<i\textrm{ and }\delta_i'\leq \nu \delta_i,\]
and
\[\Gamma_j'=\Gamma_j\textrm{ for }1\leq j<i,\]
\[\Gamma_i'=\Gamma_i+\Gamma_{i+1},\textrm{ and }\Gamma_j'=\Gamma^{(j)}\textrm{ for }i<j\leq h,\]
\listintertext{or}
\item there exists $1\leq i<n$ together with $1\geq \delta\geq \epsilon^{O(1)}\tau$, such that there is $\Delta'\subset \Delta$ with 
\[\abs{\Delta'}\ll \epsilon^{-O(1)}\nu^{-1}\delta\abs{\Delta}\]
and
\[S\subset \{ x : \ind{\Delta}\circ \ind{\Delta+\Gamma_{i+1}}(x) \geq \delta \abs{\Delta}\}\]
with
\[\tau\delta^{-1}\abs{\Delta}\abs{\Gamma_{i+1}}\geq \abs{S}\geq \epsilon^{O(1)}\tau\delta^{-1}\abs{\Delta}\abs{\Gamma_{i+1}}\]
such that 
\[\Inn{\ind{\Delta'}\circ \ind{\Delta+\Gamma_{i+1}}, \ind{S}}\geq \epsilon^{O(1)}\tau^{O(1/h)}\nu\delta\abs{\Delta}\abs{S}.\]
\end{enumerate}
\end{lemma}
\begin{proof}
We apply Lemma~\ref{lemma:visc3} to $\Delta$, with $n=h$. Suppose that we are in the first case, so that there exists some $1\leq i<h$ together with $\delta_i'\leq \nu \delta_i$ such that $\Delta$ has viscosity $\epsilon^{O(1)}$ at depths $(\vec{\delta}',\vec{\Gamma}')$ with
\[\vec{\delta}' = (\delta_1,\ldots,\delta_{i-1},\delta_i')\]
and
\[\vec{\Gamma}' = (\Gamma_1,\ldots,\Gamma_{i-1},\Gamma_i+\Gamma_{i+1}).\]
We now extend this scale vector (which has length $i$) to one of length $h$ using repeated applications of Lemma~\ref{lemma:visc2}, which produces the first case. 

Suppose then that we are in the second case of Lemma~\ref{lemma:visc3}. By the definition of multiscale viscosity all the $\delta_i$ in the depth vector $\vec{\delta}$ must lie in the range $[\epsilon \tau,1]$. In particular, by the pigeonhole principle, there must exist some $1\leq i<h$ such that $\delta_{i+1}\leq (\epsilon\tau)^{-\frac{1}{h-1}}\delta_i$. The second case of the lemma now follows, using the fact that $\Delta_{i+1}\subset \Delta$, choosing $S=S_{i+1}'$ and $\delta=\delta_{i+1}$, and recalling the bound in Lemma~\ref{lemma:visc3} that $\delta_{i+1}\gg \epsilon^{O(1)}\nu\delta_i$. 
\end{proof}

We will now prove Lemma~\ref{lemma:structpiece} by iteratively applying Lemma~\ref{lemmatoit}.

\begin{proof}[Proof of Lemma~\ref{lemma:structpiece}]

Note that the second case of Lemma~\ref{lemmatoit} immediately implies the conclusion of Lemma~\ref{lemma:structpiece} with $\Gamma=\Gamma_{i+1}$ and the implied constants being bounded by $\epsilon^{-O(1)}\nu^{-1}\tau^{-O(1/h)}$. It remains to verify that this iteration must exit in the second case at some point, with an appropriate choice of $\nu$, and with $\epsilon$ being appropriately bounded (and with $\Gamma_{i+1}$ being of the specified form).

We now explain the iteration process. Let $\nu>0$ be some fixed parameter, to be chosen later, and let $C>0$ be some large absolute constant, also to be chosen later. We will recursively define a sequence of triples $(\epsilon_j,\vec{\delta}_j,\vec{\Gamma}_j)$ for $j= 0,1,\ldots$ such that
\begin{enumerate}
\item $(\vec{\delta}_j,\vec{\Gamma}_j)$ is a depth vector\footnote{We caution the reader that $\vec{\delta}_j$ does not mean the $j$th component of the vector $\vec{\delta}$, but is itself a vector in $[0,1]^h$.} of length $h$,
\item $\Delta$ has viscosity $\epsilon_j$ at depths $(\vec{\delta}_j,\vec{\Gamma}_j)$, 
\item $\epsilon_0\leq \min(\tfrac{1}{2},\tau^{1/k})$ and $\epsilon_{j+1}=\epsilon_j^{C^h}$,
\item for each $j\geq 1$ there exists some $1\leq i<h$ such that
\[\delta_{j,r}=\delta_{(j-1),r}\textrm{ for }1\leq r<i\textrm{ and }\delta_{j,i}\leq \nu \delta_{(j-1),i}\]
and
\item for $1\leq i\leq h$,
\[\Gamma_{j,i} = \Gamma^{(i)}+\ell_{j,i,1}\Gamma^{(i+1)}+\cdots +\ell_{j,i,(h-i)}\Gamma^{(h)}\]
where $\ell_{j,i,r}\geq 0$ are integers satisfying $\sum_i\sum_{r=1}^{h-i}\ell_{j,i,r} \leq j$.
\end{enumerate}

We generate $(\epsilon_0,\vec{\delta}_0,\vec{\Gamma}_0)$ by applying Lemma~\ref{lemma:visc1} to find viscosity at some depth vector of length 1, which we then extend to length $h$ via repeated applications of Lemma~\ref{lemma:visc2}. In particular, $\Delta$ has viscosity $\epsilon_0$ at some depth vector $(\vec{\delta}_0,\vec{\Gamma}_0)$ of length $h$, with $\Gamma_{0,i}=\Gamma^{(i)}$, where $\epsilon_0= \min(\tfrac{1}{2},\tau^{C^h/k})$, provided $C$ is chosen sufficiently large. Note that $\tau^{C^h/k}\geq \tau$ provided $h\leq c\log k$ for some sufficiently small constant $c>0$, which we can ensure by the hypotheses of Lemma~\ref{lemma:structpiece}.

Suppose then that $j\geq0$ and we have constructed $(\epsilon_j,\vec{\delta}_j,\vec{\Gamma}_j)$ satisfying the above conditions. Suppose that 
\[4j+2\leq t.\]
We will apply Lemma~\ref{lemmatoit} to this triple. We first verify that the hypotheses of Lemma~\ref{lemmatoit} hold. That $\epsilon_j\leq \tfrac{1}{2}\tau^{1/k}$ holds immediately. By condition (5), for $1\leq i\leq h$,
\[\Gamma_{j,i}\subset \Gamma^{(i)}+j \Gamma^{(i+1)}\subset (j+1)\Gamma^{(i)}.\]
In particular, by the definition of additive framework, for $1\leq i<h$,
\begin{align*}
2\Gamma_{j,i}-2\Gamma_{j,i+1}
&\subset 2\Gamma^{(i)}+t\Gamma^{(i+1)}\\
&\subset 3\Gamma^{(i)}\\
&\subset \Gamma_{\mathrm{top}}.
\end{align*}
Similarly, for $1\leq i<h$,
\begin{align*}
\abs{\Gamma_{j,i}+\Gamma_{j,i+1}}
&\leq \Abs{\Gamma^{(i)}+t\Gamma^{(i+1)}}\\
&\leq 2\Abs{\Gamma^{(i)}}\\
&\leq 2\abs{\Gamma_{j,i}}.
\end{align*}
Thus all the conditions of Lemma~\ref{lemmatoit} are satisfied. If the second conclusion of Lemma~\ref{lemmatoit} holds, then we stop the construction at $(\epsilon_j,\vec{\delta}_j,\vec{\Gamma}_j)$. As we shall see, in this case, we have satisfied the conclusion of Lemma~\ref{lemma:structpiece} as required.

Suppose then that the first conclusion of Lemma~\ref{lemmatoit} holds. We claim that this produces a new triple $(\epsilon_{j+1},\vec{\delta}_{j+1},\vec{\Gamma}_{j+1})$ that satisfies the conditions above. Indeed, this first conclusion produces some such triple, with
\[\epsilon_{j+1} = (\epsilon_j)^{C^h},\]
such that $\Delta$ has viscosity $\epsilon_{j+1}$ at depths $(\vec{\delta}_{j+1},\vec{\Gamma}_{j+1})$. Condition (4) is part of the conclusion of Lemma~\ref{lemmatoit}. Finally, to check condition (5), we note that there exists some $1\leq i_0< h$ such that $\Gamma_{j+1,i}=\Gamma_{j,i}$ for $1\leq i<i_0$, that $\Gamma_{j+1,i}=\Gamma^{(i)}$ for $i_0<i\leq h$, and 
\begin{align*}
\Gamma_{j+1,i_0}
&=\Gamma_{j,i_0}+\Gamma_{j,i_0+1}\\
&=\Gamma^{(i_0)}+\brac{\sum_{r=1}^{h-i_0}\ell_{j,i_0,r}\Gamma^{(i_0+r)}}+\Gamma^{(i_0+1)}+\brac{\sum_{r=1}^{h-i_0-1}\ell_{j,i_0+1,r}\Gamma^{(i_0+1+r)}}\\
&=\Gamma^{(i_0)}+\sum_{r=1}^{h-i_0}\ell_{j+1,i_0,r}\Gamma^{(i_0+r)},
\end{align*}
say. In particular,
\[\sum_{1\leq i\leq h}\sum_{r=1}^{h-i}\ell_{j+1,i,r}\leq \brac{\sum_{1\leq i\leq i_0+1}\sum_{r=1}^{h-i}\ell_{j,i,r}}+1\leq j+1,\]
as required.

We have thus shown that given a sequence of triples $(\epsilon_i,\vec{\delta}_{i},\vec{\Gamma}_{i})$ for $1\leq i\leq j$ which satisfies the conditions (1)-(5) above we can apply Lemma~\ref{lemmatoit}, the first conclusion of which extends this sequence by a new triple $(\epsilon_{j+1},\vec{\delta}_{j+1},\vec{\Gamma}_{j+1})$, such that the new sequence also satisfies conditions (1)-(5). 

We will show below that, with a suitable choice of $\nu$ (the parameter which appears in condition (4)), this constructive process must halt in at most $j$  steps, where $j$ satisfies
\begin{equation}\label{eq-strucref}
C^{h(j+1)}\leq k^{1/2}\textrm{ and }4j+2\leq t.
\end{equation}
These bounds imply that 
\[\epsilon_j\geq \epsilon_0^{C^{hj}}\geq 2^{-k^{1/2}}\tau^{1/k^{1/2}}.\]
In particular, at such $j$, the condition $4j+2\leq t$ required for the above construction is met, and hence the only reason that the constructive process cannot continue is that the second conclusion of Lemma~\ref{lemmatoit} holds instead. This is exactly the conclusion of Lemma~\ref{lemma:structpiece}, with the implicit errors polynomial in $\epsilon_j \nu\tau^{1/h}$. Since $\epsilon_j\geq 2^{-k^{1/2}}\tau^{1/k^{1/2}}$, these errors are polynomial in $2^{-k^{1/2}}\tau^{1/k^{1/2}+1/h}\nu$. Our choice of $\nu$ will satisfy $\nu\geq 2^{-O(k)}\tau^{O(1/\log\log k)}$, and hence this error is acceptable for the conclusion of Lemma~\ref{lemma:structpiece}.

It remains to explain how we choose $\nu$ such that this process halts in at most $j$ steps, where \eqref{eq-strucref} is satisfied. Let
\[N = \left\lfloor \brac{c'\frac{\log k}{h}}^{1/h}\right\rfloor,\]
for some small absolute constant $c'>0$, and choose
\[\nu=(2^{-k^{1/2}}\tau^2)^{1/N}.\]
We will show that, provided $c'>0$ is sufficiently small, this constructive process must halt in at most $N^{h-1}$ steps. Note that our conditions on $h$ and $t$ guarantee that 
\[C^{h(N^{h-1}+1)}\leq k^{1/2}\textrm{ and }4N^{h-1}+2\leq t\]
as required, provided that we choose $c'>0$ sufficiently small (depending on $C$) and the constant in the statement of Lemma~\ref{lemma:structpiece} sufficiently large (depending on $C$ and $c'$). Furthermore, our upper bound on $h$ ensures that $N\gg \log\log k$, and so $\nu\geq 2^{-O(k)}\tau^{O(1/\log\log k)}$ as required.

Finally, suppose, for a contradiction, that we have carried out this constructive process at least $N^{h-1}$ times. We will use the following elementary combinatorial lemma.

\begin{lemma}\label{lemma:combin}
Let $N,r\geq 1$ and suppose that $n\geq N^r$. If we colour $\{1,\ldots,n\}$ by the integers $\{1,\ldots,r\}$ then there is some interval $I\subset \{1,\ldots,n\}$ and some $1\leq i\leq r$ such that $I$ contains at least $N$ integers coloured $i$, and no integers coloured $j<i$.
\end{lemma}
\begin{proof}
We use induction on $r$. It is clear that, when there is only one colour, $n=N$ suffices. Suppose then that $r\geq 2$, and we have coloured $\{1,\ldots,n\}$ with the colours $\{1,\ldots,r\}$. If there are $N$ integers in $\{1,\ldots,n\}$ all receiving the colour $1$, then we are done. Otherwise, by the pigeonhole principle, we can find some subset of consecutive integers of size $\geq n/N$ which contains only the colours $\{2,\ldots,r\}$. By induction, we are done, provided $n/N\geq N^{r-1}$.
\end{proof}

We colour the integers $\{1,\ldots,N^{h-1}\}$ by the integers $\{1,\ldots,h-1\}$ by assigning the colour $i$ to $j$ if the triple $(\epsilon_j,\vec{\delta}_j,\vec{\Gamma}_j)$ satisfies condition (4) with this $i$. By Lemma~\ref{lemma:combin} there exists some $1\leq i<h$ and interval $I\subset \{1,\ldots,N^{h-1}\}$ such that at least $N$ many $j\in I$ are coloured $i$ and no integers in $I$ are coloured $i'<i$.

Let $(\epsilon,\vec{\delta},\vec{\Gamma})$ be the constructed triple at the first element of $I$, and $(\epsilon',\vec{\delta}',\vec{\Gamma}')$ be the constructed triple at the final element of $I$. We claim that
\[\delta_i'\leq \nu^N\delta_i.\]
Indeed, an occurrence of condition (4) for some $i'>i$ does not change $\delta_i$, and an occurrence of condition (4) for $i$ (which must happen at least $N$ times, by construction of $I$) reduces $\delta_i$ by a factor of $\nu$.

In particular, $\delta_i'\leq \nu^N$. As noted in the definition of multiscale viscosity, however, we must have $\delta_i'\geq \epsilon'\tau$. Therefore 
\[2^{-k^{1/2}}\tau^2\geq \nu^N\geq \epsilon'\tau\geq 2^{-k^{1/2}}\tau^{1+1/k^{1/2}},\]
which is a contradiction. The proof of Lemma~\ref{lemma:structpiece}, and hence the proof of the structural theorem, is (at last!) complete.
\end{proof}
\section{Spectral boosting}\label{section:boost} 

We now come to the final substantial part of the proof. The goal of this section is to convert the additive non-smoothing data from the final conclusion of Proposition~\ref{prop:big}, combined with the structural output of Theorem~\ref{th:structure}, into a suitable density increment using a `spectral boosting' argument.

Roughly speaking, a representative outcome of Proposition~\ref{prop:big} combined with Theorem~\ref{th:structure} is some pair $X,H\subset \Delta_\alpha(A)$ of sizes $\abs{X}\approx \delta^{-1}\alpha^{-1}$ and $\abs{H}\approx \delta \alpha^{-3}$, for some $1 \gg \delta\gg \alpha^2$, such that the (relative) energy between $X$ and $H$ is $\gg \abs{X}\abs{H}^2$. Furthermore, since $H$ is contained in a large symmetry set, we know (e.g. by Lemma \ref{lemma:dimsymmetry}) that the (relative) dimension of $H$ is $\ll 1$. 

If we just use the fact that $H$ is a low-dimensional subset of $\Delta_\alpha(A)$ (ignoring its interaction with $X$), then we can deduce that $A$ has a density increment on $H^\perp$ of strength $[\delta \alpha^{-1},1]$, by the $L^2$ increment argument surrounding \eqref{L2inc_model}.\footnote{The general (relative) case is Lemma \ref{lemma:L2inc}.} If $\delta \gg \alpha$ then such an increment is sufficiently strong for our overall argument. Moreover, by iteratively pulling out several disjoint such $H$ from $\Delta_\alpha(A)$ and taking their union, the $L^2$ increment argument can deliver a density increment of strength $[\delta\alpha^{-2+c},\alpha^{-1+c}]$ for some $c>0$, say, where $\delta$ is the smallest parameter for the different $H$'s. This is strong enough for our purposes, except for the case when $\delta\approx \alpha^2$.

We thus need to be able to deal with the case where we obtain some $X, H \subset \Delta_\alpha(A)$ as above, where the parameter $\delta$ is about $\alpha^2$. For this we use a technique we call spectral boosting, in which we use information about the energies of $X$ and $H$ to obtain an increment that is of strength comparable to what the $L^2$ increment argument would give if $H$ was a subset of $\Delta_{\alpha^{1/2}}(A)$ rather than $\Delta_\alpha(A)$. In other words, if $X$ and $H$ satisfy certain additive properties, then we can `boost' the spectral level of $H$ from $\alpha$ to $\alpha^{1/2}$ from the point of view of obtaining a density increment.

To give this some additional context, note that (the very strong) energy condition $E(X,H) \gg \abs{X}\abs{H}^2$ implies, loosely speaking, that much of the mass of $H-H$ is supported on a translate of $X$, and hence on some translate of $\Delta_\alpha(A)$. Formally, 
\[ \sum_{\gamma \in \Delta_\alpha(A) - \xi} \ind{H} \circ \ind{H}(\gamma) \gg \abs{H}^2 \]
for some $\xi$. This is rather unexpected when $H$ is a subset of $\Delta_\alpha(A)$, but if $H \subset \Delta_{2\alpha^{1/2}}(A)$ instead then an argument of Bourgain~\cite{Bo:2005} (see also \cite[Lemma 4.37]{TaVu:2006}) shows that the left-hand side (with $\xi$ trivial) is indeed of order $\abs{H}^{2}$. The arguments of this section are thus an attempt at forming something of a converse to this implication of Bourgain's, at least as applied to obtaining density increments.

To make this rigorous, it turns out that we will need to control two additional quantities in addition to knowing the near-maximality of $E(X,H)$. Firstly, we need to know that the higher energy $E_{2m}(X)$ of $X$ is not too large. Given enough control over both $E(X,H)$ and $E_{2m}(X)$ we can obtain a strong conclusion about the structure of $A$, but with discrepancy-type information rather than increment-type information -- see the outcome of Lemma~\ref{lemma:modelspecboost} below. In order to convert such a discrepancy into a genuine increment, we will furthermore need to assume that $\norm{\mu_A\circ\mu_A}_{2m}$ is not too large; see Lemma \ref{lemma:modelspecboost2}. We begin below by reviewing the model version of the argument, where $G = \bbf_p^n$, before going into the general case where we need to work relative to Bohr sets.

\subsection*{Spectral boosting: a model version}
Here we give a simplified version of the arguments behind spectral boosting when $G=\bbf_p^n$. In particular, we will not need to work relative to Bohr sets in these arguments, which clears away many obscuring technicalities. In this subsection, when $H\subset \bbf_p^n$, the dimension of $H$, denoted by $\dim(H)$, will mean the size of the largest subset of $H$ which is linearly independent over $\bbf_p$.

This subsection is not logically necessary, as we will prove everything in it in much greater generality in the following subsection, but reading these proofs in the model setting first should give a much clearer explanation of the key ideas involved. Both the statement and proofs of the full spectral boosting lemmas we require for our main theorem are essentially those presented in this subsection, `modulo Bohr set technicalities'. 

For a self-contained proposition illustrating the outcome of spectral boosting in this model setting, the reader might like to look ahead to Proposition~\ref{sbmodelprop}. We believe that its two constituent parts, Lemmas \ref{lemma:modelspecboost} and \ref{lemma:modelspecboost2}, might however be more useful to bear in mind for potential future applications. 

\begin{lemma}\label{lemma:modelspecboost}
Suppose that $A\subset \bbf_p^n$ has density $\alpha$. Let $\kappa,\eta\in (0,1]$ and $K,m\geq 2$. Let $X\subset \Delta_\eta(A)\backslash\{0\}$ and $H$ be such that
\begin{enumerate}
\item 
\[E_{2m}(X)\leq (\kappa \abs{X})^{2m}\]
\listintertext{and}
\item
\[\Inn{\ind{X}\circ \ind{X},\ind{H}\circ \ind{H}}\geq K^{-1}\abs{X}\abs{H}^2.\]
\end{enumerate}
Then 
\[\langle \abs{\mu_A\circ \mu_A-1}, \Abs{\widecheck{\ind{H}}}^2\rangle\geq (\eta\alpha)^{O(1/m)}K^{-O(1)}\kappa^{-1}\eta^2\abs{H}^2.\]
\end{lemma}
For our application, we roughly have $K=O(1)$, $m\approx \log(1/\eta\alpha)$, $\kappa \approx \eta$, and $\abs{H}\gg \eta^{-1}$, so that the lower bound in the conclusion is $\gg \kappa^{-1}\eta^2 \abs{H}^2\gg \abs{H}$. If, furthermore, the absolute value signs around $\mu_A\circ\mu_A-1$ were not present, then this would quite directly give a strong (physical-side) $L^2$ density increment: if
\[ \langle \mu_A \circ \mu_A, \Abs{\widecheck{\ind{H}}}^2 \rangle \geq (1+c)\abs{H} \]
then
\[ \langle \mu_A \circ \mu_A, \Abs{\widecheck{\ind{H}}}^2*\mu_V \rangle \geq (1 + c)\abs{H} \]
where $V = H^\perp$, so that $\widecheck{\ind{H}}$ is invariant under shifts by elements of $V$, and $V$ has codimension $\dim(H)$. Thus
\[ \norm{ \mu_A * \mu_V }_\infty \geq \norm{ \mu_A \circ \mu_A * \mu_V }_\infty \geq 1 + c. \]
The presence of the absolute value signs poses an obstacle to this argument, of course, but Lemma \ref{lemma:modelspecboost2} shows how, under a suitable bound for $\norm{\mu_A\circ\mu_A}_{2m}$, we can convert the discrepancy-type conclusion of the lemma into a genuine increment.
\begin{proof}
Let $f=\mu_A\circ \mu_A-1$ so that
\[\widehat{f}(\gamma)=
\begin{cases}\abs{\widehat{\mu_A}(\gamma)}^2&\textrm{ when }\gamma\neq0\textrm{ and}\\
0&\textrm{ when }\gamma=0.
\end{cases}\]
In particular, $\widehat{f}(\gamma)\geq 0$ for all $\gamma$ and furthermore we have  $ \widehat{f}(\gamma)\geq \eta^2\ind{X}(\gamma)$. 

We have, with $g=f\cdot \Abs{\widecheck{\ind{H}}}^2$, 
\begin{align*}
\Inn{\widehat{g}, \ind{X}}
&=\Inn{\widehat{f}\ast \ind{H}\circ \ind{H}, \ind{X}}\\
&\geq \eta^2\Inn{\ind{H}\circ \ind{H}, \ind{X}\circ \ind{X}}\\
&\geq K^{-1}\eta^2\abs{X}\abs{H}^2.
 \end{align*}
By H\"{o}lder's inequality followed by the Cauchy--Schwarz inequality 
\begin{align*}
\Inn{\widehat{g}, \ind{X}}^{2m}
&=\Inn{g,\widecheck{\ind{X}}}^{2m}\\
&\leq \norm{g}_1^{2m-2}\langle \abs{g},\Abs{\widecheck{\ind{X}}}^m\rangle^2\\
&\leq  \norm{g}_1^{2m-2}\norm{g}_2^2E_{2m}(X).\end{align*}
By assumption, $E_{2m}(X)\leq(\kappa \abs{X})^{2m}$. It follows that
\[\norm{g}_1^{2m-2}\norm{g}_2^2
\geq (K^{-1}\kappa^{-1}\eta^2\abs{H}^2)^{2m}.\]
Since
\begin{align*}
\norm{g}_2^2
&\leq \norm{f}_2^2\Norm{\widecheck{\ind{H}}}_\infty^4\\
&\leq \norm{\mu_A\circ \mu_A}_2^2\abs{H}^4\\
&\leq \alpha^{-1}\abs{H}^4
\end{align*}
it follows that, taking the $(2m-2)$th root,
\[\Inn{\abs{f},\Abs{\widecheck{\ind{H}}}^2}=\norm{g}_1
\geq \brac{K^{-1}\kappa^{-1} \eta^2\alpha^{1/2}}^{\frac{1}{m-1}} K^{-1}\kappa^{-1}\eta^2\abs{H}^2,\]
and the lemma follows.
\end{proof}

The second component of spectral boosting is the conversion of the discrepancy-type conclusion of Lemma \ref{lemma:modelspecboost} into a genuine increment. The key input for this is the following $L^\infty$ almost-periodicity result proved in \cite{ScSi:2016}.

\begin{lemma}[\cite{ScSi:2016}, Theorem 3.2]\label{lemma:modelap}
Let $\epsilon \in (0,1/2)$. Let $S, M, L\subset \bbf_p^n$, where $S$ has density $\sigma$, and let $\nu \coloneqq \abs{M}/\abs{L}$. Then there is a subspace $V\leq \bbf_p^n$ of codimension  
 \[ \mathrm{codim}(V) \ls_{\nu\sigma\epsilon} \epsilon^{-2} \]
such that
\[ \norm{ \mu_S \ast \mu_M \ast 1_L \ast \mu_{V} - \mu_S\ast \mu_M \ast 1_L }_\infty \leq \epsilon. \]
\end{lemma}
 
\begin{lemma}\label{lemma:modelspecboost2}
Suppose that $A\subset \bbf_p^n$ has density $\alpha$. If there is a set $H$ and $\delta\in(0,1)$ such that
\[\Inn{\abs{\mu_A\circ \mu_A-1}, \Abs{\widecheck{\ind{H}}}^2}\geq \delta\abs{H}\]
and $K\geq 1$ is such that
\[\norm{\mu_A\circ \mu_A}_{2m}\leq K\]
for some $m\geq \log(4\abs{H}/\alpha\delta)$, then there is a subspace $W\leq \bbf_p^n$ of codimension
\[\mathrm{codim}(W)\leq \dim(H) + \tilde{O}_{\alpha\delta/K}(\delta^{-4}K^{2})\]
such that $\norm{\ind{A}\ast \mu_{W}}_\infty \geq 1+2^{-6}\delta$.
\end{lemma}
Again, in our application we roughly have $\delta\gg 1$ and $K,\dim(H)\ll 1$, so this density increment is very strong, of strength $[1,1]$. 
\begin{proof}
Let $f=\mu_A\circ \mu_A-1$. We begin by noting that since $\widehat{f}\geq 0$ we have $\Inn{ f, \Abs{\widecheck{\ind{H}}}^2}= \Inn{ \widehat{f},\ind{H}\circ \ind{H}}\geq 0$. In particular, using the fact that $\max(x,0)=(x+\abs{x})/2$, the hypothesis yields
\[\Inn{\max(f,0),\Abs{\widecheck{\ind{H}}}^2}\geq \tfrac{1}{2}\delta\abs{H}.\]
We now note that since $\Norm{\widecheck{\ind{H}}^2}_1=\abs{H}$ the total contribution from those $x$ such that $f(x)< \tfrac{1}{4}\delta$ is at most $\tfrac{1}{4}\delta\abs{H}$. It follows that if
\[T = \{ x: f(x)\geq \tfrac{1}{4}\delta\}\]
(and in particular $f(x)\geq 0$ for all $x\in T$) then
\[\Inn{\ind{T}f,\Abs{\widecheck{\ind{H}}}^2}\geq\tfrac{1}{4}\delta\abs{H}.\]
We now use H\"{o}lder's inequality to bound the left-hand side above by
\[\norm{f}_{2m}\Inn{\ind{T},\Abs{\widecheck{\ind{H}}}^{2+\frac{2}{2m-1}}}^{1-1/2m}\leq \norm{f}_{2m}\brac{\frac{\abs{H}}{\Inn{\ind{T},\Abs{\widecheck{\ind{H}}}^2}^{1/2}}}^{1/m}\Inn{\ind{T},\Abs{\widecheck{\ind{H}}}^{2}}.\]
The first factor is, by assumption and the triangle inequality, at most $2K$, say. Furthermore, since $\norm{f}_\infty \leq \alpha^{-1}$, we know that $\Inn{\ind{T},\Abs{\widecheck{\ind{H}}}^2}\geq \tfrac{1}{4}\alpha\delta\abs{H}$. Provided $m$ is sufficiently large, therefore, we have
\[\Inn{ \ind{T}, \Abs{\widecheck{\ind{H}}}^2}\geq \frac{1}{16K}\delta\abs{H}.\]
In particular, if we let $V$ be the subspace which annihilates $H$ (so that $V$ has codimension at most $\dim(H)$) then 
\[\Inn{\ind{T}\ast \mu_V, \Abs{\widecheck{\ind{H}}}^2}\geq \frac{1}{16K}\delta\abs{H},\]
and so in particular there exists some $x$ such that
\[\ind{T}\ast \mu_V(x)\geq \delta/16K.\]
Recall that, by definition of $T$, for any subset $T'\subset T$ we have
\[\Inn{ \mu_{A}\circ \mu_{A},\ind{T'}}=\Inn{f, \ind{T'}}+\mu(T')\geq (1+\tfrac{1}{4}\delta)\mu(T').\]
Combining the previous two facts, there is some $T'\subset V$ of density $\mu_V(T')\geq \delta/16K$ and a translate of $A$, say $A'$, such that 
\[\Inn{\ind{T'}\ast \ind{A'},\ind{A}}=\alpha^2\Inn{\mu_{A}\circ\mu_{A'},\ind{T'}} \geq (1+\tfrac{1}{4}\delta)\alpha^2\mu(T').\]
Dividing the left-hand side into cosets of $V$, we deduce that 
\begin{equation}\label{eq:modelboost2}
\frac{1}{\abs{V}}\sum_{y\in\bbf_p^n}\Inn{\ind{T'}\ast \ind{A'},\ind{A\cap (V+y)}}\geq (1+\tfrac{1}{4}\delta)\alpha^2\mu(T').
\end{equation}
Let $Y\subset \bbf_p^n$ be the set of those $y\in \bbf_p^n$ such that $\abs{A\cap (V+y)}\geq 2^{-6}\delta\alpha^2 \abs{V}$. The contribution to the left-hand side of \eqref{eq:modelboost2} from those $y\not\in Y$ is at most 
\[\frac{p^n}{\abs{V}} \cdot \frac{2^{-6}\delta\alpha^2 \abs{V}}{p^n}\mu(T')=2^{-6}\delta\alpha^2\mu(T'),\]
and hence
\begin{equation}\label{eq:modelthis}
\frac{1}{\abs{V}}\sum_{y\in Y}\Inn{\ind{T'}\ast \ind{A'},\ind{A\cap (V+y)}}\geq (1+2^{-4}\delta)\alpha^2\mu(T').
\end{equation}
Similarly, 
\begin{align*}
\alpha
&=\frac{1}{p^n\abs{V}}\sum_{y\in \bbf_p^n}\abs{A\cap (V+y)}\\
&=\frac{1}{p^n\abs{V}}\sum_{y\in Y}\abs{A\cap (V+y)}+E
\end{align*}
where $E\leq 2^{-6}\delta\alpha^2\leq 2^{-6}\delta\alpha$. It follows that 
\begin{equation}\label{eq:modelthat}
 (1+2^{-4}\delta)\alpha^2\mu(T')\geq (1+2^{-5}\delta)\alpha \mu(T')\frac{1}{p^n\abs{V}}\sum_{y\in Y}\abs{A\cap (V+y)}.
 \end{equation}
Combining \eqref{eq:modelthis} and \eqref{eq:modelthat} and averaging over $y\in Y$ we find some $y\in Y$ such that, if we let $A''=A\cap (V+y)$, then 
\[\Inn{\ind{T'}\ast \ind{A'},\ind{A''}}\geq(1+2^{-5}\delta)\alpha \mu(A'')\mu(T').\]
Since $T'\subset V$ and $A'' \subset V+y$, we can replace $A'$ by $A'''=A'\cap(V+y)$ without affecting the value of the inner product. As an immediate consequence, $\abs{A'''} \geq (1+2^{-5}\delta)\alpha \abs{T'}$. We may assume that $\abs{A'''}\leq 2\alpha \abs{V}$, or else we are done, letting $W=V$. 

We now apply Lemma \ref{lemma:modelap} relative to $V$, with the choices
\[S=-A''\quad M= A'''\quad L = T'\textrm{ and }\epsilon = 2^{-7}\delta\mu_{V}(T'),\]
 noting that we have $\sigma=\mu_V(A'')\geq 2^{-6}\delta\alpha^2$, since $y\in Y$, and 
\[\nu= \frac{\abs{A'''}}{\abs{T'}}\geq \alpha.\]
We therefore produce a new subspace $W\leq V$, of codimension in $V$
\begin{align*}
\mathrm{codim}_V(W)
&\ls_{\delta\alpha\mu_V(T')}(\delta\mu_{V}(T'))^{-2}\\
&\ls_{\delta\alpha/K} (\delta^2/K)^{-2}
\end{align*}
such that
\begin{align*}
\norm{\ind{T'}\circ \ind{A''}\ast \ind{A'''}\ast \mu_{W}-\ind{T'}\circ \ind{A''}\ast \ind{A'''}}_\infty 
&\leq \epsilon\mu(A'')\mu(A''')\\
&= 2^{-7}\delta\mu_{V}(T')\mu(A'')\mu(A''')\\
&\leq 2^{-6}\delta\alpha \mu(A'')\mu(T'),
\end{align*}
and so
\[\Inn{\ind{T'},\ind{A''}\circ \ind{A'''}\ast \mu_{W}} \geq (1+2^{-6}\delta)\alpha\mu(A'')\mu(T').\]
The claim follows after bounding the left-hand side above by 
\[\norm{\ind{A''}\circ \ind{T'}}_1\norm{\ind{A'''}\ast \mu_{W}}_\infty=\mu(A'')\mu(T')\norm{\ind{A'''}\ast \mu_{W}}_\infty,\]
and recalling that $A'''$ is a subset of some translate of $A$.
\end{proof}

By combining Lemma~\ref{lemma:modelspecboost} with Lemma~\ref{lemma:modelspecboost2}, and making some simplifying choices of parameters, we obtain the following proposition, which illustrates what spectral boosting allows us to deduce. 

\begin{proposition}\label{sbmodelprop}
There is a constant $C>0$ such that the following holds. Suppose that $A\subset \bbf_p^n$ has density $\alpha$. Let $\eta\in (0,1]$ and $m\geq C\log(2/\alpha\eta)$. 

Suppose that there are $X\subset \Delta_\eta(A)\backslash\{0\}$ and $H\subset \bbf_p^n$ such that, for some parameter $K\geq 2$, 
\begin{enumerate}
\item $\abs{H} \geq K^{-1}\eta^{-1}$, 
\item $E_{2m}(X)\leq (K\eta\abs{X})^{2m}$,
\item $\Inn{\ind{X}\circ \ind{X},\ind{H}\circ \ind{H}}\geq K^{-1}\abs{X}\abs{H}^2$, and
\item $\norm{\mu_A\circ \mu_A}_{2m}\leq K$.
\end{enumerate}
Then there is a subspace $W\leq \bbf_p^n$ of codimension
\[\mathrm{codim}(W)\leq \dim(H) + \tilde{O}_{\alpha/K}(K^{O(1)})\]
such that 
\[\norm{\ind{A}\ast \mu_{W}}_\infty \geq 1+K^{-O(1)}.\]
\end{proposition}

\subsection*{Spectral boosting: the full version}

We now turn to proving the full technical generality of spectral boosting that we require for our application. We stress that all of the ideas in this section are essentially the same as those in the model case $G=\bbf_p^n$, `modulo Bohr set technicalities'. The reader is strongly advised to read the proofs of the previous subsection before following those presented here. 

We first state and prove the generalised version of Lemma~\ref{lemma:modelspecboost}. Recall that $X$ being $\Gamma$-orthogonal means that $\ind{X+\Gamma}=\ind{X}\ast \ind{\Gamma}$. 
\begin{lemma}\label{lemma:specboost1}
There is a constant $C>0$ such that the following holds. Let $B,B',B''$ be any symmetric sets, and suppose that $A\subset B$ has density $\alpha\in(0,1/2]$. Let $\Gamma=\Delta_{1/2}(B')$. Let $K\geq 2$ be some parameter. Suppose that $X,H\subset \widehat{G}$ are such that
\begin{enumerate}
\item $X$ is $\Gamma$-orthogonal, 
\item if $\gamma\in X+\Gamma$ then $\abs{\widehat{\bal{A}{B}}}^2\circ \abs{\widehat{\mu_{B''}}}^2(\gamma)\geq\eta^2\mu(B)^{-1}$, 
\item there is some $m\geq 2$ and $\kappa\in(0,1)$ such that 
\[E_{2m}(X; \abs{\widehat{\mu_{B''}}}^2)\leq (\kappa\abs{X})^{2m},\]
\listintertext{and}
\item 
\begin{enumerate}
\item $\Norm{\ind{H}\ast\abs{\widehat{\mu_{B''}}}^2}_\infty \leq 2$ and
\item \[\Inn {\ind{X}\circ \ind{X},\ind{H}\circ \ind{H}\circ \ind{\Gamma}}\geq K^{-1}\abs{X}\abs{H}^2.\]
\end{enumerate}
\end{enumerate}
Then 
\[\Inn{\Abs{\bal{A}{B}\circ \bal{A}{B}}, \Abs{\widecheck{\ind{H}}}^2(\mu_{B''}\circ \mu_{B''})}\geq (\eta\alpha)^{O(1/m)}K^{-O(1)}\kappa^{-1}\eta^2\abs{H}^2\mu(B)^{-1}.\]

\end{lemma}
The hypotheses may seem opaque, but they are natural generalisations of the hypotheses in Lemma~\ref{lemma:modelspecboost}. For example, condition (2) says that $X+\Gamma$ behaves like an `analytically smoothed' spectrum of $\bal{A}{B}$ at level $\eta$. Condition 4(a) is saying that $H$ satisfies a strong orthogonality condition with respect to the smoothing factor used in (2).

\begin{proof}
For brevity, let $f=\bal{A}{B}\circ \bal{A}{B}$ and  $g=f\cdot \Abs{\widecheck{\ind{H}}}^2$, so that, since by orthogonality $\ind{X}\ast \ind{\Gamma}=\ind{X+\Gamma}$,
\begin{align*}
\Inn{\widehat{g}, \ind{X}\circ \abs{\widehat{\mu_{B''}}}^2}
&=\Inn{ \abs{\widehat{\bal{A}{B}}}^2\ast \ind{H}\circ \ind{H}, \ind{X}\circ \abs{\widehat{\mu_{B''}}}^2}\\
&\geq \tfrac{1}{4}\eta^2\mu(B)^{-1}\Inn{\ind{H}\circ \ind{H}, \ind{X}\circ \ind{X+\Gamma}}\\
&\geq \tfrac{1}{4}K^{-1} \eta^2\mu(B)^{-1}\abs{X}\abs{H}^2.
 \end{align*}
By H\"{o}lder's inequality followed by the Cauchy--Schwarz inequality 
\begin{align*}
\Inn{\widehat{g}, \ind{X}\circ \abs{\widehat{\mu_{B''}}}^2}^{2m}
&=\Inn{g,\widecheck{\ind{X}}(\mu_{B''}\circ\mu_{B''})}^{2m}\\
&\leq \langle \abs{g},\mu_{B''}\circ \mu_{B''}\rangle^{2m-2}\langle \abs{g},(\mu_{B''}\circ\mu_{B''})\Abs{\widecheck{\ind{X}}}^m\rangle^2\\
&\leq \langle \abs{g},\mu_{B''}\circ \mu_{B''}\rangle^{2m-2}\langle \abs{g}^2,\mu_{B''}\circ \mu_{B''}\rangle E_{2m}(X;\abs{\widehat{\mu_{B''}}}^2)
.\end{align*}
It follows that
\[\langle \abs{g},\mu_{B''}\circ \mu_{B''}\rangle^{2m-2}\langle \abs{g}^2,\mu_{B''}\circ \mu_{B''}\rangle \geq (\tfrac{1}{4}K^{-1} \kappa^{-1}\eta^2\mu(B)^{-1}\abs{H}^2)^{2m}.\]
Since
\begin{align*}
\langle \abs{g}^2,\mu_{B''}\circ \mu_{B''}\rangle
& = \Inn{\abs{\widehat{\bal{A}{B}}}^2\circ \abs{\widehat{\bal{A}{B}}}^2\ast \ind{H}\ast \ind{H}\circ\ind{H}\circ \ind{H},\abs{\widehat{\mu_{B''}}}^2}\\
&\leq 2\Norm{\widehat{\bal{A}{B}}^2}_1^2\abs{H}^4\\
&\leq 2\alpha^{-2}\mu(B)^{-2}\abs{H}^4,
\end{align*}
it follows that, taking $(2m-2)$th roots,
\[
\langle \abs{f}\Abs{\widecheck{\ind{H}}}^2,\mu_{B''}\circ \mu_{B''}\rangle \gg \brac{K^{-1}\kappa^{-1}\eta^2\alpha}^{\frac{1}{m-1}} K^{-1}\kappa^{-1}\eta^2\mu(B)^{-1}\abs{H}^2\]
as required.
\end{proof}

Our task is now to convert the weak correlation of Lemma~\ref{lemma:specboost1} into a strong density increment. As in the model setting, we will require an $L^\infty$ almost-periodicity result for three convolutions, which was proved in \cite{ScSi:2016}.

\begin{lemma}[\cite{ScSi:2016}, Theorem 5.4]\label{lemma:Linfty-ap_Bohr}
Let $\epsilon \in (0,1/2)$. Let $S, M, L\subset G$, and let $B \subset G$ be a regular Bohr set of rank $d$. Suppose that $\nu\in (0,1)$ is such that $\abs{S+B} \leq \nu^{-1}\abs{S}$ and $\abs{M}\geq \nu \abs{L}$. Then there is a regular Bohr set $B' \subset B$ of rank at most $d + d'$ and size
\[\abs{B'}\geq (\epsilon \nu/dd')^{O(d+d')}\abs{B},\]
where
 \[d'\leq d+\tilde{O}_{\epsilon\nu}(\epsilon^{-2}), \]
such that
\[ \norm{ \mu_S \ast \mu_M \ast 1_L \ast \mu_{B'} - \mu_S \ast \mu_M \ast 1_L }_\infty \leq \epsilon. \]
\end{lemma}

The following lemma is the generalised form of Lemma~\ref{lemma:modelspecboost2}. 

\begin{lemma}\label{lemma:specboost2}

There exists some constant $c>0$ such that the following holds. Let $B$ and $B'$ be regular Bohr sets, both of rank at most $d$.  Suppose that $A\subset B$ has density $\alpha\in(0,1/2)$, and that $\rho,\delta\in (0,1/2)$ are some parameters such that $B'\subset B_\rho$ where $\rho \leq c\delta \alpha^2/d$. Suppose that $D\geq 1$ and $H\subset \widehat{G}$ are such that
\begin{enumerate}
\item
\[\Inn{\ind{H}\circ \ind{H}, \abs{\widehat{\mu_{B'}}}^2}\leq 2\abs{H},\]
\item $H$ is $D$-covered by $\Delta_{1/2}(B')$,
\item
\[\Inn{\Abs{\bal{A}{B}\circ \bal{A}{B}}, \Abs{\widecheck{\ind{H}}}^2(\mu_{B'}\circ \mu_{B'})}\geq \delta\abs{H}\mu(B)^{-1},\]
\listintertext{and}
\item $m\geq c^{-1}\log(\abs{H}/\delta\alpha)$ and $K\geq 1$ are such that 
\[\norm{\bal{A}{B}\circ\bal{A}{B}}_{2m(\mu_{B'}\circ \mu_{B'})}\leq K\mu(B)^{-1}.\]
\end{enumerate}
There is a regular Bohr set $B''\subset B'$ of rank at most $\rk(B')+d'$ and size
\[\abs{B''}\geq (\alpha\delta/K\abs{H}dd')^{O(d+d')}\abs{B'},\]
where
\[d'\leq D+\tilde{O} _{\alpha\delta/K}(\delta^{-4}K^{2}),\]
such that
\[\norm{\ind{A}\ast \mu_{B''}}_\infty \geq (1+2^{-9}\delta)\alpha.\]
\end{lemma}
\begin{proof}
Let $f=\bal{A}{B}\circ\bal{A}{B}$ for brevity. Since $\Inn{ f, \Abs{\widecheck{\ind{H}}}^2(\mu_{B'}\circ\mu_{B'})}= \Inn{ \abs{\widehat{\bal{A}{B}}}^2, \ind{H}\circ \ind{H}\circ\abs{\widehat{\mu_{B'}}}^2}\geq 0$, using the fact that $\max(x,0)=(x+\abs{x})/2$, we have 
\[\Inn{\max(f,0)\Abs{\widecheck{\ind{H}}}^2,\mu_{B'}\circ \mu_{B'}}\geq \tfrac{1}{2}\delta\abs{H}\mu(B)^{-1}.\]
Since 
\begin{equation}\label{eq:boost1}
\Norm{\Abs{\widecheck{\ind{H}}}^2(\mu_{B'}\circ\mu_{B'})}_1=\Inn{\ind{H}\circ \ind{H}, \abs{\widehat{\mu_{B'}}}^2}\leq 2\abs{H}
\end{equation}
we can further restrict this inner product to the set 
\[T = \{ x: f(x) \geq \tfrac{1}{8}\delta\mu(B)^{-1}\}\]
(and in particular $f(x)\geq 0$ for all $x\in T$), so that
\[\Inn{\ind{T}\abs{f}\Abs{\widecheck{\ind{H}}}^2,\mu_{B'}\circ \mu_{B'}}\geq \tfrac{1}{4}\delta\abs{H}\mu(B)^{-1}.\]
We now use H\"{o}lder's inequality to bound the left-hand side above by
\[\Inn{\abs{f}^{2m},\mu_{B'}\circ \mu_{B'}}^{1/2m}\Inn{\ind{T},\Abs{\widecheck{\ind{H}}}^{2+\frac{2}{2m-1}}(\mu_{B'}\circ \mu_{B'})}^{1-1/2m}.\]
The first factor is, by assumption, at most $K\mu(B)^{-1}$. The second is at most 
\[\abs{H}^{\frac{1}{m}}\Inn{\ind{T},\Abs{\widecheck{\ind{H}}}^{2}(\mu_{B'}\circ \mu_{B'})}^{1-1/2m}.\]
Since $\norm{f}_\infty \leq 4\alpha^{-1}\mu(B)^{-1}$, we have the crude bound $\Inn{\ind{T},\Abs{\widecheck{\ind{H}}}^{2}(\mu_{B'}\circ \mu_{B'})}\geq 2^{-4}\alpha\delta\abs{H}\geq 2^{-4}\alpha\delta$. Collecting these bounds together, 
\[K\mu(B)^{-1}(2^4\alpha^{-1}\delta^{-1}\abs{H})^{\frac{1}{m}}\Inn{\ind{T},\Abs{\widecheck{\ind{H}}}^{2}(\mu_{B'}\circ \mu_{B'})}\geq \tfrac{1}{4}\delta\abs{H}\mu(B)^{-1}.\]
Provided $m$ is sufficiently large, therefore, we have
\[\Inn{ \ind{T}, \Abs{\widecheck{\ind{H}}}^2(\mu_{B'}\circ\mu_{B'})}\geq \frac{\delta}{8K}\abs{H}.\]

We will now construct $\tilde{B}$, which will be a regular Bohr set such that $\Abs{\widecheck{\ind{H}}}^2(\mu_{B'}\circ\mu_{B'})$ is approximately invariant under shifts by $\tilde{B}$, as follows. By hypothesis there is a set  $\Lambda_0$ of size $\abs{\Lambda_0}\leq D$ such that every $\gamma\in H$ can be written as the sum or difference of at most $D$ elements from $\Lambda_0$ and 2 elements from $\Delta_{1/2}(B')$. Suppose that $B'$ is a Bohr set with width function $\nu$ and frequency set $\Gamma_0$. We define the Bohr set $\tilde{B}$ as the Bohr set with frequency set $\Gamma_0\cup \Lambda_0$ and width function
\[\nu'(\gamma) = \begin{cases} \frac{\epsilon}{d} \nu(\gamma)&\textrm{if }\gamma\in \Gamma_0\textrm{ and }\\
\frac{\epsilon}{D}&\textrm{if }\gamma\in \Lambda_0,\end{cases}\]
where we take the minimum of these widths if $\gamma$ lies in $\Gamma_0\cap \Lambda_0$, and $\epsilon\leq 1/2$ will be chosen later, but in particular chosen such that $\tilde{B}$ is regular. This new Bohr set has rank at most $\rk(B')+D$, and by Lemma~\ref{lemma:bohrsiz} satisfies
\[\Abs{\tilde{B}}\geq (\epsilon/dD)^{O(d+D)}\abs{B'}.\]
As in the proof of Lemma~\ref{lemma:L2inc}, our choices ensure that if $\lambda\in H$ and $t\in \tilde{B}$ then $\abs{1-\lambda(t)}\ll \epsilon$. In particular, for any $x\in G$ and $t\in \tilde{B}$, 
\[\widecheck{\ind{H}}(x+t)=\sum_{\lambda\in H}\lambda(x+t)=\widecheck{\ind{H}}(x)+O(\epsilon\abs{H}).\]
Therefore, for any fixed $t\in \tilde{B}$,
\[\bbe_x \abs{\Abs{\widecheck{\ind{H}}(x+t)}^2\mu_{B'}\circ \mu_{B'}(x+t)-\Abs{\widecheck{\ind{H}}(x)}^2\mu_{B'}\circ \mu_{B'}(x)}\]
is equal to 
\[\bbe_x \Abs{\widecheck{\ind{H}}(x)}^2\abs{\mu_{B'}\circ \mu_{B'}(x+t)-\mu_{B'}\circ \mu_{B'}(x)}+O(\epsilon \abs{H}^2).\]
Furthermore, by regularity of $B'$ (invoked in the form of Lemma~\ref{lemma:regConv}), and since $\tilde{B}\subset B'_{\epsilon/d}$, for any $t\in \tilde{B}$,
\[\bbe_x \abs{\mu_{B'}\circ \mu_{B'}(x+t)-\mu_{B'}\circ \mu_{B'}(x)}\ll \epsilon.\]
Combining these estimates and averaging over $t\in \tilde{B}$ yields 
\[\norm{\brac{\Abs{\widecheck{\ind{H}}}^2(\mu_{B'}\circ\mu_{B'})}\ast \mu_{\tilde{B}}-\Abs{\widecheck{\ind{H}}}^2(\mu_{B'}\circ\mu_{B'})}_1 \ll \epsilon \abs{H}^2.\]
It follows that
\begin{align*}
\Inn{ \ind{T}\ast\mu_{\tilde{B}},\Abs{\widecheck{\ind{H}}}^2(\mu_{B'}\circ \mu_{B'})}
&=\Inn{ \ind{T},\Abs{\widecheck{\ind{H}}}^2(\mu_{B'}\circ \mu_{B'})}+O\brac{\epsilon\abs{H}^2}\\
&\geq \frac{\delta}{16K}\abs{H},
\end{align*}
provided we choose $\epsilon = c'\delta/K\abs{H}$ for some sufficiently small constant $c'$. Using \eqref{eq:boost1} once again, and the fact that the support of $\mu_{B'}\circ \mu_{B'}$ is $B'+B'$, yields some $z\in B'+B'$ such that
\[\ind{T}\ast \mu_{\tilde{B}}(z)\geq \delta/32 K.\]
Recalling the definition of $T$, it follows that there exists some $T'\subset \tilde{B}$ of measure $\mu_{\tilde{B}}(T')\geq \delta/32K$ and $z\in B'+B'$ such that
\[\Inn{\bal{A}{B}\circ \bal{A}{B}, \ind{T'+z}}\geq \tfrac{1}{8}\delta \mu(B)^{-1}\mu(T').\]
We now expand the left-hand side, recalling the definition $\bal{A}{B}=\mu_A-\mu_B$. By regularity of $B$ (in particular Lemma~\ref{lemma:regConv}), and since $T'+z\subset \tilde{B}+ B'+B'\subset B_{3\rho}$,
\[\Inn{\mu_A\circ \mu_B, \ind{T'+z}}=\mu(T')\mu(B)^{-1}+O(\rho d \alpha^{-1}\mu(B)^{-1}\mu(T'))\]
and trivially
\[\Inn{\mu_B\circ \mu_B, \ind{T'+z}}\leq\mu(T')\mu(B)^{-1}.\]
Provided $\rho \leq c\delta \alpha /d$ for some sufficiently small constant $c>0$, we therefore have
\[\Inn{\mu_A\circ \mu_A, \ind{T'+z}}\geq \brac{1+\frac{\delta}{16}}\mu(T')\mu(B)^{-1}.\]
In particular, if we let $A'=A-z$, then
\[\Inn{\ind{A}\circ \ind{A'}, \ind{T'}}\geq \brac{1+\frac{\delta}{16}}\alpha^2\mu(T')\mu(B).\]
Let $\tilde{B}'=\tilde{B}_{\rho'}$ be a regular Bohr set, where $\rho'$ will be chosen later, but in particular chosen such that $\tilde{B}'$ is regular. In particular, $\tilde{B}'\subset B_{\rho}$ so that, by the regularity of $B$ and Lemma~\ref{lemma:regConv},
\[\norm{\ind{B}\ast \mu_{\tilde{B}'}-\ind{B}}_1 \ll \rho d\mu(B).\]
It follows that
\[\Inn{\ind{T'}\ast \ind{A'},\ind{A}} = \Inn{\ind{T'}\ast \ind{A'},\ind{A}(\ind{B}\ast \mu_{\tilde{B}'})}+O(\rho d\mu(B)\mu(T')).\]
Provided $\rho \leq c \delta \alpha^2/d$ for some sufficiently small $c>0$, therefore, 
\[\Inn{\ind{T'}\ast \ind{A'},\ind{A}(\ind{B}\ast \mu_{\tilde{B}'})}\geq \brac{1+\frac{\delta}{32}}\alpha^2\mu(T')\mu(B).\]
We now write
\[1_A(1_B\ast \mu_{\tilde{B}'})=\frac{1}{\Abs{\tilde{B}'}}\sum_{y\in B}1_{A\cap (\tilde{B}'+y)}\]
to deduce that 
\begin{equation}\label{eq:boost2}
\frac{1}{\Abs{\tilde{B}'}}\sum_{y\in B}\Inn{\ind{T'}\ast \ind{A'},\ind{A\cap (\tilde{B}'+y)}}\geq \brac{1+\frac{\delta}{32}}\alpha^2\mu(B)\mu(T').
\end{equation}
Let $Y\subset B$ be the set of those $y\in B$ such that $\Abs{A\cap (\tilde{B}'+y)}\geq 2^{-9}\delta\alpha^2 \Abs{\tilde{B}'}$. The contribution to the left-hand side of \eqref{eq:boost2} from those $y\in B\backslash Y$ is at most 
\[\frac{\abs{B}}{\Abs{\tilde{B}'}}2^{-9}\delta\alpha^2\mu(\tilde{B}')\mu(T')=2^{-9}\delta\alpha^2\mu(B)\mu(T'),\]
and hence
\begin{equation}\label{eq:this}
\frac{1}{\Abs{\tilde{B}'}}\sum_{y\in Y}\Inn{\ind{T'}\ast \ind{A'},\ind{A\cap (\tilde{B}'+y)}}\geq (1+2^{-6}\delta)\alpha^2\mu(B)\mu(T').
\end{equation}
Similarly,
\begin{align*}
\alpha
&=\Inn{\ind{A},\mu_B}\\
&=\Inn{\ind{A},\mu_B\ast \mu_{\tilde{B}'}}+O(\rho d)\\
&=\frac{1}{\abs{B}\Abs{\tilde{B}'}}\sum_{y\in B}\Abs{A\cap (\tilde{B}'+y)}+O(\rho d)\\
&=\frac{1}{\abs{B}\Abs{\tilde{B}'}}\sum_{y\in Y}\Abs{A\cap (\tilde{B}'+y)}+O(\rho d)+E
\end{align*}
where $E\leq 2^{-9}\delta\alpha^2$. Provided $\rho\leq c \delta \alpha/d$ for some sufficiently small constant $c>0$, we can ensure that the total error term here is at most $2^{-8}\delta\alpha$, say, and hence 
\begin{equation}\label{eq:that}
 (1+2^{-6}\delta)\alpha^2\mu(B)\mu(T')\geq (1+2^{-7}\delta)\alpha \mu(B)\mu(T')\frac{1}{\abs{B}\Abs{\tilde{B}'}}\sum_{y\in Y}\Abs{A\cap (\tilde{B}'+y)}.
 \end{equation}
Combining \eqref{eq:this} and \eqref{eq:that} and averaging over $y\in Y$ we find some $y\in Y$ such that, if we let $A''=A\cap (\tilde{B}'+y)$, then 
\[\Inn{\ind{T'}\ast \ind{A'},\ind{A''}}\geq(1+2^{-7}\delta)\alpha \mu(A'')\mu(T').\]
Since $T' \subset \tilde{B}$ and $A'' \subset \tilde{B}'+y$, we can replace $A'$ by $A'\cap(\tilde{B}+\tilde{B}'+y)$ without affecting the value of the inner product. We can go even further, by the regularity of $\tilde{B}$, and replace $A'$ by $A'''=A'\cap (\tilde{B}+y)$, with an error of at most
\[\mu((\tilde{B}+\tilde{B}')\backslash \tilde{B})\mu(A'')\ll \rho'(d+D)\mu(A'')\mu(\tilde{B}),\]
which is at most $2^{-8}\delta\alpha \mu(A'')\mu(T')$, provided $\rho'\leq c\delta\alpha \mu_{\tilde{B}}(T)/(d+D)$ for some sufficiently small constant $c>0$. In particular, 
\[\Inn{\ind{T'},\ind{A''}\circ \ind{A'''}}\geq(1+2^{-8}\delta)\alpha \mu(A'')\mu(T').\]
As an immediate consequence, $\abs{A'''} \geq (1+2^{-8}\delta)\alpha \abs{T'}$. We may further assume that $\abs{A'''}\leq 2\alpha \Abs{\tilde{B}}$, or else we are done, letting $B''=\tilde{B}$. 

We now apply Lemma~\ref{lemma:Linfty-ap_Bohr} relative to $\tilde{B}'_{c/(d+D)}$ for some small constant $c>0$, chosen in particular such that $\tilde{B}'_{c/(d+D)}$ is regular, with the choices
\[S=-A''\quad M= A'''\quad L = T'\textrm{ and }\epsilon = 2^{-10}\delta\mu_{\tilde{B}}(T').\]
Using regularity we have 
\[\Abs{-A''+\tilde{B}'_{c/(d+D)}}\leq \Abs{\tilde{B}'+\tilde{B}'_{c/(d+D)}}\leq 2\Abs{\tilde{B}'},\]
and in particular the parameter $\nu$ in Lemma~\ref{lemma:Linfty-ap_Bohr} satisfies
\[\nu\geq \min\brac{\frac{\abs{A''}}{2\Abs{\tilde{B}'}}, \frac{\abs{A'''}}{\abs{T'}}}\geq \min\brac{2^{-10}\delta\alpha^2,\alpha}=2^{-10}\delta\alpha^2,\]
where we recall that since $y\in Y$ we have $\abs{A''}\geq 2^{-9}\delta\alpha^2\Abs{\tilde{B}'}$. Furthermore, recall that $\mu_{\tilde{B}}(T')\geq \delta/32K$, and so $\epsilon \geq 2^{-15}\delta^2/K$. We therefore produce a new regular Bohr set $B''\subset \tilde{B}'_{c/(d+D)}$ of rank at most $\rk(\tilde{B})+d''$ and size at least
\begin{align*}
\abs{B''}
&\geq (\epsilon \alpha\delta/d''\rk(\tilde{B}))^{O(\rk(\tilde{B})+d'')}\Abs{\tilde{B}'_{c/(d+D)}}\\
&\geq (\alpha\delta/KdD)^{O(d+D+d'')}\Abs{\tilde{B}},
\end{align*}
where
\[d'' \ls_{\alpha\delta/K} \delta^{-4}K^2,\]
such that
\begin{align*}
\norm{\ind{T'}\circ \ind{A''}\ast \ind{A'''}\ast \mu_{B''}-\ind{T'}\circ \ind{A''}\ast \ind{A'''}}_\infty 
&\leq \epsilon\mu(A'')\mu(A''')\\
&= 2^{-10}\delta\mu_{\tilde{B}}(T')\mu(A'')\mu(A''')\\
&\leq 2^{-9}\delta\alpha \mu(A'')\mu(T'),
\end{align*}
and so
\[\Inn{\ind{T'},\ind{A''}\circ \ind{A'''}\ast \mu_{B''}} \geq (1+2^{-9}\delta)\alpha\mu(A'')\mu(T').\]
The claim follows after bounding the left-hand side above by 
\[\norm{\ind{A''}\circ \ind{T'}}_1\norm{\ind{A'''}\ast \mu_{B''}}_\infty=\mu(A'')\mu(T')\norm{\ind{A'''}\ast \mu_{B''}}_\infty.\]
\end{proof}

We conclude this section by combining these lemmas and the structural result on non-smoothing sets into the following proposition, which extracts a suitable density increment from a large set with both spectral and non-smoothing properties. The hypotheses are designed to dovetail with the conclusions of Proposition~\ref{prop:big}.

When parsing the following statement, it may help to bear in mind that we will be applying this with $K=\alpha^{-O(1/k)}$.
\begin{proposition}\label{prop:specboost}
There is a constant $C>0$ such that the following holds. Let $B$, $B'$, and $B''$ be regular Bohr sets, all of rank $d$. Suppose that $A\subset B$ has density $\alpha$ and $B'\subset B_\rho$ and $B''\subset B'_{\rho'}$ with $\rho,\rho'\leq \alpha^C/d$. 

Let $\Gamma=\Delta_{1/2}(B')$ and $K,k\geq 2$ be some parameters. Let $\tau$ be another parameter satisfying $K\alpha^2\geq \tau \geq K^{-1}\alpha^2$. Suppose that $\widetilde{\Gamma}$ is an additive framework of height $h$ and tolerance $t$ with $\Gamma_{\mathrm{top}}=\Gamma$, where $h\leq C^{-1}\log\log k/\log\log\log k$ and $t\geq C\log k$.

Suppose that 
 \begin{enumerate}
\item $\Delta$ is a set of size $K\alpha^{-3}\geq \abs{\Delta}\geq K^{-1}\alpha^{-3}$ which is $\tfrac{1}{4}$-robustly $(\tau,k)$-additively non-smoothing relative to $\widetilde{\Gamma}$,
\item if $\gamma\in \Delta+\Gamma$ then 
\[\Abs{\widehat{\bal{A}{B}}}^2\circ \Abs{\widehat{\mu_{B''}}}^2(\gamma)\geq K^{-1}\alpha^2 \mu(B)^{-1},\]
and
\item \[ \Norm{\ind{\Delta}\circ \abs{\widehat{\mu_{B''}}}^2}_\infty \leq 2,\]
\end{enumerate}
If all of the above holds, then there exists some 
\[M\leq (m2^kt^hK\alpha^{-\frac{1}{\log\log k}-\frac{1}{h}})^{O(1)}\]
such that either
\begin{enumerate}
\item $\alpha \geq M^{-1}$ or
\item $A$ has a density increment of strength $[\alpha^{-1/k},\alpha^{1+1/k};\tilde{O}_{\alpha/m}(1)]$ relative to $B''$, or
\item $A$ has a density increment of strength $[M^{-1},M; \tilde{O}_\alpha(\log M)]$ relative to $B''$.
\end{enumerate}
\end{proposition}

To orient the reader, we summarise the important dependencies for the proof in Figure~\ref{figb}.

\begin{figure}[h]
\centering
\includegraphics[width=5in]{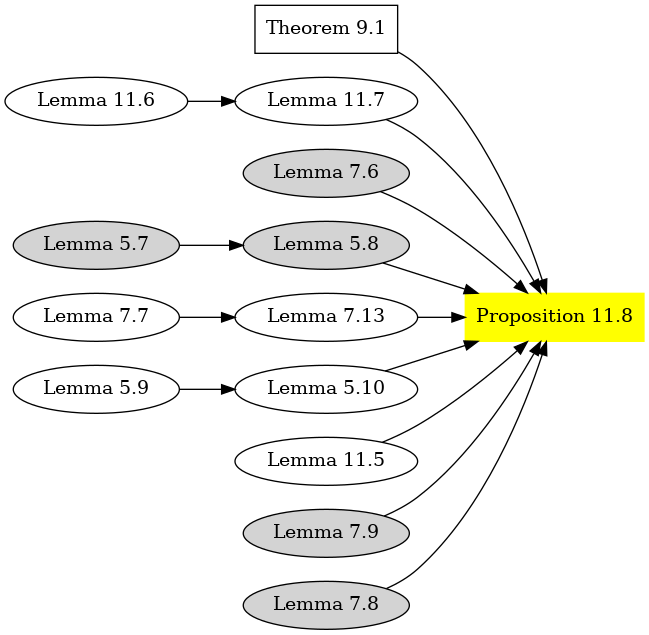}
\caption{Dependency chart for the proof of Proposition~\ref{prop:specboost}. Dependencies on lemmas from Section 4 are not shown. Those lemmas in grey are also used in the proof of Proposition~\ref{prop:big}.}
\label{figb}
\end{figure}
\begin{proof}[Proof of Proposition~\ref{prop:specboost}]
There are no new ideas in the proof of this proposition -- it is largely a matter of chaining together the tools we have assembled so far, and verifying that the technical hypotheses of each are satisfied. 

We first give a brief sketch. We apply the structural result Theorem~\ref{th:structure} to deduce that in the additively non-smoothing set $\Delta$ we can find a pair of subsets $X$ and $H$ with appropriate energy properties. We would like to apply the spectral boosting Lemma~\ref{lemma:specboost1}, but to do so we need an upper bound on the $2m$-fold relative additive energy of $X$. We therefore first argue that if such a bound is violated then we can find a large subset of $X$ (and hence in particular of $\Delta$) where we have smaller than expected dimension. We can then remove this piece from $\Delta$, and repeat the argument from the beginning. Eventually, either we find some $X$ which allows us to apply Lemma~\ref{lemma:specboost1}, or else we have found a large subset of $\Delta$ with smaller than expected dimension. In the latter case we have found a suitable density increment, and in the former we have found some correlation of $\abs{\bal{A}{B}\circ \bal{A}{B}}$ with $\Abs{\widecheck{\ind{H}}}^2$. We then apply Lemma~\ref{lemma:specboost2} to deduce a genuine density increment of the required strength, which is applicable unless the (relative) $L^{2m}$ norm of $\bal{A}{B}\circ \bal{A}{B}$ is too large, for some large $m$. In this case, however, we can deduce a density increment using almost-periodicity directly, as shown in Lemma~\ref{lemma:physdi}.

Let us begin. We first note that we can assume that $K\leq \alpha^{-1}$ and $\alpha\leq 1/4$, say, or else we are trivially in the first case. In particular, this ensures that $\tau\leq 1/2$. For the rest of this proof, all implicit constants are polynomial in $m\log(2/\alpha)t^h2^kK\alpha^{-\frac{1}{\log\log k}-\frac{1}{h}}$. We apply Theorem~\ref{th:structure} to find some $X,H\subset \Delta$ such that, for some $\delta\gg\alpha^2$ and both $\abs{H}\asymp \delta\abs{\Delta}$ and $\abs{X}\asymp\tau \delta^{-1}\abs{\Delta}$, and furthermore
\[\Inn{ \ind{X}\circ \ind{X}, \ind{H}\circ \ind{H}\circ \ind{\Gamma}}\gg \abs{H}^2\abs{X}.\]
We also have that, for some $z$,
\[H+z \subset \{ x: \ind{X}\circ \ind{X}\circ \ind{\Gamma}(x)\gg \abs{X}\}.\]
We further note that 
\[\norm{\ind{X}\circ \abs{\widehat{\mu_{B''}}}^2}_\infty \leq \norm{\ind{\Delta}\circ \abs{\widehat{\mu_{B''}}}^2}_\infty\leq 2,\]
and so the hypotheses of Lemma~\ref{lemma:dimsymmetry} are met.  We hence deduce by Lemma~\ref{lemma:dimsymmetry} that $H+z$ has $\Gamma'$-dimension $O(1)$, where $\Gamma'=\Delta_{1/2}(B'')$, provided $\rho'$ is sufficiently small (which our assumption that $\rho'\leq \alpha^C/d$ guarantees, or else we are in the first case). It follows by Lemma~\ref{lemma:dimcovering} that $H$ itself is $O(1)$-covered by $\Gamma'$.

If $\abs{X}\leq \tfrac{1}{2}\abs{\Delta}$, then we apply the structural theorem again to $\Delta\backslash X$ -- here we use the observation that if $\Delta$ is $\tfrac{1}{4}$-robustly $(\tau,k)$-additively non-smoothing and $\Delta'\subset \Delta$ has size $\abs{\Delta'}\geq \tfrac{1}{2}\abs{\Delta}$ then $\Delta'$ is $\tfrac{1}{2}$-robustly $(\tau,k)$-additively non-smoothing. 

We then repeat this process, obtaining a disjoint collection $X_1,\ldots,X_n$ of subsets of $\Delta$ with the above properties (with associated, and possibly different, $H_i$ and $\delta_i$) until $\abs{X_1\cup \cdots \cup X_n}\geq \tfrac{1}{2}\abs{\Delta}$. By dyadic pigeonholing there is some $I\subset \{1,\ldots,n\}$ and $\delta\gg \alpha^2$ such that for $i\in I$ we have $2\delta>\delta_i\geq \delta$ and
\[\sum_{i\in I}\abs{X_i}\asymp \abs{\Delta}.\]
Since $\abs{X_i}\asymp \tau \delta^{-1}\abs{\Delta}$ for all $i\in I$, we have in particular that $\abs{I}\asymp \tau^{-1}\delta$.

Let $\kappa\in(0,1)$ be chosen later. We now split into two cases, according to the additive energies of the $X_i$.
\\

\noindent\textbf{Case One:} Suppose first that, for all $i\in I$,
\[E_{2m}(X_i; \abs{\widehat{\mu_{B''}}}^2) > (\kappa\abs{X_i})^{2m}.\]
Let $L$ be some parameter to be chosen later, and let $B'''=B''_{\rho''}$ for some $\rho'' = c/Ld$ where $c>0$ is some sufficiently small absolute constant. By Lemma~\ref{lemma:fourierbohr} we have $\mu_{B''}\leq 2\mu$, with $\mu=\mu_{B''_{1+L\rho''}}\ast \mu_{B'''}^{(L)}$. In particular, provided we choose $L=C\lceil m \log(2/\alpha)\rceil$ for some sufficently large absolute constant $C>0$, we have
\[E_{2m}(X_i; \Delta_{1/2}(B''')) \geq (\tfrac{1}{2}\kappa\abs{X_i})^{2m},\]
say. We now apply Lemma~\ref{lemma:energytodimension} with $\omega=\ind{X_i}$ and $\ell=Cm\lceil \kappa^{-1}\rceil$ for some large constant $C$. We note that we can assume that $\norm{\omega}_1=\abs{X_i} \geq 2\ell$, or else $\alpha\delta \gg \kappa$ and (as we will choose $\kappa\gg \delta^{1/2}$) we are in the first case. 

Provided we choose the constant $C$ in the choice of $\ell$ sufficiently large, the second case of Lemma~\ref{lemma:energytodimension} cannot hold, and hence there exists $X_i'\subset X_i$ such that 
\[\abs{X_i'} \gg \kappa\abs{X_i}\gg \alpha^2\delta^{-1}\kappa\abs{\Delta},\]
and $\dim(X_i';\Gamma'')\ll \kappa^{-1}$, where $\Gamma''=\Delta_{1/2}(B''')$.

We now take the union of $\lfloor \abs{I}^{1/2}\rfloor$ many copies of $X_i'$, to form a new set, say $X'$. Observe that $\abs{I}^{1/2}\asymp\alpha^{-1}\delta^{1/2}$. By the disjointedness of the $X_i$,
\[\abs{X'}\gg \alpha\delta^{-1/2}\kappa \abs{\Delta}\gg \alpha^{-2}\delta^{-1/2}\kappa.\]
Furthermore, by Lemma~\ref{lemma:dimunion}, $\dim(X';\Gamma'')\ll \alpha^{-1}\delta^{1/2}\kappa^{-1}$. 
In particular, we can choose some $\kappa\asymp \delta^{1/2}$ such that $\abs{X'}\geq K\alpha^{-2-1/k}$ and $\dim(X';\Gamma'')\leq \alpha^{-1+1/k}$. Lemma~\ref{lemma:lowdiminc}, applied with $\Delta$ replaced by $X'$, implies that there exists some $C_0$ which is polynomially bounded (by absolute constants) by $\log(2/\alpha)$ such that $A$ has a density increment of strength $[\alpha^{-1/k},\alpha^{-1+1/k};C_0]$ relative to $B'''$. Since $B'''$ has the same rank as $B''$ and 
\[\abs{B'''}\geq (m\log(2/\alpha)d)^{-Cd}\abs{B''}\]
for some absolute constant $C>0$, it follows in particular that $A$ has a density increment of strength $[\alpha^{-1/k},\alpha^{-1+1/k};C_0']$ relative to $B''$, where $C_0'$ is polynomially bounded (by absolute constants) by $\log(m/\alpha)$, and we are in case (2).
\\

\noindent\textbf{Case Two:} Suppose that there exists some $i\in I$ such that
\[E_{2m}(X_i; \abs{\widehat{\mu_{B''}}}^2) \leq (\kappa\abs{X_i})^{2m},\]
with $\kappa \asymp \delta^{1/2}$ as specified in the previous case. 
We will apply spectral boosting to such an $X_i$ (and its associated $H_i$), and henceforth omit the subscripts. All of the conditions of Lemma~\ref{lemma:specboost1} are now met, with the parameter $K$ (in the language of that lemma) being $O(1)$ and the $\eta$ parameter being $\gg \alpha$. Therefore Lemma~\ref{lemma:specboost1} implies that, assuming $m$ is sufficiently large, 
\[\Inn{\abs{\bal{A}{B}\circ \bal{A}{B}}, \Abs{\widecheck{\ind{H}}}^2(\mu_{B''}\circ \mu_{B''})}\gg \kappa^{-1}\alpha^2\abs{H}^2\mu(B)^{-1}.\]
Recalling that $\abs{H}\asymp \delta\abs{\Delta}\asymp \delta\alpha^{-3}$ and $\kappa\asymp \delta^{1/2}$, and further that $\delta\gg \alpha^2$, this implies that  
\[\Inn{\abs{\bal{A}{B}\circ \bal{A}{B}}, \Abs{\widecheck{\ind{H}}}^2(\mu_{B''}\circ \mu_{B''})}\gg \abs{H}\mu(B)^{-1}.\]
Finally, we may assume that 
\[\norm{\bal{A}{B}\circ\bal{A}{B}}_{2m(\mu_{B''}\circ \mu_{B''})}\leq K\mu(B)^{-1},\]
for some $m=C\lceil \log(2/\alpha)\rceil$, with $C$ some large constant. Indeed, if this fails, then by Lemma~\ref{lemma:physdi} $A$ has a density increment of strength $[K^{-1},K,\tilde{O}_\alpha(1)]$ relative to $B''$, and we are in the third case.

The conditions of Lemma~\ref{lemma:specboost2} are now all satisfied, and hence we find some regular Bohr set $B'''\subset B''$ of rank at most $d+O(1)$ and size
\[\abs{B'''}\geq (c\alpha/d)^{C(d+O(1))}\abs{B''},\]
for some $c\gg 1$ and absolute constant $C>0$, such that
\[\norm{\ind{A}\ast \mu_{B'''}}_\infty \geq (1+\Omega(1))\alpha.\]
It follows that we can choose some $M$ bounded as described in the statement such that $A$ has a density increment of strength $[M^{-1},M;\tilde{O}_\alpha(\log M)]$ relative to $B''$, as required.
\end{proof}
\section{Concluding the proof}\label{section:concluding}
Finally, we will combine Propositions~\ref{prop:big} and \ref{prop:specboost} to prove Proposition~\ref{mainprop}, and so finally concluding the proof of Theorem~\ref{mainthm}. To orient the reader, we have provide another dependency chart summarising what goes into the proof of Theorem~\ref{mainthm} in Figure~\ref{figc}.

\begin{figure}[h]
\centering
\includegraphics[width=3in]{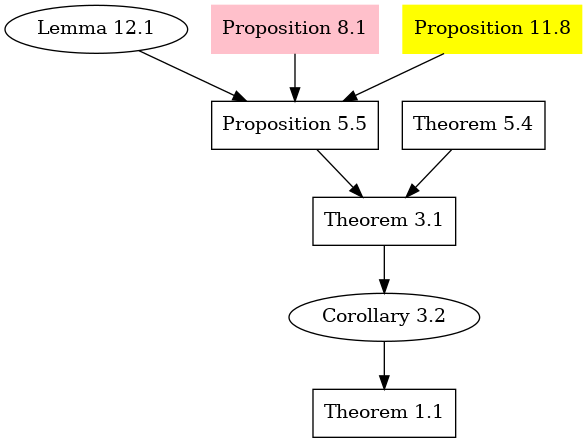}
\caption{Dependency chart for the proof of Theorem~\ref{mainthm}}
\label{figc}
\end{figure}

Before we prove Proposition~\ref{mainprop}, we require one final technical result, owing to the need (in Proposition~\ref{prop:big}) to work simultaneously with both $A\subset B$ and $A'=A\cap B'$, where $B'$ is some suitably narrowed dilated of $B$. 

The following lemma stems from Bourgain's work \cite{Bo:1999}; it will enable us to assume that (a translate of) $A$ is dense both in a Bohr set $B$ and some narrower copy $B_\delta$ simultaneously -- or else we have a density increment (with no loss of rank).

\begin{lemma}\label{lemma:TwoScales}
There is a constant $c>0$ such that the following holds. Let $B$ be a regular Bohr set of rank $d$, suppose $A \subset B$ has density $\alpha$, let $\epsilon > 0$ and suppose $B',B''\subset B_\rho$ where $\rho\leq c \alpha \epsilon/d$. Then either
\begin{enumerate}
   \item ($A$ has almost full density on both $B'$ and $B''$) there is an $x \in B$ such that $\ind{A}*\mu_{B'}(x) \geq (1-\epsilon)\alpha$ and $\ind{A}*\mu_{B''}(x) \geq (1-\epsilon)\alpha$, or
   \item (density increment) $A$ has a density increment of strength $[\epsilon,0; O(1)]$ relative to either $B'$ or $B''$.
\end{enumerate}
\end{lemma}
\begin{proof}
Provided $\rho$ is chosen sufficiently small, Lemma \ref{lemma:regConv} yields
\[ \abs{\Inn{\ind{A}\ast \mu_{B'},\mu_B}-\Inn{\ind{A},\mu_B}} \leq \norm{\mu_B*\mu_{B'} - \mu_B}_1 \leq \tfrac{1}{4} \epsilon\alpha, \]
and similarly for $B''$. Since $\Inn{\ind{A},\mu_B}= \alpha$, this implies that
\[ \mathbb{E}_{x \in B}\brac{\ind{A}*\mu_{B'}(x) + \ind{A}*\mu_{B''}(x)}\geq (2 - \tfrac{1}{2}\epsilon) \alpha, \]
and so there exists $x \in B$ such that $\ind{A}*\mu_{B'}(x) + \ind{A}*\mu_{B''}(x) \geq (2 - \tfrac{1}{2}\epsilon) \alpha$. With such an $x$, if we do not have a density increment of strength $[\epsilon, 0; 2]$ on either $B'$ or $B''$, then
\[ \ind{A}*\mu_{B'}(x) \geq (2- \tfrac{1}{2}\epsilon)\alpha - (1+\tfrac{1}{2}\epsilon)\alpha = (1-\epsilon)\alpha, \]
and similarly for $B''$, and so we are done.
\end{proof}

We now have everything we need to prove the main proposition, Proposition~\ref{mainprop}, which we restate here. 
The proof is a combination of the technical tools we have established.

\begin{customprop}{\ref{mainprop}}
There is a constant $C > 0$ such that, for all $k\geq C$, the following holds. Let $B$ be a regular Bohr set of rank $d$ and suppose that $A\subset B$ has density $\alpha$. Either 
\begin{enumerate}
\item $\alpha \geq 2^{-O(k^2)}$,
\item 
\[T(A)\gg \exp(-\tilde{O}_{\alpha}(d\log 2d))\mu(B)^2,\]
or
\item $A$ has a density increment of one of the following strengths relative to $B$:
\begin{enumerate}
\item (small increment) $[\alpha^{O(\frac{\log\log \log k}{\log\log k})}, \alpha^{-O(\frac{\log\log \log k}{\log\log k})}; \tilde{O}_\alpha(1)]$, or 
\item (large increment)  $[\alpha^{-1/k}, \alpha^{-1+1/k};\tilde{O}_\alpha(1)]$.
\end{enumerate}
\end{enumerate}
\end{customprop}
\begin{proof}[Proof of Proposition~\ref{mainprop}]
We begin by applying Lemma~\ref{lemma:TwoScales} with two Bohr sets $B'=B_\rho$ and $B''=B'_{\rho'}$, where $\rho=c\alpha \epsilon/d$, and $\rho'=c'\alpha^2/d$ where $c$ and $c'$ are small constants, chosen in particular so that both $B'$ and $B''$ are regular, and $\epsilon = c_0 \alpha^{C_0\frac{\log\log\log k}{\log\log k}}$ for some small constant $0 < c_0 \leq 1/3$ and large constant $C_0 > 0$. If the density increment holds, then we have a small increment as required.  Otherwise, the set $A-x$ has density $\alpha'$ and $\alpha''$ respectively in $B'$ and $B''$, where both $\alpha'$ and $\alpha''$ lie in $[(1-\epsilon)\alpha,(1+\epsilon)\alpha]$. (Note that this ensures that $\alpha'/2 \leq \alpha'' \leq 2\alpha'$ as required in Proposition \ref{prop:big}.) In the rest of the argument, where we apply the results from other parts of the paper, $A' = (A-x)\cap B'$ will play the role of $A$, and $A'' = (A-x)\cap B''$ the role of $A'$. Observe that both $A'$ and $A''$ are subsets of the same translate of $A$, and so a lower bound for $T(A',A'',A')$ will give a lower bound for $T(A)$ as required. Our choice of $\epsilon$ will ensure that any density increment of strength $[\delta, d'; C]$ for $A'$ encountered in the argument will give a density increment of strength $[\delta, d'; 2C]$ for $A$, and we therefore do not distinguish between these.

We now apply Proposition~\ref{prop:big}, with $h=\lceil c_1\log\log k/\log\log\log k\rceil$, and $t=\lceil C_2\log k \rceil$, for some suitable constants $c_1,C_2>0$. This means that either $\alpha \gg 1/k^2$, or the number of progressions is large, or we have a small increment, or we have a large increment, or we have a large orthogonal subset of the spectrum which is additively non-smoothing. Note that since $h \log t = \tilde{O}_\alpha(1)$, or else we have the first case of the conclusion, the constants in the increments are all $\tilde{O}_\alpha(1)$. We will not repeat the further technical parts of the conclusions here, but note that they have been constructed so that the hypotheses of Proposition~\ref{prop:specboost} are met with some $K=\alpha^{-O(1/k)}$, $\tau\gg \alpha^{2+O(1/k)}$, and $k$ replaced by $k'$ such that $k\geq k'\gg k$ (as can be seen comparing their statements).

It remains to check that the conclusions of Proposition~\ref{prop:specboost} imply the result. The number $M$ in that conclusion satisfies $M = \alpha^{-O(\frac{\log\log\log k}{\log\log k})}$ by our choices of parameters, or else $\alpha$ is large as in case (1) of our conclusion. Taking $C$ large enough in the lower bound for $k$, case (1) of Proposition~\ref{prop:specboost} cannot hold. In case (2), we get a large increment, and in case (3) we get a small increment. These increments are all for $A'$ but, provided the constants in the choice of $\epsilon$ were chosen small enough, they yield increments of the same strength for $A$, and the proof is complete.
\end{proof}

\section{Concluding remarks and conjectures}\label{section:spec}

We conclude the paper by engaging in some speculation about the correct bounds for Roth's theorem on arithmetic progressions. 

\subsection*{The correct bounds}

For brevity, let $r(N)$ denote the maximal density of a subset of $\{1,\ldots,N\}$ that contains no non-trivial three-term arithmetic progressions, so that Theorem~\ref{mainthm} states that
\begin{equation}\label{eq:thisbound}
r(N) \ll \frac{1}{(\log N)^{1+c}}
\end{equation}
for some constant $c>0$. It is extremely unlikely that this is the optimal upper bound for $r(N)$. For contrast, an elegant construction of Behrend \cite{Be:1946} implies that, for all sufficiently large $N$,
\begin{equation}\label{eq:behrend}
r(N) \gg \exp(-O((\log N)^{1/2})).
\end{equation}
Although the construction is simple, and is almost 75 years old, this lower bound has not been significantly increased. A slight improvement on Behrend's construction was found by Elkin \cite{El:2011}, with an alternative approach by Green and Wolf \cite{GrWo:2010}, but this does not change the form of the lower bound in \eqref{eq:behrend}.

We believe that the lower bound in \eqref{eq:behrend} is much closer to the truth than the upper bound in \eqref{eq:thisbound}.

\begin{conjecture}\label{conj}
There exists some absolute constants $c,c'>0$ such that, for all sufficiently large $N$,
\begin{equation}\label{eq:altbehrend}
r(N) \ll \exp(-c'(\log N)^c)).
\end{equation}
\end{conjecture}

This is a folklore conjecture that has circulated for some time, but to our knowledge has not appeared explicitly in the literature before. Aside from the evidence that the lower bound \eqref{eq:behrend} has resisted improvement for many decades, we note that \eqref{eq:altbehrend} has been established for variants where we replace three-term arithmetic progressions, solutions to $x+y=2z$, with solutions to similarly translation invariant linear equations in more variables. For example, if instead of $r(N)$ we consider $r'(N)$, the maximal density of a subset of $\{1,\ldots,N\}$ that contains no non-trivial solutions to $x+y+z=3w$, then Schoen and the second author \cite{ScSi:2016} have proved that there exists some $c>0$ such that 
\[r'(N) \ll \exp(-c(\log N)^{1/7}).\]
Furthermore, as mentioned in the introduction, the new polynomial method by Croot, Lev, and Pach \cite{CrLePa:2016} has allowed Ellenberg and Gijswijt \cite{ElGi:2016} to prove bounds for the quantity analogous to $r(N)$ over $\bbf_3^n$ corresponding to \eqref{eq:altbehrend} with $c=1$. 

The strongest possible form of Conjecture~\ref{conj}, which may well be true, is that one can take $c=1/2$, which, by Behrend's lower bound, would be best possible. While we are confident that this conjecture holds for some $c>0$, whether $c=1/2$ is permissible is much more uncertain.

\subsection*{A path to better bounds} 

Although the proof in this paper delivers an upper bound far short of the bound in Conjecture~\ref{conj}, the density increment method used is, in principle, capable of delivering such bounds. For example, we believe that it is true that, if $B$ is a regular Bohr set of rank $d$ and $A\subset B$ has density $\alpha$, then either
\begin{enumerate}
\item $T(A) \gg \exp(-\tilde{O}_\alpha(d\log 2d))\mu(B)^2$ or
\item $A$ has a density increment of strength $[1,1; \tilde{O}_\alpha(1)]$ relative to $B$.
\end{enumerate}
If this were established then a straightforward iteration (similar to the proof of Theorem~\ref{mainthm2} in Section~\ref{section:di}) would prove an upper bound of the shape \eqref{eq:altbehrend}. Indeed, if this dichotomy could be established with $O(1)$ constants rather than $\tilde{O}_\alpha(1)$ constants then \eqref{eq:altbehrend} would follow with $c=1/2$, which would be the best possible bound. (Of course, this dichotomy is vacuously satisfied if better bounds for Roth's theorem were known so that (1) always holds. The point is that establishing the density increment dichotomy that either (1) or (2) must hold would be sufficient to prove \eqref{eq:altbehrend}.)

Case (2) belongs to the regime of what we have called small increments, which many steps of our argument are already capable of delivering. In particular, spectral boosting does produce density increments of exactly this strength, provided the various error parameters going into it are small enough. The methods of this paper would, then, be strong enough to prove the above dichotomy (and hence prove an upper bound of the strength of \eqref{eq:altbehrend}), except for two significant quantitative weaknesses:

\begin{enumerate}
\item The bounds in the structural result Theorem~\ref{th:structure} are quantitatively too weak. For example, in the non-relative situation, we obtain bounds that are polynomial in $\tau^{-1/\log k}$. This would need to be improved to bounds that are polynomial in $\tau^{-1/k}$. 
\item In proving that our set $\Delta$ of Proposition~\ref{prop:big} is additively non-smoothing, we used relatively crude estimates to argue that if the higher additive energies of $\Delta$ were too large, then we could obtain a large density increment. A large density increment is unacceptable when trying to obtain bounds of the strength \eqref{eq:altbehrend}, and so an alternative method of proving that $\Delta$ is additively non-smoothing (or some other route to providing the kind of structure that our structural theorem for additively non-smoothing sets provides) would have to be found.
\end{enumerate}

If the first obstacle were overcome but not the second this would have little detectable influence on our upper bound (it would merely improve the value of the constant $c>0$, from something like $2^{-2^{2^{1000}}}$ to $2^{-1000}$). If better ideas were found for showing that subsets of spectra are additively non-smoothing, addressing the second obstacle, then this would have a much greater effect on our bounds, resulting in something like 
\[r(N) \ll \frac{1}{(\log N)^{\omega(N)}}\]
for some function $\omega(N)\to \infty$, even without any improvement in Theorem~\ref{th:structure}. As mentioned above, we believe that if both deficiencies were addressed suitably then bounds of the strength \eqref{eq:altbehrend} would follow.

We finish with two conjectures, which attempt to highlight the gaps in our knowledge of the additive structure of spectra. These conjectures are stronger forms of the kinds of results we have used in this paper, and are consistent with all the constructions of spectra that we are aware of.

These conjectures are only approximate, and we offer no precise implication between either of them and improving the upper bound for $r(N)$. Nonetheless, we believe that any significant progress towards either of the two conjectures below should, when combined with the structural and spectral boosting methods of this paper, yield improvements to the upper bound for $r(N)$. In these conjectures $G$ can be interpreted as any finite abelian group (for example, either $\bbz/N\bbz$ or $\bbf_2^n$). 

The first conjecture states that any spectrum satisfies a version of the conclusion of the structural result for additively non-smoothing sets (which is strictly weaker than actually being additively non-smoothing). 

\begin{conjecture}\label{conj1}
Let $A\subset G$ be a set of density $\alpha$ and $1\geq \eta \gg \alpha$. Let $\Delta=\Delta_\eta(A)$. There exist $X,H\subset \Delta$ such that
\[\abs{X}\abs{H}\gs_\alpha \eta^2\abs{\Delta}^2\textrm{ and }E(X,H)\gs_\alpha \abs{X}\abs{H}^2.\]
\end{conjecture}

The methods of this paper can show that, when $\abs{\Delta}\gg \eta^{-2}\alpha^{-1}$, either  Conjecture~\ref{conj1} holds, or else $A$ has a large density increment, but we believe that both the size assumption and the density increment option is unnecessary. 

The second conjecture, which is strictly weaker than the first, says that any spectrum contains a reasonably sized subset of very small dimension.
\begin{conjecture}\label{conj2}
Let $A\subset G$ be a set of density $\alpha$ and let $\eta>0$. Then for any $1\geq \epsilon \gg \eta^2$ there exists a set $\Delta'\subset \Delta_\eta(A)$ such that
\[\abs{\Delta'}\gg \epsilon \abs{\Delta_\eta(A)}\quad\textrm{and}\quad\dim(\Delta')\ls_\alpha \epsilon\eta^{-2}.\]
\end{conjecture}
In this conjecture dimension should be interpreted as the size of the smallest $\Lambda\subset \Delta'$ such that
\[\Delta'\subset \left\{ \sum_{\lambda\in\Lambda}\epsilon_\lambda \lambda : \epsilon_\lambda\in\{-1,0,1\}\right\}.\]
For comparison, Chang's lemma from \cite{Ch:2002} implies that this is true when $\epsilon\gg 1$. More generally,  Corollary~\ref{cor:massdimboot} of this paper implies this conjecture in the range $1\geq \epsilon\gg \eta$. A proof of this conjecture in the final range $\eta \gg \epsilon\gg \eta^2$ remains elusive. Lemma~\ref{lemma:dimsymmetry} coupled with a simple pigeonholing argument implies that Conjecture~\ref{conj2} follows from Conjecture~\ref{conj1}. There may, however, be a more direct alternative route that yields the second conjecture but not the first.

\subsection*{A concrete challenge}
For the benefit of those eager readers who have jumped to the final page in search of interesting conjectures, we restate these conjectures and the relevant definitions in the model case when $G=\bbf_2^n$. Any proof or disproof of either of these conjectures, even in this model setting, would be very interesting.

When $A\subset \bbf_2^n$ then for any $\eta\in(0,1]$ we define the $\eta$-large spectrum of $A$ by 
\[\Delta_\eta(A) = \left\{ x\in \bbf_2^n:  \abs{\sum_{a\in A}(-1)^{a\cdot x}}\geq \eta\abs{A}\right\}.\]
The additive energy $E(X,H)$ counts the number of solutions to $x_1+h_1=x_2+h_2$ with $x_1,x_2\in X$ and $h_1,h_2\in H$, and the dimension of a set in $\bbf_2^n$ is the size of the largest linearly independent subset. 

\begin{conjecture}
Let $A\subset \bbf_2^n$ be a set of density $\alpha=\abs{A}/2^n$ and $\eta\in[\alpha,1]$. There are $X,H\subset \Delta_\eta(A)$ such that
\[\abs{X}\abs{H}\gs_\alpha \eta^{2}\abs{\Delta}^2\textrm{ and }E(X,H)\gs_\alpha \abs{X}\abs{H}^2.\]
\end{conjecture}	

\begin{conjecture}
Let $A\subset \bbf_2^n$ be a set of density $\alpha=\abs{A}/2^n$ and $\eta\in[\alpha,1]$. There exists a set $\Delta'\subset \Delta_\eta(A)$ such that
\[\abs{\Delta'}\geq \eta^2\abs{\Delta_\eta(A)}\quad\textrm{and}\quad\dim(\Delta')\ls_\alpha 1.\]
\end{conjecture}
\bibliography{Roth}
\bibliographystyle{acm}

\end{document}